\documentclass[11pt,a4paper,english,reqno,a4paper]{amsart}
\usepackage{amsmath,amssymb,amsthm, comment,graphicx}
\usepackage{tikz}
\usepackage{fancyhdr}
\usepackage{epsfig}
\usepackage{mathrsfs}
\newtheorem{theorem}{Theorem}[section]
\newtheorem{definition}{Definition}[section]
\newtheorem{lemma}{Lemma}[section]
\newtheorem{remark}{Remark}[section]
\newtheorem{proposition}{Proposition}[section]

\numberwithin{equation}{section}

\newcommand{\R}{{\mathbb R}}

\begin{document}
\title[Hypersonic Similarity for the Two Dimensional Steady Potential Flow]
{Hypersonic Similarity for the Two Dimensional Steady Potential Flow with Large Data}
\author{Jie Kuang }
\address{ Wuhan Institute of Physics and Mathematics,
Chinese Academy of Sciences, Wuhan 430071,
P.R.China}
\email{\tt jkuang@wipm.ac.cn}

\author{Wei Xiang}
\address{Department of Mathematics
City University of Hong Kong
Kowloon, Hong Kong, People¡¯s Republic of China}
\email{\tt  weixiang@cityu.edu.hk}

\author{Yongqian Zhang}
\address{School of Mathematical Sciences, Fudan University, Shanghai 200433, China}
\email{\tt  yongqianz@fudan.edu.cn}


\date{\today}

\begin{abstract}
In this paper, we establish the first rigorous mathematical global result on the validation of the hypersonic similarity,
which is also called the Mach-number independence principle, for the two dimensional steady potential flow.
The hypersonic similarity is equivalent to the Van Dyke's similarity theory, that if the hypersonic similarity parameter $K$ is fixed, the shock solution structures (after scaling) are consistent, when the Mach number of the flow is sufficiently large. One of the difficulty is that after scaling, the solutions are usually of large data since the perturbation of the hypersonic flow is usually not small related to the sonic speed.
In order to make it, we first employ the modified Glimm scheme to construct the approximate solutions with large data and find fine structure of the elementary wave curves to obtain the global existence of entropy solutions with large data, for fixed $K$ and sufficiently large Mach number of the incoming flow $M_{\infty}$. 
Finally, we further show that for a fixed hypersonic similarity parameter $K$, if the Mach number $M_{\infty}\rightarrow\infty$, the solutions obtained above approach to the solution of the corresponding initial-boundary value problem of the hypersonic small-disturbance equations. Therefore, the Van Dyke's similarity theory is first verified rigorously.
\end{abstract}

\maketitle

\section{Introduction and Main result}\setcounter{equation}{0}
The flow is called hypersonic when the Mach number of the flow is bigger than five.
Since 1940s, there are many studies on the hypersonic flow (see \cite{tsien} for example) due to many applications in areodynamics and engineering. The main difficulty on the study of the hypersonic flow is that the density is relatively very small compared to the speed, so like the fluids behaviour near the vacuum, all the characteristics are close to each other and the shock layer is thin. On the other hand, there is one important feature of the hypersonic flow, which is called the hypersonic similarity. This property is of great significance on both the theoretical and experimental research of the thin shock layer for the hypersonic flow (see \cite{anderson} for more details).

\par Let $\theta$ be the wedge angle and let $M_{\infty}$ be the Mach number of the incoming flow (see Fig.\ref{fig1.1}). Define the similarity parameter (see (127.3) in Landau-Lifschitz \cite[Page 482]{Landau} for more details),
\begin{equation}
K=M_{\infty}\theta.
\end{equation}
Physically, the hypersonic similarity means that for a fixed similarity parameter $K$, the flow structures are similar under scaling if the Mach number $M_{\infty}$ is sufficiently large. Actually, after scaling, the flows with the same similarity parameter $K$ are governed approximately by the same equation, which is called the hypersonic small-disturbance equations and was first developed by Tsien \cite{tsien} for the two-dimensional steady irrotational flow and the three-dimensional axially symmetric steady flow. Recently, Qu-Yuan-Zhao \cite{QYZ} studied a different problem, the hypersonic limit, in which there is no hypersonic similarity structures since the wedge angle $\theta$ is fixed such that the similarity parameter $K$ changes for all $M_{\infty}$ and tends to the infinity for the hypersonic limit $M_{\infty}\rightarrow\infty$.

\vspace{5pt}
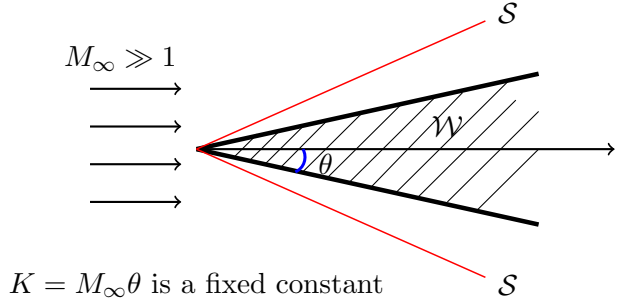
\begin{figure}[ht]
	\begin{center}
		\begin{tikzpicture}[scale=1]
		\draw [thick][->](-2,1)--(3.5,1);
		\draw [line width=0.06cm] (-2,1) --(2.5,2);
		\draw [line width=0.06cm] (-2,1) --(2.5,0);
		\draw [line width=0.02cm][red] (-2,1) --(1.8,2.7);
		\draw [line width=0.02cm][red] (-2,1) --(1.8,-0.7);
		\draw [thick][->](-3.4,1.8)--(-2.2,1.8);
		\draw [thick][->](-3.4,1.3)--(-2.2,1.3);
		\draw [thick][->](-3.4,0.8)--(-2.2,0.8);
		\draw [thick][->](-3.4,0.3)--(-2.2,0.3);
		\draw [thin](-1.5,0.9)--(-1.2,1.2);
		\draw [thin](-1.3,0.85)--(-0.9,1.25);
		\draw [thin](-1.1,0.80)--(-0.6,1.3);
		\draw [thin](-0.9,0.75)--(-0.3,1.36);
		\draw [thin](-0.6,0.7)--(0.2,1.5);
		\draw [thin](-0.3,0.6)--(0.7,1.6);
		\draw [thin](0,0.55)--(1.15,1.70);
		\draw [thin](0.3,0.5)--(1.5,1.75);
		\draw [thin](0.6,0.40)--(2.1,1.9);
		\draw [thin](0.9,0.35)--(2.2,1.65);
		\draw [thin](1.3,0.3)--(2.5,1.5);
		\draw [thin](1.7,0.20)--(2.5,1);
		\draw [line width=0.04cm][blue] (-0.6,1)to [out=-60, in=20](-0.7,0.7);
		\node at (2.1,2.8) {$\mathcal{S}$};
		\node at (2.1,-0.8) {$\mathcal{S}$};
		\node at (-3.0,2.2) {$M_{\infty}\gg1$};
		\node at (-0.3,0.8) {$\theta$};
		\node at (-2.0, -0.8) {$K=M_{\infty}\theta $ is a fixed constant};
		\node at (1.3, 1.3) {$\mathcal{W}$};
		\end{tikzpicture}
	\end{center}
	\caption{Hypersonic flow past over a slender wedge}\label{fig1.1}
\end{figure}

\par The hypersonic small-disturbance equations and the hypersonic similarity are derived as follows.
Suppose the hypersonic flow is governed by 
\begin{eqnarray}\label{eq:1.1}
\left\{
\begin{array}{llll}
\partial_{x}(\rho u)+\partial_{y}(\rho v)=0,\\[5pt]
\partial_{x}v-\partial_{y}u=0,
\end{array}
     \right.
\end{eqnarray}
where the density $\rho$ and the velocity $(u,v)$ satisfy the following Bernoulli's law:
\begin{eqnarray}\label{eq:1.2}
\frac{1}{2}(u^2+v^2)+\frac{\rho^{\gamma-1}}{\gamma-1}=B_{\infty}:=\frac{1}{2}U_{\infty}^2+\frac{\rho_{\infty}^{\gamma-1}}{\gamma-1}.
\end{eqnarray}

\par For the problem of the hypersonic flow onto a solid slender-body with boundary $y=\pm\tau b_{0}x$, without loss of the generality, let us only consider the lower half space domain, \emph{i.e.}, in the region that $x\geq0$ and $y\leq\tau b_0x$ with a fixed constant $b_0<0$ in Fig. \ref{fig1.1}. The incoming flows are given by
\begin{eqnarray}\label{eq:1.3}
(\rho, u, v)\big|_{x= 0, y\leq0}=\big(\rho_{0}, u_{0}, v_{0}\big)(y).
\end{eqnarray}
Along the boundary, the flow satisfies the impermeable slip boundary condition,
\emph{i.e.},
\begin{eqnarray}\label{eq:1.4}
(u, v)\cdot(\tau b_{0},-1)=0.
\end{eqnarray}
Let $U_{\infty}$ be a sufficiently large number. Let
$$
a_{\infty}:=\tau M_{\infty}=\tau U_{\infty}\rho_{\infty}^{\frac{1-\gamma}{2}}.
$$
Obviously, if $K$ is fixed, then $a_{\infty}$ is fixed too. So $a_{\infty}$
is also called the hypersonic similarity parameter (see Chapter 4 in \cite{anderson}).
As done in \cite{anderson,dyke}, we define the following scaling:
\begin{equation}\label{eq:1.5}
x=\bar{x},\quad y=\tau\bar{y}, \quad u=U_{\infty}(1+\tau^2\bar{u}) ,\quad v=U_{\infty}\tau \bar{v},\quad \rho=\rho_{\infty}\bar{\rho},
\end{equation}
and substitute \eqref{eq:1.5} into equations \eqref{eq:1.1} and \eqref{eq:1.2} to obtain
\begin{equation}\label{eq:1.6}
\begin{cases}
\partial_{\bar{x}}\big(\bar{\rho} (1+\tau^2\bar{u})\big)+\partial_{\bar{y}}(\bar{\rho} \bar{v})=0,\\[5pt]
\partial_{\bar{x}}\bar{v}-\partial_{\bar{y}}\bar{u}=0,\\[5pt]
\bar{u}+\frac{1}{2}(\bar{v}^2+\tau^2\bar{u}^2)+\frac{\bar{\rho}^{\gamma-1}-1}{(\gamma-1)a_{\infty}^2}=0.
\end{cases}
\end{equation}

\vspace{5pt}
\begin{figure}[ht]
\begin{center}
\begin{tikzpicture}[scale=1.15]
\draw [line width=0.03cm][->] (-2,0) --(3.5,0);
\draw [line width=0.03cm][->] (-1,-2.7) --(-1,1);
\draw [line width=0.07cm](-1,0) --(3,-1);
\draw [line width=0.07cm](-1,0) --(-1,-2.7);
\draw [line width=0.02cm][red] (-1,0) --(2.8,-1.8);
\draw [line width=0.04cm][blue] (0.3,0)to [out=-60, in=20](0.2,-0.3);
\draw [thick][->](-2.5,-1.2)to[out=30, in=-160](-1.1,-1.2);
\draw [thick][->](-2.5,-1.6)to[out=30, in=-160](-1.1,-1.6);
\draw [thick][->](-2.5,-2.0)to[out=30, in=-160](-1.1,-2.0);
\node at (3.4,-0.2) {$\bar{x}$};
\node at (-0.8,1) {$\bar{y}$};
\node at (-1.2,-0.2) {$O$};
\node at (3.3,-1) {$\Gamma$};
\node at (3.1,-1.9) {$\bar{\mathcal{S}}$};
\node at (-0.9,-2.9) {$\mathcal{I}$};
\node at (1.4,-0.25) {$\bar{\theta}=\arctan b_{0}$};
\node at (1.0,-1.5) {$\Omega$};
\node at (-1.9,-0.8) {$(\bar{\rho}_{0}, \bar{u}_{0},\bar{ v}_{0})$};
\end{tikzpicture}
\end{center}
\caption{Hypersonic similarity law }\label{fig1.2}
\end{figure}
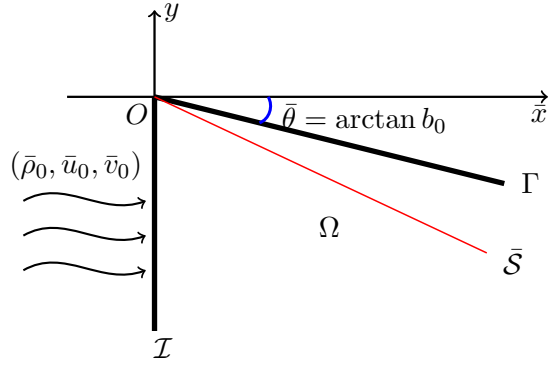

The solid boundary is now given by $\bar{y}=b_{0}\bar{x}$.
Then, the corresponding fluid domain and its boundary are given by (see Fig.\ref{fig1.2})
\begin{eqnarray*}
\Omega=\{(\bar{x}, \bar{y}): \bar{x}>0,\ \bar{y}<b_{0}\bar{x} \},
\quad \Gamma=\{(\bar{x}, \bar{y}): \bar{x}>0,\ \bar{y}=b_{0}\bar{x} \}.
\end{eqnarray*}
The unit normal of $\Gamma$ is
$\mathbf{n}=\mathbf{n}(\bar{x},b_{0}\bar{x})=\frac{(b_{0},-1)}{\sqrt{1+b^{2}_{0}}}$.
Initial condition \eqref{eq:1.3} becomes
\begin{eqnarray}\label{eq:1.7}
(\bar{\rho}, \bar{u}, \bar{v})\big|_{\mathcal{I}}=\big(\bar{\rho}_{0}, \bar{u}_{0},\bar{ v}_{0}\big)(\bar{y}),
\qquad \mathcal{I}=\{\bar{x}=0,\ \bar{y}\leq0\}.
\end{eqnarray}
Along $\Gamma$, condition \eqref{eq:1.4} now becomes
\begin{eqnarray}\label{eq:1.8}
\big((1+\tau^{2}\bar{u}), \bar{v}\big)\cdot \mathbf{n}\big|_{\Gamma}=0.
\end{eqnarray}

Physically, the hypersonic similarity is, for a fixed similarity parameter $a_{\infty}$, the structure of solutions of \eqref{eq:1.6}-\eqref{eq:1.8} is persistent if $M_{\infty}$ large (or $\tau$ is small). Mathematically, the structure of solutions of \eqref{eq:1.6}-\eqref{eq:1.8} should be investigated by the simpler equation via neglecting the terms involving $\tau^{2}$, that is the hypersonic small-disturbance equations
\begin{equation}\label{eq:1.9}
\begin{cases}
\partial_{\bar{x}}\bar{\rho}+\partial_{\bar{y}}(\bar{\rho} \bar{v})=0, \\[5pt]
\partial_{\bar{x}}\bar{v}-\partial_{\bar{y}}\bar{u}=0, \\[5pt]
\bar{u}+\frac{1}{2}\bar{v}^2+\frac{\bar{\rho}^{\gamma-1}-1}{(\gamma-1)a_{\infty}^2}=0,
\end{cases}
\end{equation}
with initial data \eqref{eq:1.7} and boundary condition that
\begin{eqnarray}\label{eq:1.10}
\bar{v}\big|_{\Gamma}=b_{0}.
\end{eqnarray}

It is also called the Van Dyke's similarity theory.
So if the Van Dyke's similarity theory can be justified rigorously, then the study of the two-dimensional steady hypersonic flow can be much simplified by studying of the hypersonic small-disturbance equaitons \eqref{eq:1.9}, because we do not face the difficulty that the characteristics are so close. On the other hand, since for the hypersonic flow, the perturbation of the velocity $(\bar{u}, \bar{v})$ is usually not small related to the sonic speed, so the solutions of \eqref{eq:1.6} and \eqref{eq:1.9} are usually with large data in the physical applications.

\par In this paper, we are going to show the Van Dyke's similarity theory rigorously. First, since the flow concerned moves
along the wedge from left to right, \emph{i.e.}, $1+\tau^2\bar{u}>0$, then from the third equation of \eqref{eq:1.6}, we have
\begin{eqnarray}\label{eq:1.11}
\begin{split}
\bar{u}(\bar{\rho}, \bar{v}; \tau^{2})=\frac{1}{\tau^{2}}\Big(\sqrt{1-t\tau^{2}}-1\Big),
\end{split}
\end{eqnarray}
where
\begin{eqnarray}\label{eq:1.12}
\begin{split}
t=\frac{2\big(\bar{\rho}^{\gamma-1}-1\big)}{(\gamma-1)a^{2}_{\infty}}+\bar{v}^{2}.
\end{split}
\end{eqnarray}

\par Then, substituting \eqref{eq:1.11} into the first two equations of \eqref{eq:1.6}, we get
\begin{equation}\label{eq:1.13}
\begin{cases}
\partial_{\bar{x}}\big(\bar{\rho} (1+\tau^2\bar{u})\big)+\partial_{\bar{y}}(\bar{\rho} \bar{v})=0, &\quad\mbox{in}\quad \Omega, \\[5pt]
\partial_{\bar{x}}\bar{v}-\partial_{\bar{y}}\bar{u}=0, &\quad\mbox{in}\quad \Omega.
\end{cases}
\end{equation}
Similarly, substituting the third equation in \eqref{eq:1.9} into the second equation in \eqref{eq:1.9}, we have
\begin{equation}\label{eq:1.14}
\begin{cases}
\partial_{\bar{x}}\bar{\rho}+\partial_{\bar{y}}(\bar{\rho} \bar{v})=0,&\quad\mbox{in}\quad \Omega, \\[5pt]
\partial_{\bar{x}}\bar{v}+\partial_{\bar{y}}\big(\frac{1}{2}\bar{v}^2+\frac{\bar{\rho}^{\gamma-1}-1}{(\gamma-1)a_{\infty}^2}\big)=0,
&\quad\mbox{in}\quad \Omega,
\end{cases}
\end{equation}
where $(\bar{\rho},\bar{v})$ satisfies the initial condition \eqref{eq:1.17} and the boundary condition \eqref{eq:1.10}.

To unify equations \eqref{eq:1.13} and \eqref{eq:1.14}, we rewrite $(\bar{\rho}, \bar{v})$ as $(\bar{\rho}^{(\tau)}, \bar{v}^{(\tau)})$, where \eqref{eq:1.14} corresponds to the case that $\tau=0$. Let $U^{(\tau)}=(\bar{\rho}^{(\tau)}, \bar{v}^{(\tau)})$ and
\begin{eqnarray}\label{eq:1.15}
W(U^{(\tau)}, \tau^{2})=\Big(\bar{\rho}^{(\tau)}\big(1+\tau^{2}\bar{u}^{(\tau)}\big), \bar{v}^{(\tau)}\Big),\quad
F(U^{(\tau)}, \tau^{2})=\Big(\bar{\rho}^{(\tau)} \bar{v}^{(\tau)}, -\bar{u}^{(\tau)}\Big).
\end{eqnarray}

\par Then, equations \eqref{eq:1.13} and \eqref{eq:1.14} can be rewritten as
\begin{eqnarray}\label{eq:1.16}
\partial_{\bar{x}}W(U^{(\tau)}, \tau^{2})+\partial_{\bar{y}}F(U^{(\tau)}, \tau^{2})=0,
\end{eqnarray}
with the initial condition
\begin{eqnarray}\label{eq:1.17}
U^{(\tau)}\big|_{\mathcal{I}}=U_{0}(y),
\end{eqnarray}
and the boundary condition
\begin{eqnarray}\label{eq:1.18}
\Big((1+\tau^{2}\bar{u}^{(\tau)}),\bar{v}^{(\tau)}\Big)\cdot(-b_{0},1)\Big|_{\Gamma}=0.
\end{eqnarray}

Now, we will introduce the definition of the entropy solutions of problem \eqref{eq:1.16}--\eqref{eq:1.18}.
\begin{definition}[Entropy solutions]\label{def:1.1}
A weak solution $U^{(\tau)}\in \big(BV_{loc}(\Omega)\cap L^{1}_{loc}(\Omega)\big)^2$ of the initial-boundary value problem \eqref{eq:1.16}--\eqref{eq:1.18} in $\Omega\subset \R^2_+$ is called an entropy solution, if for any convex entropy pair $(\mathcal{E},\mathcal{Q})$, that is, $\nabla \mathcal{Q}(W^{(\tau)},\tau)=\nabla \mathcal{E}(W^{(\tau)}, \tau^{2})\nabla F(U(W^{(\tau)}),\tau^{2})$ and $\nabla^2\mathcal{E}(W^{(\tau)},\tau^{2})\ge 0$, the entropy inequality holds: For any $\phi \in C_0^{\infty}(\R^2)$ with $\phi\ge 0$,
\begin{eqnarray}\label{eq:1.19}
\begin{split}
&\iint_{\Omega}\Big(\mathcal{E}(W^{(\tau)},\tau^{2})\partial_{\bar{x}}\phi+ \mathcal{Q}(W^{(\tau)}, \tau^{2})\partial_{\bar{y}}\phi\Big)dxdy
+ \int^{0}_{-\infty}\mathcal{E}(W^{(\tau)}_{0}, \tau^{2})\phi(0,y)dy\\[5pt]
&\qquad+\int_{\Gamma}(\mathcal{E}(W^{(\tau)},\tau^{2}),\mathcal{Q}(W^{(\tau)},\tau^{2}))\cdot\mathbf{n}ds\geq 0,
\end{split}
\end{eqnarray}
where $W^{(\tau)}_{0}=W(U_{0},\tau^{2})$ and $\mathbf{n}$ is the unit inner normal on boundary $\Gamma$.
\end{definition}

The main result in this paper is stated as follows.

\begin{theorem}[Main theorem]\label{thm:1.1}
Suppose that $\rho_{*}$ and $\rho^{*}$ are two constant states with $0<\rho_{*}<\rho^{*}<\infty$ and
$\bar{\rho}_{0}$ be the given initial density satisfying that $\bar{\rho}_{0}\in [\rho_{*},\rho^{*}]$.
There exist constants $C$, $\gamma_{0}\in (1, 2)$ and $\varepsilon_{0}>0$ such that for any $\gamma\in [1,\gamma_{0}]$
and $\tau\in (0,\varepsilon_{0})$, if
\begin{eqnarray}\label{eq:1.20}
(\gamma-1+\tau^{2})\Big(T.V.\big\{(\bar{\rho}_{0}, \bar{v}_{0}):\ (-\infty, 0]\big\}+\|b_{0}\|_{L^{\infty}}\Big)\leq C,
\end{eqnarray}
then initial-boundary value problem \eqref{eq:1.16}--\eqref{eq:1.18} admits a global entropy solutions
$(\bar{\rho}^{(\tau)}, \bar{v}^{(\tau)})$ with bounded total variations defined for all $\bar{x}>0$, \emph{i.e.},
\begin{eqnarray}\label{eq:1.21}
\sup_{\bar{x}>0}T.V.\big\{(\bar{\rho}^{(\tau)}, \bar{v}^{(\tau)})(\bar{x},\cdot); (-\infty, b_{0}\bar{x}]\big\}
+\sup_{\bar{x}>0}\|(\bar{\rho}^{(\tau)}, \bar{v}^{(\tau)})(\bar{x},\cdot)\|_{L^{\infty}((-\infty, b_{0}\bar{x}])}\leq \tilde{C},
\end{eqnarray}
where the constant $\tilde{C}>0$ is independent of $\gamma-1$ and $\tau$.
Moreover, as $\tau\rightarrow0$,
\begin{eqnarray}\label{eq:1.22}
(\bar{\rho}^{(\tau)}, \bar{v}^{(\tau)})\rightarrow (\bar{\rho}^{(0)}, \bar{v}^{(0)})=(\bar{\rho}, \bar{v}), \qquad\mbox{in}\quad
 L^{1}(\Omega\cap B_{\bar{R}}(0)),
\end{eqnarray}
for any $\bar{R}>0$, where $B_{\bar{R}}(O)=\big\{(\bar{x},\bar{y}): \bar{x}^{2}+\bar{y}^{2}\leq \bar{R}\big\}$ and $(\bar{\rho}, \bar{v})$ is the entropy solution of the initial-boundary value problem \eqref{eq:1.16}--\eqref{eq:1.18} with $\tau=0$, which satisfies that
\begin{eqnarray}\label{eq:1.23}
\sup_{\bar{x}>0}T.V.\big\{(\bar{\rho}, \bar{v})(\bar{x},\cdot); (-\infty, b_{0}\bar{x}]\big\}
+\sup_{\bar{x}>0}\|(\bar{\rho}, \bar{v})(\bar{x},\cdot)\|_{L^{\infty}((-\infty, b_{0}\bar{x}])}<\infty.
\end{eqnarray}
\end{theorem}

\begin{remark}\label{rem:1.1}
When $\tau=0$, the convex entropy pair $(\mathcal{E}(W^{(\tau)},\tau^{2}), \mathcal{Q}(W^{(\tau)},\tau^{2}))$ can be taken of the form
\begin{eqnarray}\label{eq:1.24}
\mathcal{E}(W^{(0)},0)=\frac{\rho v^{2}}{2}+\frac{\rho^{\gamma}}{a_{\infty }\gamma (\gamma-1)},\quad \ \
\mathcal{Q}(W^{(0)},0)=v\mathcal{E}(W^{(0)},0).
\end{eqnarray}
So the entropy solution $(\bar{\rho},\bar{v})$ of problem \eqref{eq:1.16}--\eqref{eq:1.18} with $\tau=0$
satisfies the entropy inequality
\begin{eqnarray}\label{eq:1.25}
\partial_{\bar{x}}\mathcal{E}(W^{(0)},0)+\partial_{\bar{y}}\mathcal{Q}(W^{(0)},0)\leq 0,
\end{eqnarray}
in the distribution sense.
\end{remark}

\begin{remark}\label{rem:1.2}
Once the solution $(\bar{\rho}^{(\tau)},\bar{v}^{(\tau)})$ of problem \eqref{eq:1.16}--\eqref{eq:1.18} is obtained,
it is easy to obtain the solutions $(\bar{\rho}^{(\tau)},\bar{u}^{(\tau)},\bar{v}^{(\tau)})$ of problem
\eqref{eq:1.6}--\eqref{eq:1.8} by solving $\bar{u}^{(\tau)}$ directly from equation \eqref{eq:1.11}.
Therefore, in this paper, we are devoted to showing Theorem \ref{thm:1.1}.
\end{remark}

\par In this paper, we will give the first rigorous mathematical proof on the Van Dyke's similarity theory.
More precisely, we will prove that solution $U^{(\tau)}$ of the initial boundary value problem \eqref{eq:1.16}--\eqref{eq:1.18} with large data
has a limit $U$ as $\tau\rightarrow 0$, where $U$ is a solution of the initial boundary value problem \eqref{eq:1.14}, \eqref{eq:1.17} and \eqref{eq:1.10}, \emph{i.e.}, problem \eqref{eq:1.16}--\eqref{eq:1.18} with $\tau=0$. To achieve this, we first establish the global
existence of entropy solutions of the initial boundary value problem \eqref{eq:1.16}--\eqref{eq:1.18}
for fixed $\tau$ with large data.

The main difficulty is that we can't apply the results in \cite{ns1,ns2, temple} directly, because equations \eqref{eq:1.16} is different from the ones that considered in \cite{ns1,ns2, temple}. Moreover, the boundary condition \eqref{eq:1.18}
is Neumann type which is also different from the one studied in \cite{ns2}, which is the Dirichlet boundary type.
As far as we know, there is no result on the steady supersonic Euler flow with large data. In order to deal with it, we first need to study fine structures of the elementary wave curves carefully and then derive the local wave interaction estimates. Fortunately, we find the fine structures to allow us to establish the wave interaction estimates as well as the estimates of the elementary waves reflection on the boundary. Based on them, we can choose weights $\mathscr{K}_{b}$ and $C_{*}$ (see \eqref{eq:4.3g} below) to construct a modified Glimm's type functional and then shows it monotonicity decreasing. Now, we can follow the standard arguments to show the global existence of entropy solutions of the initial-boundary value problem \eqref{eq:1.16}--\eqref{eq:1.18} with uniformly bound in the BV norm independent of $\tau$ provided that \eqref{eq:1.20} holds. Finally, by the uniformly bounds, we can further extact a subsequence to show that its limt as $\tau\rightarrow0$ is actually a entropy solution of problem \eqref{eq:1.16}--\eqref{eq:1.18} with $\tau=0$. It justifies the Van Dyke's similarity theory rigorously.

\par There are many literatures on the global existence of the entropy solutions of small data in the BV space for the one dimensional hyperbolic conservation laws since J. Glimm's original paper \cite{glimm} in 1960s.
There are also many literatures on the BV solutions of the two dimensional steady supersonic Euler flow with small data
(see \cite{cky2,ckz, cl, czz, kyz, xzz, zh1, zh2}).
However, there are few results on the global existence of weak solutions with large data due to the nonlinearity of the system. As far as we know, only systems with special structures can be dealt with. One of the most important example is the one dimensional isothermal gas dynamic system. The global existence of the entropy solutions of this system with large data has been proved by Nishida \cite{nishida} in 1968. Then Nishida-Smoller extended the existence result to the isentropic case with the assumption that $\gamma$ is sufficiently close to $1$ in \cite{ns1,ns2}.
Later on, the existence result was extended to the non-isentropic case by Liu in \cite{liu}. Recently, Askura-Corli \cite{asakura, ac} proved these results by using the wave-front tracking method and see also \cite{bressan, ccz2, cxz} for the related results.

\par The rest of this paper is organized as follows. In Section 2, we study some basic structure for system \eqref{eq:1.16} near $\tau=0$, including the Riemann invariants, the fine properties of the elementary wave curves, as well as the the solutions for the Riemann problem including the boundary. As a byproducts, we also give some basic structure for system \eqref{eq:1.14},
(\emph{i.e.} $\tau=0$) involving the Riemann invariants, the fine properties of the elementary wave curves, as well as the solutions for the Riemann problem including the boundary. Section 3 is devoted to the analysis of the local wave interaction estimates of various type. In Section 4, we construct the approximate solutions by the modified Glimm scheme, introduce the modified Glimm-type functional by choosing some weights, and then show that it is a decreasing functional, which leads to the global existence of the entropy solutions to the initial-boundary value problem \eqref{eq:1.16}--\eqref{eq:1.18} with large data by a standard procedure. Finally, we show that as $\tau\rightarrow0$, it approaches to the solutions of the initial-boundary value problem \eqref{eq:1.14}, \eqref{eq:1.17} and \eqref{eq:1.10}. 

\emph{Finally, we remark that in what follows, for the notational simplicity, we will denote $U^{(\tau)}=(\bar{\rho}^{(\tau)},\bar{v}^{(\tau)})$ and $(\bar{x}, \bar{y})$ as $U=(\rho, v)$ and $(x,y)$, respectively.}

\section{Riemann Problem of the initial-boundary value problem \eqref{eq:1.16}-\eqref{eq:1.18}}\setcounter{equation}{0}
In this section, we will study the basic structure of system \eqref{eq:1.16} and then consider the corresponding Riemann solutions.

\subsection{Riemann Invariants and the Shock Curves of equations \eqref{eq:1.16}}
In this subsection, we study some basic structures of the Riemann solutions of system \eqref{eq:1.16} of large data. 
By direct computation, the eigenvalues of system \eqref{eq:1.16} are
\begin{eqnarray}\label{eq:2.2}
\begin{split}
&\lambda_{-}(U, \tau^{2})=\frac{v\sqrt{1-t\tau^{2}}-a^{-1}_{\infty}\rho^{\frac{\gamma-1}{2}}
\sqrt{1-(\gamma-1)^{-1}(\gamma+1)a^{-2}_{\infty}\rho^{\gamma-1}\tau^{2}}}{1-(t+a^{-2}_{\infty}\rho^{\gamma-1})\tau^{2}},\\[5pt]
&\lambda_{+}(U, \tau^{2})=\frac{v\sqrt{1-t\tau^{2}}+a^{-1}_{\infty}\rho^{\frac{\gamma-1}{2}}
\sqrt{1-(\gamma-1)^{-1}(\gamma+1)a^{-2}_{\infty}\rho^{\gamma-1}\tau^{2}}}{1-(t+a^{-2}_{\infty}\rho^{\gamma-1})\tau^{2}},
\end{split}
\end{eqnarray}
and the corresponding right eigenvectors are
\begin{eqnarray}\label{eq:2.3}
\begin{split}
&r_{-}(U, \tau^{2})=\Big(-a^{2}_{\infty}\rho^{-\frac{\gamma-1}{2}}\big(\lambda_{-}(U,\tau^{2})+\partial_{v}u(\rho, v, \tau^{2})\big), \ a^{2}_{\infty}\rho^{-\frac{\gamma-1}{2}}\partial_{\rho}u(\rho, v, \tau^{2})\Big),\\[5pt]
&r_{+}(U, \tau^{2})=\Big(-a^{2}_{\infty}\rho^{-\frac{\gamma-1}{2}}\big(\lambda_{+}(U,\tau^{2})+\partial_{v}u(\rho, v, \tau^{2})\big), \ a^{2}_{\infty}\rho^{-\frac{\gamma-1}{2}}\partial_{\rho}u(\rho, v, \tau^{2})\Big).
\end{split}
\end{eqnarray}

\par For $u(\rho, v, \tau^{2})$, we have the following lemma.
\begin{lemma}\label{lem:2.1}
If $\gamma>1$, then we have 
\begin{eqnarray}\label{eq:2.4}
\begin{split}
\partial_{\rho}u(\rho, v, \tau^{2})=-\frac{\rho^{\gamma-2}}{a^{2}_{\infty}\sqrt{1-t \tau^{2}}},\ \ \ \
\partial_{v}u(\rho, v, \tau^{2})=-\frac{v}{\sqrt{1-t \tau^{2}}},
\end{split}
\end{eqnarray}
and
\begin{eqnarray}\label{eq:2.5}
\begin{split}
&\partial^{2}_{\rho\rho}u(\rho, v, \tau^{2})
=-\frac{(\gamma-2)\rho^{\gamma-3}\Big(1-\big(t-(\gamma-2)^{-1}a^{-2}_{\infty}\rho^{\gamma-1}\big)\tau^{2}\Big)}
{a^{2}_{\infty}(1-\tau^{2}t)^{3/2}},\\[5pt]
&\partial^{2}_{\rho v}u(\rho, v, \tau^{2})=-\frac{\rho^{\gamma-2}v \tau^{2}}{a^{2}_{\infty}(1-\tau^{2}t)^{3/2}},\ \ \
\partial^{2}_{v v}u(\rho, v, \tau^{2})=-\frac{1-2(\gamma-1)^{-1}a^{-2}_{\infty}(\rho^{\gamma-1}-1)\tau^{2}}{(1-\tau^{2}t)^{3/2}},
\end{split}
\end{eqnarray}
where $t$ is defined by \eqref{eq:1.12}.
\end{lemma}

\begin{proof}
First, by \eqref{eq:1.12}, we have
\begin{eqnarray*}
\begin{split}
\frac{\partial t}{\partial \rho}=\frac{2\rho^{\gamma-2}}{a^{2}_{\infty}},\ \  \ \
\frac{\partial t}{\partial v}=2v.
\end{split}
\end{eqnarray*}
From $u$, we also get that
\begin{eqnarray*}
\begin{split}
1+\tau^{2}u(\rho, v, \tau^{2})=\sqrt{1-t\tau^{2}}.
\end{split}
\end{eqnarray*}
So it follows that
\begin{eqnarray*}
\begin{split}
\tau^{2}\partial_{\rho}u(\rho, v, \tau^{2})=-\frac{1}{2}(1-t\tau^{2})^{-1/2}\tau^{2}\frac{\partial t}{\partial \rho},
\ \ \ \tau^{2}\partial_{v}u(\rho, v, \tau^{2})=-\frac{1}{2}(1-t\tau^{2})^{-1/2}\tau^{2}\frac{\partial t}{\partial v},
\end{split}
\end{eqnarray*}
which gives $\partial_{\rho}u(\rho, v, \tau^{2})$ and $\partial_{v}u(\rho, v, \tau^{2})$, respectively.
With 
\eqref{eq:2.4},
we can further take derivatives with respect to $\rho$, $v$ to derive \eqref{eq:2.5}. This completes the proof of the lemma.
\end{proof}

\begin{remark}\label{rem:2.1}
By Lemma \ref{lem:2.1} and \eqref{eq:2.2}, we have that
\begin{eqnarray}\label{eq:2.6}
\begin{split}
&\lambda_{\pm}(U, \tau^{2})+\partial_{v}u(\rho, v, \tau^{2})\\[5pt]
=&\frac{a^{-1}_{\infty}\rho^{\frac{\gamma-1}{2}}\bigg(a^{-1}_{\infty}\rho^{\frac{\gamma-1}{2}}v\tau^{2}
\pm\sqrt{\Big(1-(\gamma-1)^{-1}(\gamma+1)a^{-2}_{\infty}\rho^{\gamma-1}\tau^{2}\Big)(1-t\tau^{2})}\bigg)}
{\Big(1-\big(t+a^{-2}_{\infty}\rho^{\gamma-1}\big)\tau^{2}\Big)\sqrt{1-t\tau^{2}}}.
\end{split}
\end{eqnarray}

\end{remark}

\begin{lemma}\label{lem:2.2}
For the eigenvalues $\lambda_{+}$ and $\lambda_{-}$, we have
\begin{eqnarray}\label{eq:2.7}
\begin{split}
&\lambda_{\pm}(U, 0)=v\pm\frac{\rho^{\frac{\gamma-1}{2}}}{a_{\infty}},\ \ \
\lambda_{\pm}(U, 0)+\partial_{v}u(\rho,v,0)=\pm\frac{\rho^{\frac{\gamma-1}{2}}}{a_{\infty}},
\end{split}
\end{eqnarray}
and
\begin{eqnarray}\label{eq:2.8}
\begin{split}
&r_{\pm}(U,0)=(a_{\infty}, \pm \rho^{\frac{\gamma-3}{2}}).
\end{split}
\end{eqnarray}
Moreover,
\begin{eqnarray}\label{eq:2.9}
\partial_{\rho}\lambda_{\pm}(U,0)= \pm\frac{(\gamma-1)\rho^{\frac{\gamma-3}{2}}}{2a_{\infty}},\quad \ \
\partial_{v}\lambda_{\pm}(U,0)=1.
\end{eqnarray}
\end{lemma}

\begin{proof}
Firstly, by the definition of $t$, \eqref{eq:2.7} and \eqref{eq:2.8} follow directly from \eqref{eq:2.2} and \eqref{eq:2.3}. For $\partial_{\rho}\lambda_{\pm}(U,0)$, note that the characteristic equation of system \eqref{eq:1.16} is
\begin{eqnarray}\label{eq:2.1}
\big(1+\tau^{2}(u+\rho \partial_{\rho}u)\big)\lambda^{2}+\big((1+\tau^{2}u)\partial_{v}u-v\big)\lambda
+\rho \partial_{\rho}u-v\partial_{v}u=0.
\end{eqnarray}
Taking derivative on \eqref{eq:2.1} with respect to $\rho$
to obtain that
\begin{align*}
&\big[2\big(1+(u+\rho\partial_{\rho}u)\tau^{2}\big)\lambda+(1+\tau^{2}u)\partial_{v}u-v\big]\partial_{\rho}\lambda
+\tau^{2}\lambda^{2}\partial_{\rho}\big(u+\rho\partial_{\rho}u\big)\\[5pt]
&\qquad\qquad\ \ \  + \big[\tau^{2}\partial_{\rho}u\partial_{v}u+(1+\tau^{2}u)\partial^{2}_{\rho v}u\big]\lambda
+\rho\partial^{2}_{\rho\rho}u-v\partial^{2}_{\rho v}u+\partial_{\rho}u=0.
\end{align*}
So take $\tau=0$, we have
\begin{align*}
\partial_{\rho}\lambda(U,0)=-\frac{\rho\partial_{\rho\rho}u(\rho,v,0)+\partial_{\rho}u(\rho,v,0)}{2\lambda(U,0)+\partial_{v}u(\rho,v,0)-v},
\end{align*}
which gives the expression of $\partial_{\rho}\lambda_{\pm}(U,0)$ with the help of Lemma \ref{lem:2.1} and \eqref{eq:2.7}.

In the same way, we can also take derivatives on \eqref{eq:2.1} with respect to $v$ to have
\begin{align*}
\partial_{v}\lambda(U,0)=\frac{\big(1-\partial^{2}_{vv}u(\rho,v,0)\big)\lambda(U,0)+v\partial^{2}_{vv}u(\rho,v,0)+\partial_{v}u(\rho,v,0)}
{2\lambda(U,0)+\partial_{v}u(\rho,v,0)-v},
\end{align*}
which implies the expression of $\partial_{v}\lambda_{\pm}(U,0)$ by employing Lemma \ref{lem:2.1} and \eqref{eq:2.7} again.
\end{proof}

Let
\begin{equation}
\omega(U,\tau^{2})=\big(\omega_{-}, \omega_{+}\big)(U,\tau^{2})\quad \mbox{(or $\omega=(\omega_{-}, \omega_{+})$)}
\end{equation}
be the Riemann invariants satisfying
\begin{eqnarray*}
\begin{split}
\nabla_{U}\omega_{\pm}(U,\tau^{2})\cdot r_{\pm}(U,\tau^{2})=0.
\end{split}
\end{eqnarray*}
Without loss of the generality, we can assume $\omega_{\pm}(U,\tau^{2})$ is defined by solving the following two equations
\begin{eqnarray}\label{eq:2.10}
\begin{split}
&\partial_{\rho}\omega_{\pm}(U,\tau^{2}):=-a^{2}_{\infty}\rho^{-\frac{\gamma-1}{2}}\partial_{\rho}u(\rho,v,\tau^{2})
=\frac{\rho^{\frac{\gamma-3}{2}}}{\sqrt{1-t\tau^{2}}},
\end{split}
\end{eqnarray}
and
\begin{eqnarray}\label{eq:2.11}
\begin{split}
\partial_{v}\omega_{\pm}(U,\tau^{2})
&:=-a^{2}_{\infty}\rho^{-\frac{\gamma-1}{2}}\big(\lambda_{\pm}(U,\tau^{2})+\partial_{v}u(\rho, v, \tau^{2})\big)\\[5pt]
&=-\frac{\rho^{\frac{\gamma-1}{2}}v\tau^{2}
\pm a_{\infty}\sqrt{\Big(1-(\gamma-1)^{-1}(\gamma+1)a^{-2}_{\infty}\rho^{\gamma-1}\tau^{2}\Big)(1-t\tau^{2})}}
{\Big(1-\big(t+a^{-2}_{\infty}\rho^{\gamma-1}\big)\tau^{2}\Big)\sqrt{1-t\tau^{2}}}.
\end{split}
\end{eqnarray}

\begin{remark}\label{rem:2.2}
For $\tau=0$, $\omega_{\pm}(U,0)$ can be expressed explicitly as
\begin{eqnarray}\label{eq:2.12x}
\begin{split}
&r:=\omega_{-}(U, 0)=a_{\infty}v+\frac{2(\rho^{\frac{\gamma-1}{2}}-1)}{\gamma-1},\ \ \
s:=\omega_{+}(U, 0)=-a_{\infty}v+\frac{2(\rho^{\frac{\gamma-1}{2}}-1)}{\gamma-1}.
\end{split}
\end{eqnarray}
\end{remark}

\begin{lemma}\label{lem:2.3}
For $\rho>0$, there exists a constant $\epsilon_{1}>0$ sufficiently small such that for any $\tau\in(0,\epsilon_{1})$,
$U=(\rho,v)$ can be represented as a function of $\omega$. Moreover, the map $U=(\rho, v)\mapsto \omega=(\omega_{-}(U,\tau^{2}), \omega_{+}(U,\tau^{2}))$ is bijective for any fixed parameter $\rho>0$ and sufficiently small parameter $\tau^2$. Moreover
\begin{eqnarray}\label{eq:2.13}
\begin{split}
\nabla_{\omega_{-}}U\big|_{\tau=0}
=\Big(\frac{1}{2}\rho^{-\frac{\gamma-3}{2}}, \frac{1}{2a_{\infty}}\Big), \ \ \
\nabla_{\omega_{+}}U\big|_{\tau=0}
=\Big(\frac{1}{2}\rho^{-\frac{\gamma-3}{2}}, -\frac{1}{2a_{\infty}}\Big).
\end{split}
\end{eqnarray}
\end{lemma}

\begin{proof}
By Lemma \ref{lem:2.1} and Remark \ref{rem:2.1},
\begin{eqnarray*}
\begin{split}
\det\Big(\nabla_{U}\omega_{-}(U, \tau^{2}), \ \nabla_{U}\omega_{+}(U, \tau^{2})\Big)\Big|_{\tau=0}
&=\frac{a^{2}_{\infty}}{\rho^{\frac{\gamma-1}{2}}}
\begin{vmatrix}
-\partial_{\rho}u(\rho,v, 0)& -\lambda_{-}(U,0)-\partial_{v}u(\rho,v, 0)\\[10pt]
-\partial_{\rho}u(\rho,v, 0)&-\lambda_{+}(U,0)-\partial_{v}u(\rho,v, 0)
\end{vmatrix}
\\[10pt]
&=\frac{2\rho^{\gamma-2}}{a_{\infty}}>0.
\end{split}
\end{eqnarray*}
So it follows from the implicit function theorem that there exists a constant $\epsilon_{1}>0$ sufficiently small such that
for any $\tau\in(0,\epsilon_{1})$, $U$ can be solved as a function of $\omega$.

Next, we are going to prove \eqref{eq:2.13}. Taking derivatives as follows and let $\tau=0$
\begin{eqnarray*}
\left\{
\begin{array}{llll}
\partial_{\rho}\omega_{-}(U,0)\frac{\partial \rho}{\partial\omega_{-}}\big|_{\tau=0}
+\partial_{v}\omega_{-}(U,0)\frac{\partial v}{\partial\omega_{-}}\big|_{\tau=0}=1,\\[5pt]
\partial_{\rho}\omega_{+}(U,0)\frac{\partial \rho}{\partial\omega_{-}}\big|_{\tau=0}
+\partial_{v}\omega_{+}(U,0)\frac{\partial v}{\partial\omega_{-}}\big|_{\tau=0}=0,
\end{array}
\right.
\end{eqnarray*}
which gives the expression of $\nabla_{U}\omega_{-}\big|_{\tau=0}$ in \eqref{eq:2.13} by \eqref{eq:2.10} and \eqref{eq:2.11}.
In the same way, one can also get the expression of $\nabla_{U}\omega_{+}\big|_{\tau=0}$. We omit the argument for the shortness.
\end{proof}

Now, we are going to study the elementary wave curves to system
\eqref{eq:2.2} globally. Based on Lemma \ref{lem:2.3}, we will use $\omega_{-},\ \omega_{+}$ as  the variables in the phase plane for the convenience.

The elementary wave curves consist of the rarefaction wave curve and the shock wave curve. First, for the rarefaction wave curve, one of the Riemann invariants corresponding to $\lambda_{+}(U,\tau^{2})$ or $\lambda_{-}(U,\tau^{2})$ is a constant. We denoted the rarefaction wave by $\mathcal{R}_{1}$ (or $\mathcal{R}_{2}$) corresponding to $\lambda_{+}(U,\tau^{2})$ (or $\lambda_{-}(U,\tau^{2})$). So, in the phase plane, the rarefaction wave curves $\mathcal{R}_{1}$ and $\mathcal{R}_{2}$ which pass through $\omega_{0}=(\omega_{-,0}, \omega_{+,0})=(\omega_{-}, \omega_{+})(U_{0},\tau^{2})$ are
 \begin{eqnarray}\label{eq:2.14}
 \begin{split}
 &\mathcal{R}_{1}: \  \omega_{+}=\omega_{+,0},\ \  \omega_{-}>\omega_{-,0}  \qquad
 \mathcal{R}_{2}: \  \omega_{-}=\omega_{-,0},\ \  \omega_{+}<\omega_{+,0}.
 \end{split}
 \end{eqnarray}

Next, let us consider the shock wave curves for system \eqref{eq:3.2}.
The shock solutions are the Riemann solutions satisfying the following Rankine-Hugoniot conditions on the shock with shock speed $\sigma(\tau^{2})$:
\begin{equation}\label{eq:2.15}
\sigma(\tau^{2})[W(U,\tau^{2})]=[F(U,\tau^{2})],
\end{equation}
where the bracket $[\cdot]$ stands for the difference of the value of the quality concerned on across the discontinuity. In addition, across the shock, the following Lax geometry entropy conditions hold:
\begin{equation}\label{eq:2.16}
\lambda_{-}(U,\tau^{2})<\sigma_{-}(\tau^{2})<\lambda_{-}(U_{0}, \tau^{2}),\ \
or\ \   \lambda_{+}(U,\tau^{2})<\sigma_{+}(\tau^{2})<\lambda_{+}(U_{0},\tau^{2}),
\end{equation}
where $\sigma_{-}(\tau^{2})$ and $\sigma_{+}(\tau^{2})$ are the shock speeds corresponding to $\lambda_{-}(U,\tau^{2})$ and $\lambda_{+}(U,\tau^{2})$, respectively. Actually, entropy condition \eqref{eq:2.16} implies that
\begin{equation}\label{eq:2.17x}
\rho>\rho_{0},\ \  v<v_{0}, \ \  or \  \ \rho<\rho_{0},\ \  v<v_{0}.
\end{equation}
Eliminating $\sigma(\tau^{2})$ from the R-H condition \eqref{eq:2.15} yields
\begin{eqnarray}\label{eq:2.18}
\begin{split}
(\rho v-\rho_{0}v_{0})(v-v_{0})&=\Big(\rho-\rho_{0}+\tau^{2}\big(\rho u(\rho,v,\tau^{2})-\rho_{0} u(\rho_{0},v_{0}, \tau^{2})\big)\Big)
\big(u(\rho_{0}, v_{0},\tau^{2})-u(\rho,v, \tau^{2})\big).
\end{split}
\end{eqnarray}

Let $\alpha=\rho/\rho_{0}$ with $\rho_{0}>0$ and define
\begin{eqnarray}\label{eq:2.19}
\begin{split}
\mathscr{F}(\alpha, v, U_{0}; \tau^{2})&=(\alpha v-v_{0})(v-v_{0})-\Big(\alpha-1+\tau^{2}\big(\alpha u(\rho_0\alpha,v,\tau^{2})-u(\rho_{0},v_{0},\tau^{2})\big)\Big)\\[5pt]
&\quad \times \big(u(\rho_{0},v_{0},\tau^{2})-u(\rho_0\alpha,v,\tau^{2})\big).
\end{split}
\end{eqnarray}
Then equation \eqref{eq:2.18} is equivalent to equation $\mathscr{F}(\alpha, v, U_{0}; \tau^{2})=0$.
First, we will study some properties for $\mathscr{F}$ when $\tau=0$.
\begin{lemma}\label{lem:2.4}
For $\mathscr{F}$ defined by \eqref{eq:2.19} and for $\gamma>1$, equation $\mathscr{F}(\alpha, v, U_{0}; 0)=0$ admits a unique solution $v$ satisfying that
\begin{eqnarray}\label{eq:2.20}
\begin{split}
v=v_{0}-\sqrt{\frac{2\rho^{\gamma-1}_{0}(\alpha-1)(\alpha^{\gamma-1}-1)}{(\gamma-1)a^{2}_{\infty}(\alpha+1)}}.
\end{split}
\end{eqnarray}
Moreover, we have
\begin{eqnarray}\label{eq:2.21}
\begin{split}
&\frac{\partial \mathscr{F}}{\partial \alpha}\Big|_{\tau=0}=-\frac{\rho^{\gamma-1}_{0}\Big(2(\alpha^{\gamma-1}-1)+(\gamma-1)\alpha^{\gamma-2}(\alpha^{2}-1)\Big)}
{(\gamma-1)a^{2}_{\infty}(\alpha+1)},\\[5pt]
&\frac{\partial \mathscr{F}}{\partial v}\Big|_{\tau=0}=-\sqrt{\frac{2\rho^{\gamma-1}_{0}(\alpha^{2}-1)(\alpha^{\gamma-1}-1)}{(\gamma-1)a^{2}_{\infty}}},\\[5pt]
&\frac{\partial^{2}\mathscr{F}}{\partial\alpha^{2}}\Big|_{\tau=0}
=-a^{-2}_{\infty}\rho^{\gamma-1}_{0}\alpha^{\gamma-3}(\gamma \alpha+2-\gamma),\\[5pt]
&\frac{\partial^{2}\mathscr{F}}{\partial\alpha\partial v}\Big|_{\tau=0}
=v-v_{0},\quad \  \frac{\partial^{2}\mathscr{F}}{\partial v^{2}}\Big|_{\tau=0}=\alpha-1.
\end{split}
\end{eqnarray}
\end{lemma}

\begin{proof}
\eqref{eq:2.20} can be obtained by the direct computation together with the entropy condition \eqref{eq:2.17x}.
For \eqref{eq:2.17}, first for $\frac{\partial\mathscr{F}}{\partial \alpha}$, by the direct computation
\begin{align*}
\frac{\partial\mathscr{F}}{\partial \alpha}&=v(v-v_{0})+u(\rho,v,\tau^{2})-u(\rho_{0},v_{0},\tau^{2})
+\rho_{0}(\alpha-1)\partial_{\rho}u(\rho,v,\tau^{2})\\[5pt]
&\quad  +\Big(u(\rho,v,\tau^{2})+\rho_{0}\alpha\partial_{\rho}u(\rho,v,\tau^{2})\Big)
\Big(u(\rho,v,\tau^{2})-u(\rho_{0},v_{0},\tau^{2})\Big)\tau^{2}\\[5pt]
&\quad +\rho_{0}\Big(\alpha u(\rho,v,\tau^{2})-u(\rho_{0},v_{0},\tau^{2})\Big)\partial_{\rho}u(\rho,v,\tau^{2})\tau^{2}.
\end{align*}
So it follows from Lemma \ref{lem:2.1} that
\begin{align*}
\frac{\partial\mathscr{F}}{\partial \alpha}\Big|_{\tau=0}&=v(v-v_{0})+u(\rho,v,\tau^{2})-u(\rho_{0},v_{0},\tau^{2})
+\rho_{0}(\alpha-1)\partial_{\rho}u(\rho,v,\tau^{2})\\[5pt]
&=v(v-v_{0})-\frac{1}{2}\bigg(v^{2}+\frac{2(\rho^{\gamma-1}-1)}{(\gamma-1)a^{2}_{\infty}}
-v^{2}_{0}-\frac{2(\rho^{\gamma-1}_{0}-1)}{(\gamma-1)a^{2}_{\infty}}\bigg)-a^{-2}_{\infty}\rho^{\gamma-1}_{0}(\alpha-1)\alpha^{\gamma-2}.
\end{align*}
Thus the expression of $\frac{\partial\mathscr{F}}{\partial \alpha}\Big|_{\tau=0}$ in \eqref{eq:2.21} follows with the help of \eqref{eq:2.20}.

\par Next, taking derivative on $\mathscr{F}$ with respect to $v$ 
\begin{align*}
\frac{\partial\mathscr{F}}{\partial v}&=2\alpha v-(\alpha+1)v_{0}+(\alpha-1)\partial_{v}u(\rho,v,\tau^{2})\\[5pt]
&\quad \ \ +\Big(2\alpha u(\rho,v,\tau^{2})-(\alpha+1)u(\rho_{0},v_{0},\tau^{2})\Big)\partial_{v}u(\rho,v,\tau^{2})\tau^{2}.
\end{align*}
So 
\begin{align*}
\frac{\partial\mathscr{F}}{\partial v}\Big|_{\tau=0}&=2\alpha v-(\alpha+1)v_{0}+(\alpha-1)\partial_{v}u(\rho,v,0)
=(\alpha+1)(v-v_{0}).
\end{align*}
Hence, the expression of $\frac{\partial\mathscr{F}}{\partial v}\Big|_{\tau=0}$ in \eqref{eq:2.21} follows by \eqref{eq:2.20} again.

In the same way as done for deriving the expression of $\frac{\partial\mathscr{F}}{\partial \rho}\Big|_{\tau=0}$ and
$\frac{\partial\mathscr{F}}{\partial v}\Big|_{\tau=0}$, we can further take derivatives on $\frac{\partial\mathscr{F}}{\partial \rho}$
and $\frac{\partial\mathscr{F}}{\partial v}$ with respect to $\alpha$ and $v$ and let $\tau=0$, then \eqref{eq:2.21} follows from Lemma \ref{lem:2.1}.
\end{proof}

\begin{remark} \label{rem:2.3}
When $\tau=0$, it follows from Remark \ref{rem:2.2}, entropy condition \eqref{eq:2.17x}, and the straightforward calculation that
\begin{equation}\label{eq:2.58}
S_1:\,
\begin{cases}
s_0-s=\sqrt{\frac{2}{\gamma-1}}\rho_0^{\frac{\gamma-1}{2}}\left\{-\sqrt{\frac{(1-\alpha)(1-\alpha^{\gamma-1})}
	{\alpha+1}}+\sqrt{\frac{2}{\gamma-1}}(1-\alpha^{\frac{\gamma-1}{2}})\right\}\\[5pt]
r_0-r=\sqrt{\frac{2}{\gamma-1}}\rho_0^{\frac{\gamma-1}{2}}\left\{\sqrt{\frac{(1-\alpha)(1-\alpha^{\gamma-1})}
	{\alpha+1}}+\sqrt{\frac{2}{\gamma-1}}(1-\alpha^{\frac{\gamma-1}{2}})\right\}
\end{cases}
\quad 0<\alpha\leq1,
\end{equation}
and
\begin{equation}\label{eq:2.59}
S_2:\,
\begin{cases}
s_0-s=\sqrt{\frac{2}{\gamma-1}}\rho_0^{\frac{\gamma-1}{2}}\left\{-\sqrt{\frac{(\alpha-1)(\alpha^{\gamma-1}-1)}
	{\alpha+1}}-\sqrt{\frac{2}{\gamma-1}}(\alpha^{\frac{\gamma-1}{2}}-1)\right\}\\[5pt]
r_0-r=\sqrt{\frac{2}{\gamma-1}}\rho_0^{\frac{\gamma-1}{2}}\left\{\sqrt{\frac{(\alpha-1)(\alpha^{\gamma-1}-1)}
	{\alpha+1}}-\sqrt{\frac{2}{\gamma-1}}(\alpha^{\frac{\gamma-1}{2}}-1)\right\}
\end{cases}
\quad \alpha\geq1.
\end{equation}
\end{remark}

\begin{remark}\label{rem:2.4}
When $\gamma=1$ and $\tau=0$, $S_{1}$ and $S_{2}$ are of the following forms:
	\begin{equation}\label{eq:2.62x}
	S_1:\,
	\begin{cases}
	s_0-s=-\sqrt{-\frac{2(1-\alpha)}{\alpha+1}\ln \alpha}-\ln \alpha, \\[5pt]
	r_0-r=\sqrt{-\frac{2(1-\alpha)}{\alpha+1}\ln \alpha}-\ln \alpha,
	\end{cases}
	\quad 0<\alpha\leq1,
	\end{equation}
	and
	\begin{equation}\label{eq:2.63x}
	S_2:\,
	\begin{cases}
	s_0-s=-\sqrt{\frac{2(\alpha-1)}{\alpha+1}\ln \alpha}-\ln \alpha, \\[5pt]
	r_0-r=\sqrt{\frac{2(\alpha-1)}{\alpha+1}\ln \alpha}-\ln \alpha,
	\end{cases}
	\quad \alpha\geq1.
	\end{equation}
	
Eliminating $\alpha$, one has
	\begin{equation}\label{eq:2.64x}
	r_{0}-r-(s_{0}-s)=2\sqrt{\frac{1-e^{-\frac{1}{2}(r_{0}-r+s_{0}-s)}}{1+e^{-\frac{1}{2}(r_{0}-r+s_{0}-s)}}\Big(r_{0}-r+s_{0}-s\Big)},
	\end{equation}
where $r_{0}-r+s_{0}-s\geq 0$ for the $S_{1}$ wave, and $r_{0}-r+s_{0}-s\leq 0$ for the $S_{2}$ wave.
\end{remark}

Now, we will give the existence and properties of the shock wave curves near $\tau=0$.
\begin{lemma}\label{lem:2.5}
There exists a small constant $0<\epsilon_{2}<\epsilon_{1}$ such that for any $\tau\in(0,\epsilon_{2})$,
$v$ can be solved as a function of $\alpha, U_{0}, \tau^{2}$ from equation $\mathscr{F}(\alpha, v, U_{0}; \tau^{2})=0$, \emph{i.e.}, $v=\varphi(\alpha, U_{0}, \tau^{2})$. Moreover,
\begin{eqnarray}\label{eq:2.22}
\begin{split}
&\frac{\partial \varphi}{\partial \alpha}\Big|_{\tau=0}=-\sqrt{\frac{\rho^{\gamma-1}_{0}}{2(\gamma-1)a^{2}_{\infty}}}
\frac{2(\alpha^{\gamma-1}-1)+(\gamma-1)\alpha^{\gamma-2}(\alpha^{2}-1)}{\sqrt{(\alpha-1)(\alpha^{\gamma-1}-1)(\alpha+1)^{3}}},
\end{split}
\end{eqnarray}
and
\begin{eqnarray}\label{eq:2.23}
\begin{split}
\frac{\partial^{2} \varphi}{\partial \alpha^{2}}\Big|_{\tau=0}&=
\frac{1}{4}\sqrt{\frac{2\rho^{\gamma-1}_{0}}{(\gamma-1)a^{2}_{\infty}}}
\Big(\sqrt{(\alpha^{2}-1)(\alpha^{\gamma-1}-1)}(\alpha^{2}-1)(\alpha^{\gamma-1}-1)(\alpha+1)\Big)^{-1}\\[5pt]
&\ \ \ \times\Big(2(\gamma-1)\alpha^{\gamma-3}\big((2-\gamma)\alpha^{2}-2\alpha+\gamma-2\big)
(\alpha^{2}-1)(\alpha^{\gamma-1}-1)\\[5pt]
&\qquad \qquad \qquad \qquad   +4(2\alpha-1)(\alpha^{\gamma-1}-1)^{2}+(\gamma-1)^{2}\alpha^{2(\gamma-2)}(\alpha^{2}-1)^{2}\Big).
\end{split}
\end{eqnarray}
\end{lemma}

\begin{proof}
When $\alpha=1$, it is easy to see that $\rho=\rho_{0}$ and $v=v_{0}$.
Now, we only consider the case that $\alpha\neq1$. Let
\begin{align*}
G(\alpha, v, U_{0}; \tau^{2})=\frac{\mathscr{F}(\alpha, v, U_{0}; \tau^{2})}{\alpha-1}.
\end{align*}

By $\eqref{eq:2.21}_{2}$, 
\begin{align*}
\frac{\partial G}{\partial v}\Big|_{\tau=0}=\frac{(\alpha+1)(v-v_{0})}{\alpha-1}
=-\sqrt{\frac{2\rho^{\gamma-1}_{0}}{(\gamma-1)a^{2}_{\infty}}}\frac{\sqrt{(\alpha^{2}-1)(\alpha^{\gamma-1}-1)}}{\alpha-1}.
\end{align*}

Then, we know $\frac{\partial G}{\partial v}\Big|_{\tau=0}>0$ for $0<\alpha<1$, $\frac{\partial G}{\partial v}\Big|_{\tau=0}<0$
for $\alpha>1$, and
$$\lim_{\alpha\rightarrow\pm1}\frac{\partial G}{\partial v}\Big|_{\tau=0}=\mp\frac{\sqrt{2\rho^{\gamma-1}_{0}}}{a_{\infty}}\neq 0.$$

Therefore, by Lemma \ref{lem:2.4} and the implicit function theorem, there exists a small constant
$0<\epsilon_{2}<\epsilon_{1}$ such that for any $\tau\in(0,\epsilon_{2})$, equation $G(\alpha, v, U_{0}; \tau^{2})=0$
admits a unique solution $v=\varphi(\alpha, U_{0}, \tau^{2})$. It implies that
$\mathscr{F}(\alpha,\varphi(\alpha, U_{0}, \tau^{2}) , U_{0}; \tau^{2})=0$. 

\par Next, let us compute $\frac{\partial \varphi}{\partial \alpha}\Big|_{\tau=0}$. Notice that
$\mathscr{F}(\alpha,\varphi(\alpha, U_{0}, \tau^{2}) , U_{0}; \tau^{2})=0$. 
Taking derivative on it with respect to $\alpha$ yields that 
\begin{eqnarray}\label{eq:2.24}
\begin{split}
\frac{\partial\mathscr{F}(\alpha,v; \tau^{2})}{\partial \alpha}+\frac{\partial\mathscr{F}(\alpha,v; \tau^{2})}{\partial v}\frac{\partial \varphi}{\partial \alpha}=0.
\end{split}
\end{eqnarray}
Let $\tau=0$, then we can obtain \eqref{eq:2.22}, by Lemma \ref{lem:2.3}.

\par Finally, taking derivatives with respect to $\alpha$ again on \eqref{eq:2.24} yields that
\begin{eqnarray}\label{eq:2.25}
\begin{split}
\frac{\partial^{2}\varphi}{\partial \alpha^{2}}=-\frac{\partial^{2}_{\alpha\alpha}\mathscr{F}(\alpha,v; \tau^{2})
+2\partial^{2}_{\alpha v}\mathscr{F}(\alpha,v; \tau^{2})\partial_{\alpha}\varphi+\partial^{2}_{vv}\mathscr{F}(\alpha,v; \tau^{2})(\partial_{\alpha}\varphi)^{2}}{\partial_{v}\mathscr{F}(\alpha,v; \tau^{2})}.
\end{split}
\end{eqnarray}
So, by Lemma \ref{lem:2.3} and \eqref{eq:2.21}, we have \eqref{eq:2.23}. This completes the proof.
\end{proof}

\par Next, we are going to study the shock wave curves in the Riemann invariants coordinates.
First, 
we have the following properties for $\omega_{\pm}$.
\begin{lemma}\label{lem:2.6}
For $\gamma\in [1,2]$, there exists a small constant $0<\epsilon_{3}<\epsilon_{2}$ such that for any $\tau\in(0,\epsilon_{3})$,
along the shock wave curve $v=\varphi(\alpha, U_{0}; \tau^{2})$,
\begin{eqnarray}\label{eq:2.26}
\frac{\partial (\omega_{-,0}-\omega_{-})}{\partial \alpha}<0, \quad \ for\quad 0<\alpha<1,
\end{eqnarray}
and
\begin{eqnarray}\label{eq:2.27}
\frac{\partial (\omega_{+,0}-\omega_{+})}{\partial \alpha}<0, \quad \ for\quad \alpha>1,
\end{eqnarray}
where $\omega_{-}$ and $\omega_{+}$ are defined by \eqref{eq:2.10} and \eqref{eq:2.11}, and $\omega_{\pm,0}=\omega_{\pm}(U_{0},\tau^{2})$.
\end{lemma}

\begin{proof}
We only prove \eqref{eq:2.26} here since we can treat $\omega_{+}$ in the same way.
By the definition of $\omega_{-}$, along the shock wave curve, 
\begin{eqnarray*}
\begin{split}
\frac{\partial(\omega_{-,0}-\omega_{-})}{\partial \alpha}
&=-\Big(\rho_{0}\frac{\partial \omega_{-}}{\partial \rho}+\frac{\partial \omega_{-}}{\partial v}\frac{\partial \varphi}{\partial \alpha}\Big)\\[5pt]
&=a^{2}_{\infty}\rho^{-\frac{\gamma-1}{2}}\Big(\rho_{0}\partial_{\rho}u(\rho, v,\tau^{2})+\big(\lambda_{-}(U,\tau^{2})
+\partial_{v}u(\rho, v,\tau^{2})\big)\frac{\partial \varphi}{\partial \alpha}\Big).
\end{split}
\end{eqnarray*}
So, by Lemma \ref{lem:2.1}, Lemma \ref{lem:2.2} and Lemma \ref{lem:2.5}, we obtain that
\begin{eqnarray*}
\begin{split}
\frac{\partial(\omega_{-,0}-\omega_{-})}{\partial\alpha}\Big|_{\tau=0}&=
-\frac{\rho^{\frac{\gamma-1}{2}}_{0}}{2\sqrt{2(\gamma-1)(1-\alpha)(1-\alpha^{\gamma-1})(1+\alpha)^{3}}}\\[5pt]
&\ \ \ \ \ \ \  \times \Big(2(1-\alpha^{\gamma-1})+(\gamma-1)(1-\alpha^{2})\alpha^{\gamma-2}\\[5pt]
&\ \ \ \ \ \ \ \ \ \ \ \ \  +\sqrt{2(\gamma-1)(1-\alpha)(1-\alpha^{\gamma-1})(1+\alpha)^{3}\alpha^{\gamma-3}}\Big)\\[5pt]
&<0,
\end{split}
\end{eqnarray*}
for $0<\alpha<1$. It completes the proof of the lemma.
\end{proof}

Denote
\begin{eqnarray}\label{eq:2.28}
\beta_{-}=\omega_{-,0}-\omega_{-},\ \ \  \beta_{+}=\omega_{+,0}-\omega_{+}.
\end{eqnarray}
By Lemma \ref{lem:2.6} and the implicit function theorem, $\alpha$ can be regarded as a function of
$\beta_{-}$ or $\beta_{+}$, \emph{i.e.}, $\alpha=\alpha_{1}(\beta_{-}, U_{0}; \tau^{2})$ and $\alpha=\alpha_{2}(\beta_{+}, U_{0}; \tau^{2})$.
So along the shock wave curves, 
\begin{eqnarray}\label{eq:2.29}
\beta_{+}=\Phi_{1}(\beta_{-}, U_{0}; \tau^{2}):=\omega_{+,0}-\omega_{+}(\alpha_{1}(\beta_{-}, U_{0}; \tau^{2}),\tau^{2}),
\end{eqnarray}
which is called the $\mathcal{S}_{1}$ shock curve, or
\begin{eqnarray}\label{eq:2.30}
\beta_{-}=\Phi_{2}(\beta_{+}, U_{0}; \tau^{2}):=\omega_{-,0}-\omega_{-}(\alpha_{2}(\beta_{+}, U_{0}; \tau^{2}),\tau^{2}),
\end{eqnarray}
which is called the $\mathcal{S}_{2}$ shock curve.

For the $\mathcal{S}_{1}$ shock wave curve, we have the following lemma.
\begin{lemma}\label{lem:2.7}
For $\gamma\in[1,2]$ and $0<\alpha<1$, there exists a constant $\epsilon_{4}>0$ sufficiently small such that
for $\tau \in (0, \epsilon_{4})$, the shock curve $\mathcal{S}_{1}$ starting at $(\omega_{-,0}, \omega_{+,0})$ is
\begin{eqnarray}\label{eq:2.31}
\beta_{+}=\Phi_{1}(\beta_{-}, U_{0}; \tau^{2})
=\int^{\beta_{-}}_{0}\Psi_{1}(\alpha, U_{0}; \tau^{2})\Big|_{\alpha=\alpha_{1}(\beta, U_{0};\tau^{2})}d\beta,
\end{eqnarray}
where $\beta_{-}=\omega_{-,0}-\omega_{-}>0$. Moreover, 
\begin{eqnarray}\label{eq:2.32}
0<\frac{\partial\Phi_{1}(\beta_{-}, U_{0};0)}{\partial \beta_{-}}<1, \qquad
\frac{\partial^{2}\Phi_{1}(\beta_{-}, U_{0}; 0)}{\partial \beta^{2}_{-}}>0.
\end{eqnarray}
Finally, if $\alpha>\varepsilon_{0}>0$, then 
\begin{eqnarray}\label{eq:2.33}
C_{1}\varepsilon_{0}<\frac{\partial\Phi_{1}(\beta_{-}, U_{0};\tau^{2})}{\partial \beta_{-}}<1,
\end{eqnarray}
where $C_{1}>0$ is a constant depending only on the data and $\varepsilon_{0}$, and independent of $\tau$.
\end{lemma}

\begin{proof}
By \eqref{eq:2.29} and Lemma \ref{lem:2.6}, we can define
\begin{eqnarray*}
\Psi_{1}(\alpha, U_{0}; \tau^{2}):=\frac{\partial \Phi_{1}(\beta_{-}, U_{0}; \tau^{2})}{\partial \beta_{-}}.
\end{eqnarray*}
So \eqref{eq:2.31} follows.
Moreover,
\begin{eqnarray*}
\begin{split}
\frac{\partial \Phi_{1}(\beta_{-}, U_{0}; \tau^{2})}{\partial \beta_{-}}
=\frac{\frac{\partial (\omega_{+,0}-\omega_{+})}{\partial \alpha}}{\frac{\partial (\omega_{-,0}-\omega_{-})}{\partial \alpha}}
=\frac{\rho_{0}\partial_{\rho}u(\rho,v,\tau^{2})+\big(\lambda_{+}(U,\tau^{2})+\partial_{v}u(\rho,v,\tau^{2})\big)\partial_{\alpha}\varphi}
{\rho_{0}\partial_{\rho}u(\rho,v,\tau^{2})+\big(\lambda_{-}(U,\tau^{2})+\partial_{v}u(\rho,v,\tau^{2})\big)\partial_{\alpha}\varphi}.
\end{split}
\end{eqnarray*}

\par With the help of \eqref{eq:2.6} and \eqref{eq:2.22}, we have 
\begin{eqnarray*}
\begin{split}
\Psi_{1}\big|_{\tau=0}&=\frac{\rho_{0}\partial_{\rho}u(\rho,v,\tau^{2})\big|_{\tau=0}+\big(\lambda_{+}(U,\tau^{2})
+\partial_{v}u(\rho,v,\tau^{2})\big)
\partial_{\alpha}\varphi\big|_{\tau=0}}
{\rho_{0}\partial_{\rho}u(\rho,v,\tau^{2})\big|_{\tau=0}+\big(\lambda_{-}(U,\tau^{2})
+\partial_{v}u(\rho,v,\tau^{2})\big)\partial_{\alpha}\varphi\big|_{\tau=0}}\\[5pt]
&=-\frac{2(1-\alpha^{\gamma-1})+(\gamma-1)(1-\alpha^{2})\alpha^{\gamma-2}
-\sqrt{2(\gamma-1)(1-\alpha)(1-\alpha^{\gamma-1})(1+\alpha)^{3}\alpha^{\gamma-3}}}
{2(1-\alpha^{\gamma-1})+(\gamma-1)(1-\alpha^{2})\alpha^{\gamma-2}
+\sqrt{2(\gamma-1)(1-\alpha)(1-\alpha^{\gamma-1})(1+\alpha)^{3}\alpha^{\gamma-3}}}.
\end{split}
\end{eqnarray*}

By Lemma \ref{lem:2.6},
we know that $\beta_{-}=\omega_{-,0}-\omega_{-}$ is monotonically decreasing with respect to $\alpha$
when $0<\alpha\leq1$. Notice that $\beta_{-}=0$ when $\alpha=1$. Therefore, for $0<\alpha<1$,
$\beta_{-}=\omega_{-,0}-\omega_{-}>0$.

\par Next, let us consider $\frac{\partial^{2}\Phi_{1}(\beta_{-}, U_{0}; \tau^{2})}{\partial \beta^{2}_{-}}$. Note that
\begin{eqnarray*}
\frac{\partial \Psi_{1}(\alpha, U_{0}; \tau^{2})}{\partial \alpha}=\Big(\rho_{0}\partial_{\rho}u(\rho,v,\tau^{2})
+\big(\lambda_{-}(U,\tau^{2})+\partial_{v}u(\rho,v,\tau^{2})\big)\partial_{\alpha}\varphi\Big)^{-2}\mathcal{J}(U,\tau^{2}),
\end{eqnarray*}
where
\begin{eqnarray*}
\begin{split}
\mathcal{J}(U,\tau^{2})&=\rho^{2}_{0}\Big(\big(\lambda_{-}-\lambda_{+}\big)\partial^{2}_{\rho\rho}u(\rho,v,\tau^{2})
+\big(\partial_{\rho}\lambda_{+}-\partial_{\rho}\lambda_{-}\big)\partial_{\rho}u(\rho,v,\tau^{2})\Big)
\partial_{\alpha}\varphi\\[5pt]
&\ \ \ +\rho_{0}\Big(2\big(\lambda_{-}-\lambda_{+}\big)\partial^{2}_{\rho v}u(\rho,v,\tau^{2})
+\big(\lambda_{-}+\partial_{v}u(\rho,v,\tau^{2})\big)\partial_{\rho}\lambda_{+}\\[5pt]
&\ \ \ -\big(\lambda_{+}+\partial_{v}u(\rho,v,\tau^{2})\big)\partial_{\rho}\lambda_{-}
+\big(\partial_{v}\lambda_{+}-\partial_{v}\lambda_{-}\big)\partial_{\rho}u(\rho,v,\tau^{2})\Big)
(\partial_{\alpha}\varphi)^{2}\\[5pt]
&\ \ \ +\Big(\big(\partial_{v}\lambda_{+}+\partial^{2}_{vv}u(\rho,v,\tau^{2})\big)\big(\lambda_{-}+\partial_{v}u(\rho,v,\tau^{2})\big)
\\[5pt]
&\ \ \ -\big(\partial_{v}\lambda_{-}+\partial^{2}_{vv}u(\rho,v,\tau^{2})\big)
\big(\lambda_{+}+\partial_{v}u(\rho,v,\tau^{2})\big)\Big)(\partial_{\alpha}\varphi)^{3}
+\rho_{0}\big(\lambda_{+}-\lambda_{-}\big)\partial_{\rho}u(\rho,v,\tau^{2})\partial^{2}_{\alpha\alpha}\varphi.
\end{split}
\end{eqnarray*}

\par When $\tau=0$, by Lemma \ref{lem:2.1} and Lemma \ref{lem:2.5}, we get
\begin{eqnarray*}
\begin{split}
\mathcal{J}(U,0)&=a^{-3}_{\infty}\rho^{\frac{3(\gamma-1)}{2}}_{0}\alpha^{\frac{3\gamma-7}{2}}
\Big((\gamma-3)\partial_{\alpha}\varphi\big|_{\tau=0}-2\alpha\partial^{2}_{\alpha\alpha}\varphi\big|_{\tau=0}\Big)\\[5pt]
&=\frac{2\big[(\gamma+1)\alpha^{2}-2\alpha+3-\gamma\big]\Big(\frac{1-\alpha^{\gamma-1}}{\gamma-1}\Big)^{2}
-4\alpha^{\gamma-1}(1-\alpha^{2})\frac{1-\alpha^{\gamma-1}}{\gamma-1}+\alpha^{\gamma-2}(\alpha^{2}-1)^{2}}
{\sqrt{2(\gamma-1)(1-\alpha)(1-\alpha^{\gamma-1})(1+\alpha)^{3}}(1-\alpha^{2})(1-\alpha^{\gamma-1})}\\[5pt]
&\quad \ \times \Big(-a^{4}_{\infty}(\gamma-1)^{2}\rho^{2(\gamma-1)}_{0}\alpha^{\frac{3\gamma-7}{2}}\Big).
\end{split}
\end{eqnarray*}

On the other hand, we also have
\begin{eqnarray*}
\begin{split}
&\ \ \  \Big(\rho_{0}\partial_{\rho}u(\rho, v,\tau^{2})
+\big(\lambda_{-}(U,\tau^{2})+\partial_{v}u(\rho, v,\tau^{2})\big)\partial_{\alpha}\varphi\Big)^{-2}\\[5pt]
&=\frac{2(\gamma-1)a^{4}_{\infty}\rho^{-2(\gamma-1)}_{0}\alpha^{-(\gamma-1)}(1-\alpha)(1-\alpha^{\gamma-1})(1+\alpha)^{3}}
{\Big(2(1-\alpha^{\gamma-1})+(\gamma-1)\alpha^{\gamma-2}(1-\alpha^{2})
+\sqrt{2(\gamma-1)(1-\alpha)(1-\alpha^{\gamma-1})(1+\alpha)^{3}\alpha^{\gamma-3}}\Big)^{2}}.
\end{split}
\end{eqnarray*}

With the above two equalities, we have
\begin{eqnarray*}
\begin{split}
\frac{\partial \Psi_{1}}{\partial \alpha}\Big|_{\tau=0}
&=-\frac{(\gamma-1)^{2}\alpha^{\frac{\gamma-5}{2}}\sqrt{2(\gamma-1)(1-\alpha)(1-\alpha^{\gamma-1})}(1+\alpha)}
{(1-\alpha)(1-\alpha^{\gamma-1})}\\[5pt]
& \ \ \ \times \frac{2\big[(\gamma+1)\alpha^{2}-2\alpha+3-\gamma\big]\Big(\frac{1-\alpha^{\gamma-1}}{\gamma-1}\Big)^{2}
-4\alpha^{\gamma-1}(1-\alpha^{2})\frac{1-\alpha^{\gamma-1}}{\gamma-1}+\alpha^{\gamma-2}(\alpha^{2}-1)^{2}}
{\Big(2(1-\alpha^{\gamma-1})+(\gamma-1)\alpha^{\gamma-2}(1-\alpha^{2})
+\sqrt{2(\gamma-1)(1-\alpha)(1-\alpha^{\gamma-1})(1+\alpha)^{3}\alpha^{\gamma-3}}\Big)^{2}}.
\end{split}
\end{eqnarray*}

Let
\begin{align*}
J(\alpha, \gamma):&=2\big[(\gamma+1)\alpha^{2}-2\alpha+3-\gamma\big]\Big(\frac{1-\alpha^{\gamma-1}}{\gamma-1}\Big)^{2}
-4(1-\alpha^{2})\alpha^{\gamma-1}\Big(\frac{1-\alpha^{\gamma-1}}{\gamma-1}\Big)
+\big(1-\alpha^{2}\big)^{2}\alpha^{\gamma-2}.
\end{align*}
Note that
\begin{eqnarray*}
\begin{split}
\Delta&=\Big(-4(1-\alpha^{2})\alpha^{\gamma-1}\Big)^{2}-8\big[(\gamma+1)\alpha^{2}-2\alpha+3-\gamma\big]
\big(1-\alpha^{2}\big)^{2}\alpha^{\gamma-2}\\[5pt]
&=8\big(1-\alpha^{2}\big)^{2}\alpha^{\gamma-2}\big[2\alpha^{\gamma}-(\gamma+1)\alpha^{2}+2\alpha-3+\gamma\big]\\[5pt]
&=8\big(1-\alpha^{2}\big)^{2}\alpha^{\gamma-2}\Delta_{0}(\alpha,\gamma),
\end{split}
\end{eqnarray*}
where $\Delta_{0}(\alpha,\gamma)=2\alpha^{\gamma}-(\gamma+1)\alpha^{2}+2\alpha-3+\gamma$.
Obviously, $\Delta_{0}(1,\gamma)=0$, and for $0<\alpha<1$ and $1\leq\gamma\leq 2$, we have
$\partial_{\alpha}\Delta_{0}(\alpha,\gamma)=2\gamma \alpha(\alpha^{\gamma-2}-1)+2(1-\alpha)>0$.
So $\Delta<0$ when $0<\alpha<1$ and $1\leq\gamma\leq 2$.
Therefore, $J(\alpha,\gamma)>0$ when $0<\alpha<1$ and $1\leq\gamma\leq 2$.
Thus
$$
\frac{\partial \Psi_{1}}{\partial \alpha}\Big|_{\tau=0}<0,
$$
when $0<\alpha<1$ and $1\leq \gamma \leq 2$. So for $0<\alpha<1$ and $1\leq \gamma \leq 2$,
\begin{align*}
\frac{\partial^{2} \Phi_{1}(\beta_{-}, U_{0};0)}{\partial \beta^{2}_{-}}=
\Big(\frac{\partial(\omega_{-,0}-\omega_{-})}{\partial \alpha}\Big)^{-1}\Big|_{\tau=0}
\frac{\partial \Psi_{1}(\alpha, U_{0}; \tau^{2})}{\partial \alpha}\Big|_{\tau=0}>0.
\end{align*}
Moreover, because $\Psi_{1}(1,U_{0}; 0)=0$, and $\big|\Psi_{1}\big|_{\tau=0}\big|<1$ for $0<\alpha<1$, we have
\begin{equation}
0<\frac{\partial\Phi_{1}(\beta_{-}, U_{0};0)}{\partial \beta_{-}}<1.
\end{equation}
Finally, for $\alpha>\varepsilon_{0}>0$, we can choose $\epsilon_{4}>0$ sufficiently small and a constant $C_{1}>0$ independent of $\tau$ such that when $\tau\in(0,\epsilon_{4})$,
$C_{1}\varepsilon_{0}<\Psi_{1}(\alpha, U_{0}; \tau^{2})<1$.
\end{proof}
Based on the proof, actually, when $\tau=0$, we have the following lemma.
\begin{lemma}\label{lem:2.8}
	If $\gamma\in[1,2]$, then the shock curve $S_1$ starting at $(r_0,s_0)$ can be written as
	\begin{equation}\label{eq:2.60}
	s_0-s=g_1(r_0-r,\,\rho_0)=\int_0^{r_0-r}h_1(\alpha)\left|_{\alpha=\alpha_1(\frac{\beta}{\rho_0^{(\gamma-1)/2}})}\right.\textrm{d}\beta,
	\end{equation}
	where $0\leq\frac{\partial g_1(\beta,\rho_0)}{\partial\beta}<1$, $\frac{\partial^2 g_1(\beta,\rho_0)}{\partial^2\beta}\geq0$, and $\beta=r_0-r\geq0$.
\end{lemma}

Next, let us consider the shock wave curve $\mathcal{S}_{2}$.
\begin{lemma}\label{lem:2.9}
If $\gamma\in[1,2]$ and $\alpha>1$, there exists a constant $\epsilon_{5}>0$ sufficiently small such that
for $\tau \in (0, \epsilon_{5})$, the shock wave curve $\mathcal{S}_{2}$ starting at $(\omega_{-,0}, \omega_{+,0})$
can be expressed as
\begin{eqnarray}\label{eq:2.34}
\beta_{-}=\Phi_{2}(\beta_{+}, U_{0}; \tau^{2})
=\int^{\beta_{+}}_{0}\Psi_{2}(\alpha, U_{0}; \tau^{2})\Big|_{\alpha=\alpha_{2}(\beta, U_{0};\tau^{2})}d\beta,
\end{eqnarray}
where $\beta_{+}=\omega_{+,0}-\omega_{+}<0$ and
\begin{eqnarray}\label{eq:2.35}
0<\frac{\partial\Phi_{2}(\beta_{+}, U_{0}; 0)}{\partial \beta_{+}}<1, \qquad
\frac{\partial^{2}\Phi_{2}(\beta_{+}, U_{0}; 0)}{\partial \beta^{2}_{+}}<0.
\end{eqnarray}
Moreover, if $\alpha<\varepsilon^{-1}_{0}$, it holds that
\begin{eqnarray}\label{eq:2.36}
0<\frac{\partial\Phi_{2}(\beta_{+}, U_{0}; \tau^{2})}{\partial \beta_{+}}<1-C_{2}\varepsilon_{0},
\end{eqnarray}

where constant $C_{2}>0$, depending on the data and $\varepsilon_{0}$, is independent of $\tau$.
\end{lemma}

\begin{proof}
By Lemma \ref{lem:2.6} and the implicit function theorem, we define
\begin{eqnarray*}
\Psi_{2}(\alpha, U_{0}; \tau^{2}):=\frac{\partial \Phi_{2}(\beta_{+}, U_{0}; \tau^{2})}{\partial \beta_{+}}.
\end{eqnarray*}
Then \eqref{eq:2.34} follows.
Furthermore, by the straightforward calculation, 
\begin{eqnarray*}
\begin{split}
\frac{\partial \Phi_{2}(\beta_{+}, U_{0}; \tau^{2})}{\partial \beta_{+}}
=\frac{\frac{\partial (\omega_{-,0}-\omega_{-})}{\partial \alpha}}{\frac{\partial (\omega_{+,0}-\omega_{+})}{\partial \alpha}}
=\frac{\rho_{0}\partial_{\rho}u(\rho, v,\tau^{2})+\big(\lambda_{-}(U,\tau^{2})+\partial_{v}u(\rho,v,\tau^{2})\big)\partial_{\alpha}\varphi}
{\rho_{0}\partial_{\rho}u(\rho, v,\tau^{2})+\big(\lambda_{+}(U,\tau^{2})+\partial_{v}u(\rho,v,\tau^{2})\big)\partial_{\alpha}\varphi}.
\end{split}
\end{eqnarray*}

\par When $\tau=0$, it follows from Lemma \ref{lem:2.1}, Remark \ref{rem:2.1} and Lemma \ref{lem:2.5} that
\begin{eqnarray*}
\begin{split}
\Psi_{2}\big|_{\tau=0}
&=-\frac{2(\alpha^{\gamma-1}-1)+(\gamma-1)(\alpha^{2}-1)\alpha^{\gamma-2}
-\sqrt{2(\gamma-1)(\alpha-1)(\alpha^{\gamma-1}-1)(\alpha+1)^{3}\alpha^{\gamma-3}}}
{2(\alpha^{\gamma-1}-1)+(\gamma-1)(\alpha^{2}-1)\alpha^{\gamma-2}
+\sqrt{2(\gamma-1)(\alpha-1)(\alpha^{\gamma-1}-1)(\alpha+1)^{3}\alpha^{\gamma-3}}}.
\end{split}
\end{eqnarray*}

By Lemma \ref{lem:2.5}, we know that $\beta_{+}=\omega_{+,0}-\omega_{+}$
is monotonically decreasing with respect to $\alpha$ when $\alpha>1$.
 Note that $\beta_{+}=0$ when $\alpha=1$, so $\beta_{+}=\omega_{+,0}-\omega_{+}>0$ when $\alpha>1$.

Next, let us consider $\frac{\partial^{2}\Phi_{2}(\beta_{+}, U_{0}; \tau^{2})}{\partial \beta^{2}_{+}}$. Note that
\begin{eqnarray*}
\frac{\partial \Psi_{2}(\alpha, U_{0}; \tau^{2})}{\partial \alpha}=\Big(\rho_{0}\partial_{\rho}u(\rho,v,\tau^{2})
+\big(\lambda_{+}(U,\tau^{2})+\partial_{v}u(\rho,v,\tau^{2})\big)\partial_{\alpha}\varphi\Big)^{-2}\tilde{\mathcal{J}}(U,\tau^{2}),
\end{eqnarray*}
where
\begin{eqnarray*}
\begin{split}
\tilde{\mathcal{J}}(U,\tau^{2})&=\rho^{2}_{0}\Big(\big(\lambda_{+}-\lambda_{-}\big)\partial^{2}_{\rho\rho}u(\rho,v,\tau^{2})
+\big(\partial_{\rho}\lambda_{-}-\partial_{\rho}\lambda_{+}\big)\partial_{\rho}u(\rho,v,\tau^{2})\Big)
\partial_{\alpha}\varphi\\[5pt]
&\ \ \ +\rho_{0}\Big(2\big(\lambda_{+}-\lambda_{-}\big)\partial^{2}_{\rho v}u(\rho,v,\tau^{2})
+\big(\lambda_{+}+\partial_{v}u(\rho,v,\tau^{2})\big)\partial_{\rho}\lambda_{-}\\[5pt]
&\ \ \ -\big(\lambda_{-}+\partial_{v}u(\rho,v,\tau^{2})\big)\partial_{\rho}\lambda_{+}
+\big(\partial_{v}\lambda_{-}-\partial_{v}\lambda_{+}\big)\partial_{\rho}u(\rho,v,\tau^{2})\Big)
(\partial_{\alpha}\varphi)^{2}\\[5pt]
&\ \ \ +\Big(\big(\partial_{v}\lambda_{-}+\partial^{2}_{vv}u(\rho,v,\tau^{2})\big)\big(\lambda_{+}+\partial_{v}u(\rho,v,\tau^{2})\big)\\[5pt]
&\ \ \ -\big(\partial_{v}\lambda_{+}+\partial^{2}_{vv}u(\rho,v,\tau^{2})\big)
\big(\lambda_{-}+\partial_{v}u(\rho,v,\tau^{2})\big)\Big)(\partial_{\alpha}\varphi)^{3}
+\rho_{0}\big(\lambda_{-}-\lambda_{+}\big)\partial_{\rho}u(\rho,v,\tau^{2})\partial^{2}_{\alpha\alpha}\varphi.
\end{split}
\end{eqnarray*}

\par So, for $\tau=0$, by Lemma \ref{lem:2.1} and Lemma \ref{lem:2.5}, we have
\begin{eqnarray*}
\begin{split}
\frac{\partial \Psi_{2}}{\partial \alpha}\Big|_{\tau=0}
&=\frac{(\gamma-1)^{2}\alpha^{\frac{\gamma-5}{2}}\sqrt{2(\gamma-1)(\alpha-1)(\alpha^{\gamma-1}-1)}(\alpha+1)}
{(\alpha-1)(\alpha^{\gamma-1}-1)}\\[5pt]
& \ \ \ \times \frac{2\big[(\gamma+1)\alpha^{2}-2\alpha+3-\gamma\big]\Big(\frac{\alpha^{\gamma-1}-1}{\gamma-1}\Big)^{2}
-4\alpha^{\gamma-1}(1-\alpha^{2})\frac{1-\alpha^{\gamma-1}}{\gamma-1}+\alpha^{\gamma-2}(\alpha^{2}-1)^{2}}
{\Big(2(\alpha^{\gamma-1}-1)+(\gamma-1)\alpha^{\gamma-2}(\alpha^{2}-1)
+\sqrt{2(\gamma-1)(\alpha-1)(\alpha^{\gamma-1}-1)(\alpha+1)^{3}\alpha^{\gamma-3}}\Big)^{2}}.
\end{split}
\end{eqnarray*}
Define
\begin{align*}
\tilde{J}(\alpha,\gamma):&=2\big[(\gamma+1)\alpha^{2}-2\alpha+3-\gamma\big]\Big(\frac{\alpha^{\gamma-1}-1}{\gamma-1}\Big)^{2}
-4(\alpha^{2}-1)\alpha^{\gamma-1}\Big(\frac{\alpha^{\gamma-1}-1}{\gamma-1}\Big)
+\big(\alpha^{2}-1\big)^{2}\alpha^{\gamma-2}.
\end{align*}
Similar as the argument in the proof of Lemma \ref{lem:2.7}, we can show that $\tilde{J}(\alpha,\gamma)>0$ when $\alpha>1$ and $1\leq\gamma\leq 2$. Thus, we have $\frac{\partial \Psi_{2}}{\partial \alpha}\Big|_{\tau=0}>0$ when $\alpha>1$ and $1\leq \gamma \leq 2$. So
\begin{align*}
\frac{\partial^{2} \Phi_{2}(\beta_{+}, U_{0};0)}{\partial \beta^{2}_{+}}=
\Big(\frac{\partial(\omega_{+,0}-\omega_{+})}{\partial \alpha}\Big)^{-1}\Big|_{\tau=0}
\frac{\partial \Psi_{2}(\alpha, U_{0}; \tau^{2})}{\partial \alpha}\Big|_{\tau=0}<0,
\end{align*}
for $\alpha>1$ and $1\leq \gamma \leq 2$.
Moreover, by the facts that $\Psi_{2}(1,U_{0}; 0)=0$ and that $|\Psi_{2}\big|_{\tau=0}|<1$ for $\alpha>1$,
we have $0<\Psi_{2}\big|_{\tau=0}<1$.
For given $\varepsilon_{0}$, we can choose $\epsilon_{5}>0$ sufficiently small and a positive constant $C_{2}$ independent of $\tau$ such that for $\tau\in(0,\epsilon_{5})$ and $\alpha<\varepsilon^{-1}_{0}$
\begin{align*}
0<\Psi_{2}(\alpha, U_{0}; \tau^{2})<1-C_{2}\varepsilon_{0}.
\end{align*}
This completes the proof of the lemma.
\end{proof}
Based on the proof, we actually have the following lemma for $\tau=0$.
\begin{lemma}\label{lem:2.10}
If $\gamma\in[1,2]$, then the shock curve $S_2$ starting at $(r_0,s_0)$ can be rewritten as
\begin{equation}\label{eq:2.61}
r_0-r=g_2(s_0-s,\,\rho_0)\equiv\int_0^{s_0-s}h_2(\alpha)\left|_{\alpha=\alpha_2(\frac{\beta}{\rho_0^{(\gamma-1)/2}})}\right.\textrm{d}\beta,
\end{equation}
where $0<\frac{\partial g_2(\beta,\rho_0)}{\partial\beta}<1$, $\frac{\partial^2 g_2(\beta,\rho_0)}{\partial^2\beta}<0$,
and $\beta=s_0-s\leq0$.
\end{lemma}

\subsection{Riemann solutions of equations \eqref{eq:1.16}}
Based on lemma  \ref{lem:2.7} and lemma \ref{lem:2.8}, for any constant state
$\omega_{L}=(\omega_{-,L},\omega_{+,L})$, let
$$
\textsl{z}_{1}=\omega_{-,L}-\omega_{-} \qquad\mbox{and}\qquad \  \textsl{z}_{2}=\omega_{+,L}-\omega_{+}.
$$

Define
\begin{eqnarray}\label{eq:2.37}
\begin{split}
&\mathscr{H}^{(1)}_{1}(\textsl{z}_{1}, \omega_{L}; \tau^{2})=-\textsl{z}_{1}+\omega_{-,L},\\[5pt]
&\mathscr{H}^{(2)}_{1}(\textsl{z}_{1}, \omega_{L}; \tau^{2})=
\left\{
\begin{array}{lllll}
-\Phi_{1}(\textsl{z}_{1}, U_{L}; \tau^{2})+\omega_{+,L},\ \   &\textsl{z}_{1}>0,\\[5pt]
\omega_{+,L},\  \  &\textsl{z}_{1}<0,
\end{array} \right.
\end{split}
\end{eqnarray}
and
\begin{eqnarray}\label{eq:2.38}
\begin{split}
&\mathscr{H}^{(1)}_{2}(\textsl{z}_{2}, \omega_{L};\tau^{2})=
\left\{
\begin{array}{lllll}
-\Phi_{2}(\textsl{z}_{2}, U_{L}; \tau^{2})+\omega_{-,L},\ \   &\textsl{z}_{2}<0,\\[5pt]
\omega_{-,L},\ \  &\textsl{z}_{2}>0,
\end{array}
\right.\\[5pt]
&\mathscr{H}^{(2)}_{2}(\textsl{z}_{2}, \omega_{L}; \tau^{2})=-\textsl{z}_{2}+\omega_{+,L},
\end{split}
\end{eqnarray}
where functions $\Phi_{1}$ and $\Phi_{2}$ are given in Lemma \ref{lem:2.7} and Lemma \ref{lem:2.8}, respectively.
Let
\begin{eqnarray}\label{eq:2.39}
\begin{split}
&\mathscr{H}_{1}(\textsl{z}_{1}, \omega_{L};\tau^{2})=\big(\mathscr{H}^{(1)}_{1}, \mathscr{H}^{(2)}_{1}\big)(\textsl{z}_{1}, \omega_{L}; \tau^{2}),\\[5pt]
&\mathscr{H}_{2}(\textsl{z}_{2}, \omega_{L}, \tau^{2})=\big(\mathscr{H}^{(1)}_{2}, \mathscr{H}^{(2)}_{2}\big)(\textsl{z}_{2}, \omega_{L};\tau^{2}),
\end{split}
\end{eqnarray}
and finally denote
\begin{eqnarray}\label{eq:2.40}
\begin{split}
\mathscr{H}(\boldsymbol{z}, \omega_{L}; \tau^{2})=:\mathscr{H}_{1}(\textsl{z}_{1}, \mathscr{H}_{2}(\textsl{z}_{2}, \omega_{L}; \tau^{2}); \tau^{2}), \ \ \ \boldsymbol{z}=(\textsl{z}_{1}, \textsl{z}_{2}).
\end{split}
\end{eqnarray}

\par Then, we can parameterize the 1-waves by $\textsl{z}_{1}$ and parameterize the 2-waves by $\textsl{z}_{2}$. For the case that $\tau=0$, we set $\boldsymbol{z}:=\mathbf{z}=(z_{1},z_{2})$ and $\omega:=\omega^{0}=(r,s)$.

\smallskip
\par Now, let us consider the Riemann problem of \eqref{eq:1.16} with large initial data at $x=x_{0}$
\begin{eqnarray}\label{eq:2.41}
U(x, y)\big|_{x=x_{0}}=
\left\{
\begin{array}{lllll}
U_{L},\ \ \ &y<y_{0},\\[5pt]
U_{R},\ \ \ &y>y_{0},
\end{array} \right.
\end{eqnarray}
where  $U_{L}=(\rho_L,v_L)$ and $U_{R}=(\rho_R,v_R)$ are two given constant states satisfying $\rho_{L}>0$ and $\rho_{R}>0$ (see Fig. \ref{fig2.1x}). We have the following proposition that gives the solvability and the invariant region of the Riemann problem of \eqref{eq:1.16} and \eqref{eq:2.41}.

\vspace{5pt}
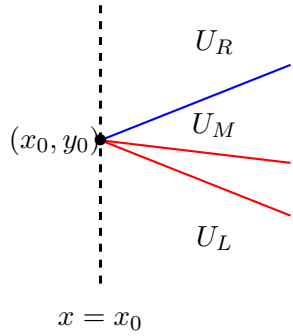
\begin{figure}[ht]
\begin{center}
\begin{tikzpicture}[scale=1]
\draw [line width=0.04cm][dashed](-1.5,-1.4) --(-1.5,2.3);
\draw [thick][blue](-1.5,0.5)--(1,1.5);
\draw [thick][red](-1.5,0.5)--(1,0.2);
\draw [thick][red](-1.5,0.5)--(1,-0.5);
\node at (-1.5, 0.5){$\bullet$};
\node at (-1.5, -1.9){$x=x_{0}$};
\node at (0, -0.8){$U_{L}$};
\node at (0, 0.7){$U_{M}$};
\node at (0, 1.8){$U_{R}$};
\node at (-2.1, 0.5){$(x_{0},y_{0})$};
\end{tikzpicture}
\end{center}
\caption{Riemann problem without boundary}\label{fig2.1x}
\end{figure}

\begin{proposition}\label{prop:2.1}
Suppose that $\omega_{-,L}+\omega_{+,R}>-\frac{4-\varepsilon_{0}}{\gamma-1}$ for some constant $0<\varepsilon_{0}<4$, then there exists a sufficiently small constant $\epsilon_{6}>0$ such that for any $\tau \in [0, \epsilon_{6})$, Riemann problem \eqref{eq:1.16} and \eqref{eq:2.41}
admits a unique piecewise smooth solution $U(x,y)$ without the vacuum state. Moreover solution $U(x,y)$ satisfies
\begin{eqnarray}\label{eq:2.42}
\begin{split}
&\omega_{-}(U(x,y), \tau^{2})+\omega_{+}(U(x,y),\tau^{2})\geq \omega_{-,L}+\omega_{+,R},
\end{split}
\end{eqnarray}
where $\omega_{\pm,L}=\omega_{\pm}(U_L, \tau^{2})$ and $\omega_{\pm,R}=\omega_{\pm}(U_R, \tau^{2})$.
\end{proposition}

\begin{proof}
The existence of the solutions of Riemann problem \eqref{eq:1.16} and \eqref{eq:2.41} is equivalent to the existence of solutions $\boldsymbol{z}$ of the following equation,
\begin{eqnarray}\label{eq:2.43}
\begin{split}
\omega_{R}=\mathscr{H}(\boldsymbol{z}, \omega_{L}; \tau^{2}).
\end{split}
\end{eqnarray}


From \eqref{eq:2.40}, we know that
\begin{eqnarray*}
\begin{split}
\det\big(\nabla_{\boldsymbol{z}}\mathscr{H}\big)(\boldsymbol{z}, \omega_{L};\tau^{2})=
\det\Big(\nabla_{\mathscr{H}_{2}}\mathscr{H}_{1}\cdot \nabla_{\textsl{z}_{2}}\mathscr{H}_{2},
\ \  \nabla_{\textsl{z}_{1}}\mathscr{H}_{1}\Big)(\boldsymbol{z}, \omega_{L}; \tau^{2}).
\end{split}
\end{eqnarray*}

Based on the sign of $\textsl{z}_1$ and $\textsl{z}_2$, we divide the proof into three cases for checking the sign of
the determinant above to show the existence of solution $\boldsymbol{z}$ of equation \eqref{eq:2.43}.

\par \emph{Case} \rm (i). \ \  $\textsl{z}_{1}>0$ and $\textsl{z}_{2}<0$.
By the definition of $\mathscr{H}_1^{(1)}$ and $\mathscr{H}_1^{(2)}$, we know that
\begin{eqnarray*}
\begin{split}
\mathscr{H}^{(1)}_{2}(\textsl{z}_{2}, U_{L}; \tau^{2})=-\Phi_{2}(\textsl{z}_{2}, U_{L}; \tau^{2})+\omega_{-,L},  \ \
\mathscr{H}^{(2)}_{2}(\textsl{z}_{2}, U_{L}; \tau^{2})=-\textsl{z}_{2}+\omega_{+,L},
\end{split}
\end{eqnarray*}
and
\begin{eqnarray*}
\begin{split}
&\mathscr{H}^{(1)}_{1}(\textsl{z}_{1}, \mathscr{H}_{2}(\textsl{z}_{2}, \omega_{L};\tau^{2});\tau^{2})
=-\textsl{z}_{1}+\mathscr{H}^{(1)}_{2}(\textsl{z}_{2}, \omega_{L};\tau^{2}),\\[5pt]
&\mathscr{H}^{(2)}_{1}(\textsl{z}_{1},\mathscr{H}_{2}(\textsl{z}_{2}, \omega_{L};\tau^{2});\tau^{2})
=-\Phi_{1}\Big(\textsl{z}_{1},U\big(\mathscr{H}_{2}(\textsl{z}_{2}, \omega_{L};\tau^{2})\big);\tau^{2}\Big)
+\mathscr{H}^{(2)}_{2}(\textsl{z}_{2}, \omega_{L};\tau^{2}).
\end{split}
\end{eqnarray*}
So
\begin{eqnarray*}
\begin{split}
&\nabla_{\mathscr{H}_{2}}\mathscr{H}_{1}\cdot\nabla_{\textsl{z}_{2}}\mathscr{H}_{2}\\[5pt]
=&\left(
\begin{array}{cccccc}
1 & \ \ 0 \\[15pt]
-\nabla_{U}\Phi_{1}\cdot\partial_{\mathscr{H}^{(1)}_{2}}U
& \ \   -\nabla_{U}\Phi_{1}\cdot\partial_{\mathscr{H}^{(2)}_{2}}U+1
\end{array}
\right)
\cdot\Big(-\partial_{\textsl{z}_{2}}\Phi_{2}(\textsl{z}_{2},U_{L};\tau^{2}),\ -1 \Big)^{\top}\\[5pt]
=&\bigg(-\partial_{\textsl{z}_{2}}\Phi_{2}(\textsl{z}_{2},U_{L};\tau^{2}), \  \  \nabla_{U}\Phi_{1}\cdot\partial_{\mathscr{H}^{(1)}_{2}}U
+\nabla_{U}\Phi_{1}\cdot\partial_{\mathscr{H}^{(2)}_{2}}U\cdot\partial_{\textsl{z}_{2}}\Phi_{2}(\textsl{z}_{2},U_{L};\tau^{2})-1
\bigg)^{\top},
\end{split}
\end{eqnarray*}
and
\begin{eqnarray*}\label{eq:2.12}
\begin{split}
\nabla_{\textsl{z}_{1}}\mathscr{H}_{1}=\bigg(-\partial_{\textsl{z}_{1}}
\Phi_{1}\big(\textsl{z}_{1},U(\mathscr{H}_{2}(\textsl{z}_{2}, \omega_{L};\tau^{2})) ;\tau^{2}\big),\ -1\bigg)^{\top}.
\end{split}
\end{eqnarray*}

Note that it follows from Lemma \ref{lem:2.3} that, 
\begin{align*}
\nabla_{U}\Phi_{1}\cdot\partial_{\mathscr{H}^{(1)}_{2}}U\bigg|_{\tau=0}&=\nabla_{U}\Phi_{1}\cdot\partial_{\mathscr{H}^{(2)}_{2}}U\bigg|_{\tau=0}\\[5pt]
&=\frac{1}{2}\rho^{\frac{3-\gamma}{2}}\partial_{\rho}\Phi_{1}\Big(\textsl{z}_{1}, U\big(\mathscr{H}_{2}(\textsl{z}_{2}, \omega(U_{L},0);0)\big); 0\Big)\\[5pt]
&=\frac{\gamma-1}{4}\partial_{\rho^{\frac{\gamma-1}{2}}}\Phi_{2}\Big(\textsl{z}_{1}, U\big(\mathscr{H}_{1}(\textsl{z}_{2}, \omega(U_{L},0);0)\big); 0\Big)\\[5pt]
&=-\frac{1}{2}\partial_{\textsl{z}_{1}}\Phi_{1}\Big(\textsl{z}_{1}, U\big(\mathscr{H}_{2}(\textsl{z}_{2}, \omega(U_{L},0);0)\big); 0\Big).
\end{align*}

Then, by Lemma \ref{lem:2.7} and Lemma \ref{lem:2.8},
\begin{eqnarray*}
\begin{split}
\det\big(\nabla_{\boldsymbol{z}}\mathscr{H}\big)(\boldsymbol{z}, \omega_{L};\tau^{2})\bigg|_{\tau=0}
=&\det\Big(\nabla_{\mathscr{H}_{2}}\mathscr{H}_{1}\cdot \nabla_{\textsl{z}_{2}}\mathscr{H}_{2},
\ \  \nabla_{\textsl{z}_{1}}\mathscr{H}_{1}\Big)(\boldsymbol{z}, \omega_{L}; \tau^{2})\big|_{\tau=0}\\[5pt]
=&-1+\frac{1}{2}\partial_{\textsl{z}_{1}}\Phi_{1}\Big(\textsl{z}_{1}, U\big(\mathscr{H}_{2}(\textsl{z}_{2}, \omega(U_{L},0);0)\big) ; 0\Big)
\cdot \Big(\partial_{\textsl{z}_{2}}\Phi_{2}\big(\textsl{z}_{2}, U_{L} ; 0\big)-1\Big)\\[5pt]
<&-1.
\end{split}
\end{eqnarray*}
Thus, for $\tau$ sufficiently small, we can get the existence of solution $\boldsymbol{z}$ of equation \eqref{eq:2.43} by applying the implicit function theorem.
Moreover, by the signs of $z_{1}$ and $z_{2}$, 
\begin{eqnarray*}
\begin{split}
\omega_{+}(U(x,y), \tau^{2})&=\omega_{+,R}-\Phi_{1}(\textsl{z}_{1}, U\big(\mathscr{H}_{2}(\textsl{z}_{2}, \omega_{L};\tau^{2})\big); \tau^{2})
&>\omega_{+,R}, \quad\mbox{for } \ \  \textsl{z}_{1}>0,
\end{split}
\end{eqnarray*}
and
\begin{eqnarray*}
\begin{split}
\omega_{-}(U(x,y),\tau^{2})&=\Phi_{1}\Big(\textsl{z}_{2},U_{L};\tau^{2}\Big)+\omega_{-,L}
&>\omega_{-,L}, \quad\mbox{for } \ \  \textsl{z}_{2}<0,
\end{split}
\end{eqnarray*}
which leads to the estimate \eqref{eq:2.42}.

\smallskip
\par \emph{Case} \rm (ii).\ \  $\textsl{z}_{1}<0$ and $\textsl{z}_{2}<0$ (or $\textsl{z}_{1}>0$ and $\textsl{z}_{2}>0$).
Without loss of the generality, we only consider the case that $\textsl{z}_{1}<0$ and $\textsl{z}_{2}<0$,
since the other case can be treated in the same way.
For the case that $\textsl{z}_{1}<0$ and $\textsl{z}_{2}<0$, notice that
\begin{eqnarray*}
\begin{split}
\mathscr{H}^{(1)}_{2}(\textsl{z}_{2},U_{L}; \tau^{2})=-\Phi_{2}(\textsl{z}_{2},U_{L}; \tau^{2})+\omega_{-,L},  \ \ \
\mathscr{H}^{(2)}_{2}(\textsl{z}_{2}, U_{L}; \tau^{2})=-\textsl{z}_{2}+\omega_{+,L},
\end{split}
\end{eqnarray*}
and
\begin{eqnarray*}
\begin{split}
&\mathscr{H}^{(1)}_{1}(\textsl{z}_{1}, \mathscr{H}_{2}(\textsl{z}_{2}, \omega_{L}; \tau^{2});\tau^{2})=-\textsl{z}_{1}
+\mathscr{H}^{(1)}_{2}(\textsl{z}_{2}, \omega_{L};\tau^{2}),\\[5pt]
&\mathscr{H}^{(2)}_{1}(\textsl{z}_{1}, \mathscr{H}_{2}(\textsl{z}_{2}, \omega_{L};\tau^{2});\tau^{2})=-\textsl{z}_{2}+\omega_{+,L}.
\end{split}
\end{eqnarray*}
So
\begin{eqnarray*}
\begin{split}
\det\big(\nabla_{\boldsymbol{z}}\mathscr{H}\big)(\boldsymbol{z}, \omega_{L};\tau^{2})\bigg|_{\tau=0}=\Phi_{2}(\textsl{z}_{2},U_{L}; \tau^{2})-1<-C\varepsilon_{0},
\end{split}
\end{eqnarray*}
and
\begin{eqnarray*}
\begin{split}
\omega_{-, R}=-\textsl{z}_{1}-\Phi_{2}(\textsl{z}_{2},U_{L}; \tau^{2})+\omega_{-,L},\ \ \ \omega_{+,R}=-\textsl{z}_{2}+\omega_{+,L}.
\end{split}
\end{eqnarray*}
Hence we can obtain the existence of solution $\boldsymbol{z}$ of equation \eqref{eq:2.43} directly, and
\begin{eqnarray*}
\begin{split}
\omega_{-}(U(x,y),\tau^{2})>\omega_{-,L}   \quad\mbox{and } \ \ \omega_{+}(U(x,y),\tau^{2})=\omega_{+,R},
\end{split}
\end{eqnarray*}
which leads to the estimate \eqref{eq:2.42}.

\smallskip
\par \emph{Case} \rm (iii).\ \  $\textsl{z}_{1}<0$ and $\textsl{z}_{2}>0$ . In this case, notice that
\begin{eqnarray*}
\begin{split}
\mathscr{H}^{(1)}_{2}(\textsl{z}_{2}, \omega_{L}; \tau^{2})=\omega_{-,L},  \ \ \
\mathscr{H}^{(2)}_{2}(\textsl{z}_{2}, \omega_{L}; \tau^{2})=-\textsl{z}_{2}+\omega_{+,L},
\end{split}
\end{eqnarray*}
and
\begin{eqnarray*}
\begin{split}
&\mathscr{H}^{(1)}_{1}(\textsl{z}_{1}, \mathscr{H}_{2}(\textsl{z}_{2}, \omega_{L}; \tau^{2}); \tau^{2})
=-\textsl{z}_{1}+\mathscr{H}^{(1)}_{2}(\textsl{z}_{2},\omega_{L}; \tau^{2}),\\[5pt]
&\mathscr{H}^{(2)}_{1}(\textsl{z}_{1}, \mathscr{H}_{2}(\textsl{z}_{2}, \omega_{L}; \tau^{2});\tau^{2})
=\mathscr{H}^{(2)}_{2}(\textsl{z}_{2}, \omega_{L}; \tau^{2}).
\end{split}
\end{eqnarray*}

Then
$$
\textsl{z}_{1}=\omega_{-,L}-\omega_{-,R} \qquad\mbox{ and }\qquad \textsl{z}_{2}=\omega_{+,L}-\omega_{+,R},
$$
So we obtain the existence of solution $\boldsymbol{z}$ directly, and in this case it is easy to see
$$\omega_{-}(U,\tau^{2})+\omega_{+}(U,\tau^{2})= \omega_{-,L}+\omega_{+,R}.$$

Moreover, notice that
\begin{eqnarray*}
\begin{split}
\rho^{\frac{\gamma-1}{2}}\Big|_{\tau=0}&=\frac{\gamma-1}{4}\Big(\omega_{-}(U,0)+\omega_{+}(U,0)\Big)+1\\[5pt]
&\geq \frac{\gamma-1}{4}\Big(\omega_{+}(U_{R},0)+\omega_{-}(U_{L},0)\Big)+1\\[5pt]
&>\hat{C}>0.
\end{split}
\end{eqnarray*}
Based on this fact and combining the arguments for \emph{Cases} \rm(i)-\rm(iii) together, we can choose $\epsilon_{6}>0$ sufficiently small such that for $\tau\in(0,\epsilon_{6})$, equation \eqref{eq:2.43} and then Riemann problem \eqref{eq:1.16} and \eqref{eq:2.41} admits a unique solution $\boldsymbol{z}$ without the vacuum states. Moreover, estimate \eqref{eq:2.42} follows. 
It completes the proof of the Proposition.
\end{proof}

\par Next, let us study the Riemann problem involving boundary. 
Define
\begin{eqnarray*}
&&\Omega_{0}=\{(x,y):\  x_{0}\leq x<x_{1}, \ y\leq b_{0}(x-x_{0})+y_{0} \}, \\[5pt]
&&\Gamma_{0}=\{(x,y): \ x_{0}\leq x<x_{1}, y=b_{0}(x-x_{0})+y_{0}\}.
\end{eqnarray*}
\par Let us consider the following Riemann problem (see Fig. \ref{fig2.1w}):
\begin{eqnarray}\label{eq:2.44}
\left\{
\begin{array}{llll}
\partial_{x}W(U,\tau^{2})+\partial_{y}F(U,\tau^{2})=0, & \quad \ in\ \  \Omega_{0},\\[5pt]
U(x,y)=U_{L},                     & \quad \ on\ \  \Omega_{0}\cap\{x=x_{0}\},\\[5pt]
v(x,y)=\big(1+\tau^{2}u(\rho, v,\tau^{2})\big)b_{0},                     &\quad\ on\ \   \Gamma_{0},
\end{array}
\right.
\end{eqnarray}
where $b_{0}<0$ and $U_{L}=(\rho_L,v_L)$ is a given constant state satisfying $\rho_{L}>0$.

\vspace{5pt}
\begin{figure}[ht]
\begin{center}
\begin{tikzpicture}[scale=1]
\draw [thin](-1.5,-1.4) --(-1.5,2.0);
\draw [thin](1.5,-1.4) --(1.5,2.0);
\draw [line width=0.06cm] (-2.5,1.8) --(2.5,1.2);
\draw [thick][red](-1.5,1.66)--(0.5,0.66);
\draw [thick][red](-1.5,1.66)--(0.4,0.36);
\draw [thick][red](-1.5,1.66)--(0.2,0);
\draw [thin](-1.5,1.7)--(-1.28,1.92);
\draw [thin](-1.3,1.65)--(-1.08,1.87);
\draw [thin](-1.1,1.62)--(-0.88,1.84);
\draw [thin](-0.9,1.59)--(-0.68,1.81);
\draw [thin](-0.7,1.56)--(-0.48,1.78);
\draw [thin](-0.5,1.54)--(-0.28,1.75);
\draw [thin](-0.3,1.53)--(-0.08,1.72);
\draw [thin](-0.1,1.49)--(0.12,1.69);
\draw [thin](0.1,1.46)--(0.32,1.66);
\draw [thin](0.3,1.44)--(0.52,1.63);
\draw [thin](0.5,1.42)--(0.72,1.60);
\draw [thin](0.7,1.39)--(0.92,1.57);
\draw [thin](0.9,1.37)--(1.12,1.56);
\draw [thin](1.1,1.35)--(1.32,1.55);
\draw [thin](1.3,1.32)--(1.52,1.51);
\node at (-1.5, 1.68){$\bullet$};
\node at (-1.5, -1.8){$x=x_{0}$};
\node at (1.5, -1.8){$x=x_{1}$};
\node at (-0.4, 0){$U_{L}$};
\node at (-0.4, 1.3){$U$};
\node at (-2.2, 1.4){$(x_{0},y_{0})$};
\node at (2.9, 1.2){$\Gamma_{0}$};
\end{tikzpicture}
\end{center}
\caption{Riemann problem with boundary}\label{fig2.1w}
\end{figure}
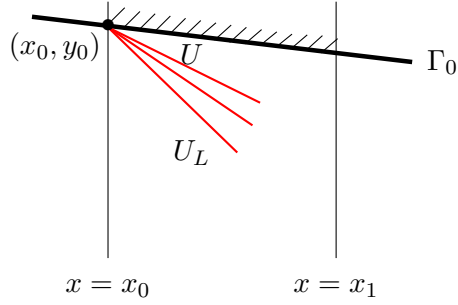

\par We have the following lemma on the solvability of Riemann problem \eqref{eq:2.44}.
\begin{proposition}\label{prop:2.2}
Assume that $\omega_{-, L}-a_{\infty}b_{0}>-\frac{2-\varepsilon_0}{\gamma-1}$ for some $0<\varepsilon_0<2$,
then there exists a small constant $\epsilon_{7}>0$ such that for any $\tau\in[0,\epsilon_{7})$, Riemann problem \eqref{eq:2.44} admits a unique piecewise smooth solution $U(x,y)$
consists of a single $2$-shock or a $2$-rarefaction wave without the vacuum states. 
Here, $\omega_{-,L}=\omega_{-}(U_L,\tau^{2})$.
\end{proposition}

\begin{proof}
It is easy to see that the existence of solutions of Riemann problem \eqref{eq:2.44} is equivalent to the existence of solutions $\textsl{z}_{2}$ of the following system
\begin{eqnarray}\label{eq:2.46}
\left\{
\begin{array}{llll}
\omega=\mathscr{H}_{2}(\textsl{z}_{2}, \omega_{L}; \tau^{2}),\\[5pt]
v=\mathcal{V}(\omega, \tau^{2})=\big(1+\tau^{2}u(\rho,v,\tau^{2})\big)b_{0}.
\end{array}
\right.
\end{eqnarray}
Let
\begin{eqnarray*}
\mathscr{G}(\textsl{z}_{2}, \omega_{L}, b_{0}; \tau^{2})
=\mathcal{V}(\mathscr{H}_{2}(\textsl{z}_{2}, \omega_{L}; \tau^{2}), \tau^{2})-\big(1+\tau^{2}u(\rho, v,\tau^{2})\big)b_{0},
\end{eqnarray*}
and consider equation $\mathscr{G}(\textsl{z}_{2}, \omega_{L}, b_{0}; \tau^{2})=0$ for $\tau$ sufficiently small.

Note that
\begin{eqnarray*}
\begin{split}
\mathscr{G}\big(\textsl{z}_{2}, \omega(U_{L},0), b_{0}; 0\big)
=\frac{1}{2a_{\infty}}\Big(\mathscr{H}^{(1)}_{2}(\textsl{z}_{2}, \omega(U_{L},0); 0)-\mathscr{H}^{(2)}_{2}(\textsl{z}_{2}, \omega(U_{L},0); 0)\Big)-b_{0}.
\end{split}
\end{eqnarray*}

If $b_{0}<v_{L}$, \emph{i.e.}, $\omega_{-}(U_{L},0)-\omega_{+}(U_{L},0)>2a_{\infty}b_{0}$, then
\begin{eqnarray*}
\begin{split}
\mathscr{G}\big(\textsl{z}_{2}, \omega(U_{L},0), b_{0}; 0\big)=\frac{1}{2a_{\infty}}\Big(\textsl{z}_{2}
-\Phi_{2}(\textsl{z}_{2}, U_{L};0)+\omega_{-}(U_{L},0)-\omega_{+}(U_{L},0)\Big)-b_{0}.
\end{split}
\end{eqnarray*}
By Lemma \ref{lem:2.8}, we get that
\begin{eqnarray*}
\begin{split}
\frac{\partial \mathscr{G}\big(\textsl{z}_{2}, \omega(U_{L},0), b_{0}; 0\big)}{\partial \textsl{z}_{2}}
=\frac{1}{2a_{\infty}}\Big(1-\partial_{\textsl{z}_{2}}\Phi_{2}(\textsl{z}_{2}, U_{L};0)\Big)>\frac{C_{2}\varepsilon_{0}}{2a_{\infty}}>0.
\end{split}
\end{eqnarray*}
On the other hand, we notice that $\mathscr{G}\in C^{2}$ with respect to $\textsl{z}_{2}$, then

\begin{eqnarray*}
\begin{split}
\mathscr{G}(0, \omega(U_{L},0), b_{0}; 0)=\frac{1}{2a_{\infty}}\big(\omega_{-}(U_{L},0)-\omega_{+}(U_{L},0)\big)-b_{0}>0,
\end{split}
\end{eqnarray*}
and
\begin{eqnarray*}
\begin{split}
\lim_{\textsl{z}_{2}\rightarrow-\infty}\mathscr{G}(\textsl{z}_{2},\omega(U_{L},0), b_{0}; 0)
&=\frac{1}{2a_{\infty}}\lim_{\textsl{z}_{2}\rightarrow-\infty}\bigg(1-\frac{\Phi_{2}(\textsl{z}_{2}, U_{L}; 0)}{\textsl{z}_{2}}\bigg)\textsl{z}_{2}
+\frac{\omega_{-}(U_{L},0)-\omega_{+}(U_{L},0)}{2a_{\infty}}-b_{0}\\[5pt]
&=-\infty.
\end{split}
\end{eqnarray*}

So, by the intermediate value theorem and the implicit function theorem, there exists a small constant $\epsilon'_{7}>0$ such that when $\tau\in[0,\epsilon'_{7})$, equation \eqref{eq:2.46} admits a unique solution $\textsl{z}_{2}<0$ which consists
of a shock wave belonging to the second family. There is no vacuum state, which can be verified by the observation that
%
$\omega_{-}(U,0)>\omega_{-}(U_{L},0)$, which leads to
\begin{eqnarray*}
\begin{split}
\rho^{\frac{\gamma-1}{2}}\Big|_{\tau=0}&=\frac{\gamma-1}{4}\Big(\omega_{+}(U,0)+\omega_{-}(U,0)\Big)+1\\[5pt]
&=\frac{\gamma-1}{2}\Big(\omega_{-}(U,0)-a_{\infty}b_{0}\Big)+1\\[5pt]
&>\frac{\gamma-1}{2}\Big(\omega_{-}(U_{L},0)-a_{\infty}b_{0}\Big)+1\\[5pt]
&>\tilde{C}.
\end{split}
\end{eqnarray*}

Second, if $b_{0}>v_{L}$, \emph{i.e.}, $\omega_{-}(U_{L},0)-\omega_{+}(U_{L},0)<2a_{\infty}b_{0}$, then states $U_{L}$ and $U$ are connected by a $2$-rarefaction wave $R_{2}$.
So, by \eqref{eq:2.38} and \eqref{eq:2.46}, we know that
\begin{eqnarray*}
\begin{split}
\omega_{-}(U,0)=\omega_{-}(U_{L},0),\ \ \ \omega_{+}(U,0)=\omega_{-}(U_{L},0)-2a_{\infty}b_{0}.
\end{split}
\end{eqnarray*}
This also gives that
\begin{eqnarray*}
\begin{split}
\rho^{\frac{\gamma-1}{2}}\Big|_{\tau=0}&=\frac{\gamma-1}{2}\Big(\omega_{-}(U_{L},0)-a_{\infty}b_{0}\Big)+1>\tilde{C},
\end{split}
\end{eqnarray*}
which means that the vacuum states dose not appear.
Moreover,
\begin{align*}
\mathscr{G}(\textsl{z}_{2}, \omega_{L}, b_{0}; \tau^{2})
=\mathcal{V}( (\omega_{-,L}, -\textsl{z}_{2}+\omega_{+,L}), \tau^{2})-\big(1+\tau^{2}u(\rho,v,\tau^{2})\big)b_{0}.
\end{align*}
so 
\begin{align*}
\frac{\partial \mathscr{G}(\textsl{z}_{2}, \omega_{L}, b_{0}; \tau^{2})}{\partial \textsl{z}_{2}}\Big|_{\tau=0}
=\frac{1}{2a_{\infty}}>0.
\end{align*}

Hence by the implicit function theorem, there exists a small constant
$\epsilon''_{7}>0$ such that for $\tau\in[0,\epsilon''_{7})$, equation \eqref{eq:2.46} admits a unique solution $\textsl{z}_{2}>0$
such that $U_{L}$ and $U$ are connected by a $2$-rarefaction wave $\mathcal{R}_{2}$ without the vacuum state.
\par Finally, take $\epsilon_{7}=\min\{\epsilon'_{7}, \epsilon''_{7}\}$, then when $\tau\in[0,\epsilon_{7})$,
we can get the existence of solutions of Riemann problem \eqref{eq:2.44} without the vacuum states.
\end{proof}

\section{Local interaction estimates}
In order to control the total variation of the approximate solutions which will be constructed in the next section, we need to study the local interaction estimates of the elementary waves of large data. Firstly, let us consider the estimates on the difference of the Riemann invariance of the same family along the corresponding shock wave curve.
\begin{figure}[ht]
\begin{center}
\begin{tikzpicture}[scale=0.9]
\draw [line width=0.05cm] (-1.5,-2.2) --(-1.5,2.2);
\draw [line width=0.05cm] (1.5,-2.2) --(1.5,2.2);
\draw [line width=0.03cm][red] (1.5,1.7)to[out=180, in=50](-1.5,0.5);
\draw [line width=0.03cm][red](1.5,-0.2)to[out=180, in=60](-1.5,-1.7);
\draw [thick][dashed] (1.5,1.72)--(-1.5,1.72);
\draw [thick][dashed] (1.5,-0.18)--(-1.5,-0.18);
\node at (1.5, 1.7){$\bullet$};
\node at (-1.5, 0.5){$\bullet$};
\node at (1.5, -0.2){$\bullet$};
\node at (-1.5, -1.7){$\bullet$};
\node at (2.3, 1.7){$(r_{0},s_{1})$};
\node at (2.3, -0.18){$(r_{0},s_{0})$};
\node at (-2.2, 0.5){$(r,s_2)$};
\node at (-2.2, -1.7){$(r,s)$};
\node at (1.5, -2.5){$r_{0}$};
\node at (-1.5, -2.5){$r$};
\node at (0, 1){$S_{1}$};
\node at (0, -1){$S_{1}$};
\end{tikzpicture}
\end{center}
\caption{Lemma \ref{lem:3.1}}\label{fig2.1xw}
\end{figure}
Let us consider them for the case that $\tau=0$ first. By Remark \ref{rem:2.2}, as shown in Fig. \ref{fig2.1xw}, let
\begin{eqnarray}\label{eq:3.13a}
\begin{split}
&r_{0}:=a_{\infty}v_{0}+ \frac{2(\rho^{\frac{\gamma-1}{2}}_{0}-1)}{\gamma-1}=a_{\infty}v_{1}+ \frac{2(\rho^{\frac{\gamma-1}{2}}_{1}-1)}{\gamma-1},\ \  s_{0}:=-a_{\infty}v_{0}+ \frac{2(\rho^{\frac{\gamma-1}{2}}_{0}-1)}{\gamma-1},\\[5pt]
&r:=a_{\infty}v+ \frac{2(\rho^{\frac{\gamma-1}{2}}-1)}{\gamma-1}=a_{\infty}v_{2}+ \frac{2(\rho^{\frac{\gamma-1}{2}}_{2}-1)}{\gamma-1},\ \  s:=-a_{\infty}v+ \frac{2(\rho^{\frac{\gamma-1}{2}}-1)}{\gamma-1},
\end{split}
\end{eqnarray}
and
\begin{eqnarray}\label{eq:3.13b}
\begin{split}
s_{1}:=-a_{\infty}v_{1}+ \frac{2(\rho^{\frac{\gamma-1}{2}}_{1}-1)}{\gamma-1},\quad
s_{2}:=-a_{\infty}v_{2}+ \frac{2(\rho^{\frac{\gamma-1}{2}}_{2}-1)}{\gamma-1}.
\end{split}
\end{eqnarray}
Then we have the following lemmas.
\begin{lemma}\label{lem:3.1}
Suppose $\tau=0$ and $s_1>s_0$. For two $S_1$ shock wave curves starting at points $(r_0,s_1)$ and $(r_0,s_0)$
and ending at points $(r,s_2)$ and $(r,s)$ respectively, if $0<\rho_{*}<\rho_i<\rho^{*}<\infty$ for $i=0$ and $1$, then there exists a constant $C_{3}>0$ depending only on
$\rho_{*}$ and $\rho^{*}$, such that
\begin{equation}\label{eq:3.13x}
0\leq(s_0-s)-(s_1-s_2)\leq C_{3}(\gamma-1)(s_1-s_{0})(r_0-r).
\end{equation}
\end{lemma}

\begin{proof}
Let $\Delta r=r_{0}-r$ and $\Delta s=s_{1}-s_{0}$. Notice that
\begin{eqnarray*}
\begin{split}
\rho_{0}-\rho_{1}=\frac{\gamma-1}{4}(s_{0}-s_{1})\leq0.
\end{split}
\end{eqnarray*}

Hence by Lemma \ref{lem:2.2}, for $\xi\in \Big(\frac{\beta}{\rho^{\frac{\gamma-1}{2}}_{1}}, \frac{\beta}{\rho^{\frac{\gamma-1}{2}}_{0}}\Big)$, we have
\begin{eqnarray*}
\begin{split}
s_0-s-(s_1-s_2)=\int^{\Delta r}_{0}\frac{\partial h_{1}}{\partial \alpha}\Big|_{\alpha=\alpha_{1}(\xi)}
\frac{\partial \alpha}{\partial\xi}\Big(\frac{\beta}{\rho^{\frac{\gamma-1}{2}}_{0}}
-\frac{\beta}{\rho^{\frac{\gamma-1}{2}}_{1}}\Big)d\beta\geq 0.
\end{split}
\end{eqnarray*}

So in order to show \eqref{eq:3.13x}, we only need to show
\begin{equation}\label{3.14}
(s_0-s)-(s_1-s_2)\leq C_{1}(\gamma-1)(s_1-s_{0})(r_0-r).
\end{equation}

\par Let $s_{2}=s^{*}(\Delta r, \Delta s; \gamma-1)$.
Then, by Lemma \ref{lem:2.9} and Lemma \ref{lem:2.10}, we know that $s^{*}$ is a $C^{2}$-function
of $\Delta r$, $\Delta s$ and $\gamma-1$.

\par For $\gamma=1$ and $\alpha=\frac{\rho_1}{\rho_0}$, we have
\begin{eqnarray*}
\begin{split}
\Delta s=-\sqrt{-\frac{2(1-\alpha)}{1+\alpha}\ln \alpha}-\ln \alpha,\ \ \
\Delta r=\sqrt{-\frac{2(1-\alpha)}{1+\alpha}\ln \alpha}-\ln \alpha.
\end{split}
\end{eqnarray*}

Notice that
\begin{eqnarray*}
\begin{split}
\frac{\partial\Delta r}{\partial \alpha}=-\frac{-2\alpha \ln\alpha+1-\alpha^{2}
-\sqrt{-2(1-\alpha)(1+\alpha)^{3}\ln \alpha}}{\sqrt{-2(1-\alpha)(1+\alpha)^{3}\ln \alpha}}<0.
\end{split}
\end{eqnarray*}

Then, by the implicit function theorem, $\alpha$ is a function of $\Delta r$ as $\alpha=\alpha(\Delta r)$, which is independent on $\rho_{0}$ and $\rho_{1}$.
Hence $\Delta s$ is a function of $\Delta r$ which is independent on $\rho_{0}$ and $\rho_{1}$.
Based on this observation, we thus deduce that for $\gamma=1$
\begin{eqnarray*}
\begin{split}
&s^{*}(0, 0; 0)-s^{*}(\Delta r, 0;0)-\big(s^{*}(0, \Delta s; 0)-s^{*}(\Delta r, \Delta s; 0) \big)=0.
\end{split}
\end{eqnarray*}
So
\begin{eqnarray}\label{eq:3.15}
\begin{split}
&s_{0}-s-\big(s_{1}-s_{2}\big)\\[5pt]
&\ \ =s^{*}(0, 0; \gamma-1)-s^{*}(\Delta r, 0;\gamma-1)-\big(s^{*}(0, \Delta s; \gamma-1)-s^{*}(\Delta r, \Delta s; \gamma-1) \big)\\[5pt]
&\ \ =s^{*}(0, 0; \gamma-1)-s^{*}(0, 0; 0)-\big(s^{*}(\Delta r, 0;\gamma-1)-s^{*}(\Delta r, 0;0)\big)\\[5pt]
&\ \ \ \ \ -\big(s^{*}(0, \Delta s; \gamma-1)-s^{*}(0, \Delta s; 0)\big)
+\big(s^{*}(\Delta r, \Delta s; \gamma-1)-s^{*}(\Delta r, \Delta s; 0)\big)\\[5pt]
&\ \ =(\gamma-1)\int^{1}_{0}e(\Delta r, \Delta s; \mu (\gamma-1))d \mu,
\end{split}
\label{2.27x}
\end{eqnarray}
where
\begin{eqnarray*}
\begin{split}
e(\Delta r, \Delta s; \mu (\gamma-1))&=\partial_{\gamma-1}s^{*}(0, 0; \mu(\gamma-1))-\partial_{\gamma-1}s^{*}(\Delta r, 0;\mu(\gamma-1))\\[5pt]
&\ \ \ -\partial_{\gamma-1}s^{*}(0, \Delta s; \mu(\gamma-1))+\partial_{\gamma-1}s^{*}(\Delta r, \Delta s;\mu(\gamma-1))\\[5pt]
&=O(1)\Delta r \Delta s.
\end{split}
\end{eqnarray*}

Substituting the estimate for $e(\Delta r, \Delta s; \mu (\gamma-1))$ into \eqref{eq:3.15}, we proved \eqref{eq:3.13x}.
It completes the proof.
\end{proof}

\vspace{5pt}
\begin{figure}[ht]
\begin{center}
\begin{tikzpicture}[scale=0.9]
\draw [line width=0.05cm] (-3.2,-1.5) --(3.6,-1.5);
\draw [line width=0.05cm] (-3.2,1) --(3.6,1);
\draw [line width=0.03cm][red] (-2.5,-1.5)to[out=80, in=195](-0.5,1);
\draw [line width=0.03cm][red](0.5,-1.5)to[out=80, in=190](2.8,1);
\draw [thick][dashed] (-2.52,-1.5)--(-2.52,1);
\draw [thick][dashed] (0.48,-1.5)--(0.48,1);
\node at (-2.5, -1.5){$\bullet$};
\node at (-0.5, 1){$\bullet$};
\node at (0.5, -1.5){$\bullet$};
\node at (2.8, 1){$\bullet$};
\node at (-2.6, -1.8){$(r_{1},s_{0})$};
\node at (0.6, -1.8){$(r_{0},s_{0})$};
\node at (-0.5, 1.3){$(r_2,s)$};
\node at (2.8, 1.3){$(r,s)$};
\node at (4, -1.5){$s_{0}$};
\node at (4, 1){$s$};
\node at (-1.2, 0){$S_{2}$};
\node at (1.8, 0){$S_{2}$};
\end{tikzpicture}
\end{center}
\caption{Lemma \ref{lem:3.2}}\label{fig2.2x}
\end{figure}

\vspace{5pt}
Similarly, we also have the estimate on the difference of $r$ on $S_{2}$ (see Fig. \ref{fig2.2x}).

\begin{lemma}\label{lem:3.2}
Assume $\tau=0$ and $r_0>r_1$. For two $S_2$ shock wave curves starting at points $(r_1,s_0)$ and $(r_0,s_0)$, and ending at points $(r_{2},s)$ and $(r,s)$, respectively, if $0<\rho_{*}<\rho_i<\rho^{*}<\infty$
for $i=0$ and $1$, then there exists a constant $C'_{3}>0$ depending only on
$\rho_{*}$ and $\rho^{*}$, such that
\begin{equation}\label{eq:3.16}
0\leq(r-r_{0})-(r_{2}-r_{1})\leq C'_{3}(\gamma-1)(r_0-r_{1})(s-s_{0}).
\end{equation}
\end{lemma}

\vspace{5pt}
\begin{figure}[ht]
\begin{center}\
\begin{tikzpicture}[scale=0.9]
\draw [line width=0.05cm] (-1.5,-2.2) --(-1.5,2.2);
\draw [line width=0.05cm] (1.5,-2.2) --(1.5,2.2);
\draw [line width=0.03cm][green] (1.5,1.7)to[out=180, in=50](-1.5,0.5);
\draw [line width=0.03cm][green](1.5,-0.2)to[out=180, in=60](-1.5,-1.7);
\draw [thick][dashed] (1.5,1.72)--(-1.5,1.72);
\draw [thick][dashed] (1.5,-0.18)--(-1.5,-0.18);
\node at (1.5, 1.7){$\bullet$};
\node at (-1.5, 0.5){$\bullet$};
\node at (1.5, -0.2){$\bullet$};
\node at (-1.5, -1.7){$\bullet$};
\node at (2.6, 1.7){$(\omega_{-,0},\omega_{+,1})$};
\node at (2.6, -0.18){$(\omega_{-,0},\omega_{+,0})$};
\node at (-2.5, 0.5){$(\omega_{-},\omega_{+,2})$};
\node at (-2.4, -1.7){$(\omega_{-},\omega_{+})$};
\node at (1.5, -2.5){$\omega_{-,0}$};
\node at (-1.5, -2.5){$\omega_{-}$};
\node at (0, 1){$\mathcal{S}_{1}$};
\node at (0, -1){$\mathcal{S}_{1}$};
\end{tikzpicture}
\end{center}
\caption{Lemma \ref{lem:3.3}}\label{fig2.1x1}
\end{figure}
\vspace{5pt}

Now, let us consider the case $\tau\neq 0$ in the following lemmas.
\begin{lemma}\label{lem:3.3}(see Fig. \ref{fig2.1x1})
Assume $\omega_{+,1}>\omega_{+,0}$. For two $\mathcal{S}_1$ shock wave curves starting at points $(\omega_{-,0},\omega_{+,1})$
and $(\omega_{-,0},\omega_{+,0})$ corresponding to $(\rho_1,v_1)$ and $(\rho_0,v_0)$ respectively, and ending at points
$(\omega_{-},\omega_{+,2})$ and $(\omega_{-},\omega_{+})$ corresponding to $(\rho_2,v_2)$ and $(\rho,v)$, respectively. If $0<\rho_{*}<\rho_i<\rho^{*}<\infty$ for $i=0$ and $1$, then there exists a constant $C_{4}>0$ depending only on $\rho_{*}$ and $\rho^{*}$,
such that
\begin{equation}\label{eq:3.17}
\omega_{+,0}-\omega_{+}-(\omega_{+,1}-\omega_{+,2})
\leq C_{4}(\gamma-1+\tau^{2})(\omega_{+,1}-\omega_{+,0})(\omega_{-,0}-\omega_{-}).
\end{equation}
\end{lemma}

\begin{proof}
Let $\Delta \omega_{-}=\omega_{-,0}-\omega_{-}$, $\Delta \omega_{+}=\omega_{+,1}-\omega_{+,0}$.
%
%
%
and let $\omega_{+,2}=\omega^{*}(\Delta \omega_{-}, \Delta \omega_{+},\tau^{2})$.
For the case that $\tau=0$, by Lemma \ref{lem:3.1},
\begin{align*}
\omega^{*}(0, 0,0)-\omega^{*}(\Delta \omega_{-}, 0,0)-\omega^{*}(0, \Delta \omega_{+},0)
+\omega^{*}(\Delta \omega_{-}, \Delta \omega_{+},0)=\mathcal{O}(1)(\gamma-1)\Delta \omega_{-}\Delta \omega_{+}.
\end{align*}
Therefore, we have
\begin{align*}
& \omega_{+,0}-\omega_{+}-(\omega_{+,1}-\omega_{+,2})\\[5pt]
&\ \ =\omega^{*}(0, 0, \tau^{2})-\omega^{*}(\Delta \omega_{-}, 0,\tau^{2})
-\Big(\omega^{*}(0, \Delta \omega_{+},\tau^{2})-\omega^{*}(\Delta \omega_{-}, \Delta \omega_{+},\tau^{2})\Big)\\[5pt]
&\ \ =\omega^{*}(0, 0, \tau^{2})-\omega^{*}(0, 0,0)-\Big(\omega^{*}(\Delta \omega_{-},0 ,\tau^{2})-\omega^{*}(\Delta \omega_{-}, 0,0)\Big)\\[5pt]
&\quad \ \ -\Big(\omega^{*}(0,\Delta \omega_{+},\tau^{2})-\omega^{*}(0,\Delta \omega_{-}, 0)\Big)
+\omega^{*}(\Delta \omega_{-}, \Delta \omega_{+},\tau^{2})-\omega^{*}(\Delta \omega_{-}, \Delta \omega_{+},0)\\[5pt]
&\quad \ \ +\omega^{*}(0, 0,0)-\omega^{*}(\Delta \omega_{-}, 0,0)-\omega^{*}(0, \Delta \omega_{+},0)
+\omega^{*}(\Delta \omega_{-}, \Delta \omega_{+},0)\\[5pt]
&\ \ =\tau^{2}\int^{1}_{0}\mathbf{e}(\Delta \omega_{-},\Delta \omega_{+},\mu \tau^{2})d\mu
+\mathcal{O}(1)(\gamma-1)\Delta \omega_{-}\Delta \omega_{+},
\end{align*}
where
\begin{align*}
\mathbf{e}(\Delta \omega_{-}, \Delta \omega_{+},\mu \tau^{2})
&=\partial_{\mu}\omega^{*}(0, 0,\mu \tau^{2})-\partial_{\mu}\omega^{*}(\Delta \omega_{-}, 0,\mu \tau^{2})\\[5pt]
&\ \ \ -\partial_{\mu}\omega^{*}(0, \Delta \omega_{+},\mu \tau^{2})
+\partial_{\mu}\omega^{*}(\Delta \omega_{-}, \Delta \omega_{+},\mu \tau^{2})\\[5pt]
&=\mathcal{O}(1)\Delta \omega_{-}\Delta \omega_{+}.
\end{align*}

Combining the above two estimates together, we have \eqref{eq:3.17}.
\end{proof}

\vspace{5pt}
\begin{figure}[ht]
\begin{center}
\begin{tikzpicture}[scale=0.9]
\draw [line width=0.05cm] (-3.2,-1.5) --(3.6,-1.5);
\draw [line width=0.05cm] (-3.2,1) --(3.6,1);
\draw [line width=0.03cm][blue] (-2.5,-1.5)to[out=80, in=195](-0.5,1);
\draw [line width=0.03cm][blue](0.5,-1.5)to[out=80, in=190](2.8,1);
\draw [thick][dashed] (-2.52,-1.5)--(-2.52,1);
\draw [thick][dashed] (0.48,-1.5)--(0.48,1);
\node at (-2.5, -1.5){$\bullet$};
\node at (-0.5, 1){$\bullet$};
\node at (0.5, -1.5){$\bullet$};
\node at (2.8, 1){$\bullet$};
\node at (-2.6, -1.8){$(\omega_{-,1},\omega_{+,0})$};
\node at (0.6, -1.8){$(\omega_{-,0},\omega_{+,0})$};
\node at (-0.5, 1.3){$(\omega_{-,2},\omega_{+})$};
\node at (2.8, 1.3){$(\omega_{-},\omega_{+})$};
\node at (4, -1.5){$\omega_{+,0}$};
\node at (4, 1){$\omega_{+}$};
\node at (-1.2, 0){$\mathcal{S}_{2}$};
\node at (1.8, 0){$\mathcal{S}_{2}$};
\end{tikzpicture}
\end{center}
\caption{Lemma \ref{lem:3.4}}\label{fig2.2}
\end{figure}

Similarly, we also have the estimate on the difference of $\omega_{-}$ on $\mathcal{S}_{2}$
shock wave curves. 
\begin{lemma}\label{lem:3.4}(see Fig. \ref{fig2.2})
Assume $\omega_{-,0}>\omega_{-,1}$. For $\mathcal{S}_2$ shock wave curves starting at points $(\omega_{-,1},\omega_{+,0})$
and $(\omega_{-,0},\omega_{+,0})$ corresponding to $(\rho_1,v_1)$ and $(\rho_0,v_0)$ respectively, and ending at points
$(\omega_{-,2},\omega_{+})$ and $(\omega_{-},\omega_{+})$ corresponding to $(\rho,v)$ and $(\rho_2,v_2)$, respectively. If $0<\rho_{*}<\rho_i<\rho^{*}<\infty$ for $i=0$ and $1$, then there exists a constant $C'_{4}>0$ depending only on $\rho_{*}$ and $\rho^{*}$,
such that
\begin{equation}\label{eq:3.18}
\omega_{-}-\omega_{-,0}-(\omega_{-,2}-\omega_{-,1})
\leq C'_{4}(\gamma-1+\tau^{2})(\omega_{-,0}-\omega_{-,1})(\omega_{+}-\omega_{+,0}).
\end{equation}
\end{lemma}

\begin{figure}[ht]
\begin{center}
\begin{tikzpicture}[scale=0.95]
\draw [line width=0.04cm] (-3.5,-6.5) --(-3.5,0.5);
\draw [line width=0.04cm] (-0.5,-6.5) --(-0.5,0.5);
\draw [line width=0.04cm] (2.5,-6.5) --(2.5,0.5);

\draw [thin] (-3.5,-0.5)--(2.5,-1.5);
\draw [thin] (-3.5,-2.5)--(2.5,-3.5);
\draw [thin] (-3.5,-4.5)--(2.5,-5.5);

\draw [thick][blue](-3.5,-0.5)--(-1.5,-1.5);
\draw [thick][red] (-3.5,-4.5)--(-1.5,-4.3);

\draw [thick][red] (-0.5,-3)--(1.3,-2.5);
\draw [thick][red] (-0.5,-3)--(1.3,-4.2);

\node at (-1.0, -1.6){$\beta(\pi)$};
\node at (-1.0, -4.3){$\nu(o)$};
\node at (1.9, -2.5){$\beta'(\pi')$};
\node at (1.9, -4.2){$\nu'(o')$};

\node at (-2.0, -5.3){$(\omega_{-, L},\omega_{+, L})$};
\node at (-2.0, -2.2){$(\omega_{-, M},\omega_{+, M})$};
\node at (-2.0, -0.2){$(\omega_{-, R},\omega_{+, R})$};

\node at (0.8, -4.7){$(\omega_{-, L},\omega_{+, L})$};
\node at (1.2, -3.2){$(\omega'_{-, M},\omega'_{+, M})$};
\node at (1.1, -1.8){$(\omega_{-, R},\omega_{+, R})$};

\node at (-2.0, -3.5){$\Omega_{\Delta, k-1}$};
\node at (1.0, -0.7){$\Omega_{\Delta, k}$};
\node at (-3.5, -6.9){$x_{k-1}$};
\node at (-0.5, -6.9){$x_{k}$};
\node at (2.5, -6.9){$x_{k+1}$};
\end{tikzpicture}
\end{center}
\caption{Local interaction estimates away from the boundary}\label{fig3.1x1}
\end{figure}
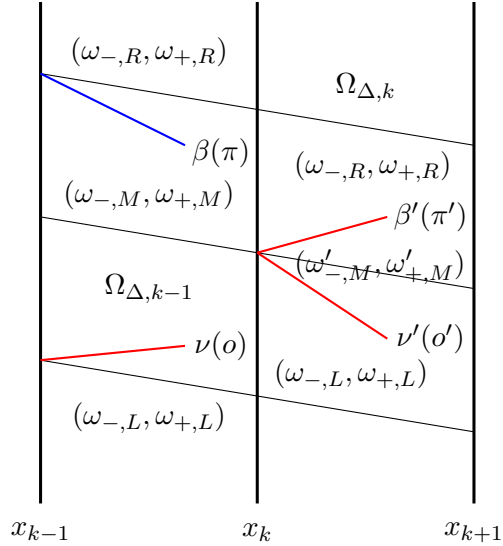

Now we are ready to introduce the local interaction estimates case by case. Let $\nu$ and $\nu'$ be the wave strength of shock wave $\mathcal{S}_1$ before and after the interaction.
Let $\beta$ and $\beta'$ be the wave strength of shock wave $\mathcal{S}_2$ before and after the interaction.
And let $o, \pi$ and $o', \pi'$ be the wave strength of rarefaction wave $\mathcal{R}_{1}$ and $\mathcal{R}_{2}$
before and after the interaction respectively.


\begin{lemma}\label{lem:3.5}
Let $\gamma\in[1,2]$, and let $0<\hat{\rho}<\check{\rho}<\infty$. Then, for $\rho\in [\hat{\rho}, \check{\rho}]$,
there exist positive constants $C_{0}>0$, $C_{5}>0$ and $\delta\in (0,1)$ independent of $\gamma$, $\beta$, $\nu$ and $\rho$,
such that the following interaction estimates hold:
\begin{itemize}
\item[(1)] For the case that $\mathcal{S}_{2}+\mathcal{S}_{1}\rightarrow \mathcal{S}'_{1}+\mathcal{S}'_{2}$,\ \emph{i.e.},
for the wave strength interaction that
$\beta+\nu\rightarrow \nu'+\beta'$, one of the following estimates holds:
\begin{eqnarray*}\label{eq:2.17}
\begin{split}
&(a)\  |\nu'|+|\beta'|\leq |\beta|+|\nu|+C_{5}(\gamma-1+\tau^{2})|\beta||\nu|,\\[5pt]
&(b)\ |\nu'|=|\nu|-\zeta, \quad |\beta'|\leq |\beta|+C_{5}(\gamma-1+\tau^{2})|\beta||\nu|+\eta,\\[5pt]
&(c)\ |\beta'|=|\beta|-\zeta, \quad |\nu'|\leq |\nu|+C_{5}(\gamma-1+\tau^{2})|\beta||\nu|+\eta,
\end{split}
\end{eqnarray*}
where $0\leq \eta \leq \delta\zeta$;

\item[(2)] For the case that $\mathcal{S}_{2}+\mathcal{R}_{1}\rightarrow \mathcal{R}'_{1}+\mathcal{S}'_{2}$,
\ \emph{i.e.}, for the wave interaction that $\beta+o \rightarrow o'+\beta'$, we have  $|\beta'|=|\beta|$;

\item[(3)] For the case that $\mathcal{S}_{2}+\mathcal{S}_{2}\rightarrow \mathcal{R}'_{1}+\mathcal{S}'_{2}$,
\ \emph{i.e.}, for the wave interaction that $\beta_{1}+\beta_{2}\rightarrow o'+\beta'$, we have
$|\beta'|=|\beta_{1}|+|\beta_{2}|$;

\item[(4)] For the case that $\mathcal{S}_{2}+\mathcal{R}_{2}\rightarrow \mathcal{S}'_{1}+\mathcal{S}'_{2}$
(or $\mathcal{S}_{2}+\mathcal{R}_{1}\rightarrow \mathcal{R}'_{1}+\mathcal{S}'_{2}$),\ \emph{i.e.},
for the wave interaction that $\beta+\pi\rightarrow \nu'+\beta'$ (or $\beta+o\rightarrow o'+\beta'$),
there exist $1$-shock wave $\nu_{0}$ and $2$-shock wave $\beta_{0}$ such that the wave interaction
$\beta_{0}+\nu_{0}\rightarrow \nu'+\beta'$ is the same as the one in \rm (1) and the following estimate hold:
\begin{eqnarray*}
\begin{split}
|\nu_{0}|+|\beta_{0}|\leq |\beta|-C_{0}|\nu_{0}|;
\end{split}
\end{eqnarray*}

\item[(5)] For the case that $\mathcal{R}_{2}+\mathcal{S}_{2}\rightarrow \mathcal{S}'_{1}+\mathcal{S}'_{2}$
(or $\mathcal{R}_{2}+\mathcal{S}_{1}\rightarrow \mathcal{S}'_{1}+\mathcal{R}'_{2}$), \ \emph{i.e.},
for the wave interaction that $0+\beta\rightarrow \nu'+\beta'$ (or $\pi+\nu\rightarrow \nu'+0'$),
we have  $|\nu'|+|\beta'|\leq |\beta|-C_{0}|\nu'|$;

\item[(6)] For the case that  $\mathcal{R}_{2}+\mathcal{R}_{1}\rightarrow \mathcal{R}'_{1}+\mathcal{R}'_{2}$,
\ \emph{i.e.}, for the wave interaction that $\pi+o \rightarrow o'+\pi'$, we have $|o|+|\pi|=|o'+|\pi'|$;

\item[(7)] For the case that $\mathcal{S}_{1}+\mathcal{R}_{1}\rightarrow \mathcal{S}'_{1}+\mathcal{S}'_{2}$,
\ \emph{i.e.}, for the wave interaction that $\nu+o\rightarrow \nu'+\beta'$, we have $|\nu'|+|\beta'|\leq |\nu|-C_{0}|\beta'|$;

\item[(8)] For the case that $\mathcal{S}_{1}+\mathcal{S}_{1}\rightarrow\mathcal{S}'_{1}+\mathcal{R}'_{2}$,
\ \emph{i.e.}, for the wave interaction that $\nu_{1}+\nu_{2}\rightarrow \nu'+\pi'$, we have $|\nu'|=|\nu_{1}|+|\nu_{2}|$.
\end{itemize}
\end{lemma}

\begin{proof}
We will show this Lemma case by case.

\par First, let us study the first case.
In this case, an $\mathcal{S}_{2}$ shock wave from the left with wave strength $\beta$ interacts with an $\mathcal{S}_{1}$
shock wave from the right with wave strength $\nu$. Both of them enter into $\Lambda$. Denote by $\nu'$ and $\beta'$
the wave strength of the resulting shock waves $\mathcal{S}'_{1}$ and $\mathcal{S}'_{2}$ issuing out from $\Lambda$
after the wave interaction.

\vspace{5pt}
\begin{figure}[ht]
\begin{center}
\begin{tikzpicture}[scale=1]
\draw [line width=0.05cm] (6,2.5)to[out=-90, in=30](5,0.5);
\draw [line width=0.05cm] (5,0.5)to[out=170, in=50](3.0,-0.2);
\draw [line width=0.05cm] (6,2.5)to[out=150, in=30](3.8,2.3);
\draw [line width=0.05cm] (3.8,2.3)to[out=-90, in=60](3,-0.2);
\draw [thin](6.1,2.5)--(6.6,2.5);
\draw [thin](5.3,0.5)--(6.6,0.5);
\draw [thin](3.6,2.3)--(2.2,2.3);
\draw [thin](2.9,-0.2)--(2.2,-0.2);
\draw [thick][<->](6.3,2.4)--(6.3,0.6);
\draw [thick][<->](2.7,2.2)--(2.7,-0.1);
\draw [thin](6.0,2.6)--(6.0,3.4);
\draw [thin](3.8,2.6)--(3.8,3.4);
\draw [thin](3.0,-0.3)--(3.0,-1.1);
\draw [thin](5.1,0.1)--(5.1,-1.1);
\draw [thick][<->](5.9,3.0)--(3.9,3.0);
\draw [thick][<->](3.1,-0.6)--(5.0,-0.6);

\node at (7.1, 2.7) {$(\omega_{-,L},\omega_{+,L})$};
\node at (2.6, 2.6) {$(\omega'_{-,M},\omega'_{+,M})$};
\node at (6.2, 0.2) {$(\omega_{-,M},\omega_{+,M})$};
\node at (1.8, -0.5) {$(\omega_{-,R},\omega_{+,R})$};

\node at (6.6, 1.6) {$|\beta|$};
\node at (2.3, 1.3) {$|\beta'|$};
\node at (5, 3.4) {$|\nu'|$};
\node at (4, -1) {$|\nu|$};

\node at (5.5, 1.5) {$\mathcal{S}_{2}$};
\node at (4.3, 0.2) {$\mathcal{S}_{1}$};
\node at (5.0, 2.4) {$\mathcal{S}'_{1}$};
\node at (3.4, 1.5) {$\mathcal{S}'_{2}$};
\end{tikzpicture}
\end{center}
\caption{Interactions between $\mathcal{S}_{2}$ and $\mathcal{S}_{1}$ waves}\label{fig3.3}
\end{figure}
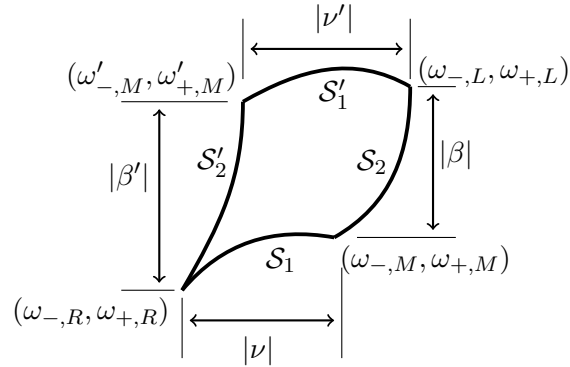
\vspace{5pt}

Let us consider the estimate in the $(\omega_{-},\omega_{+})$ plane.
Let $(\omega_{-,L}, \omega_{+,L}), (\omega_{-,M}, \omega_{+,M}), (\omega_{-, R}, \omega_{+,R})$ be the left, middle and right states before the wave interaction, \emph{i.e.}, $(\omega_{-,L}, \omega_{+,L})$ and $(\omega_{-,M}, \omega_{+,M})$ are connected by $\mathcal{S}_1$ shock, and $(\omega_{-, M}, \omega_{+,M})$ and $(\omega_{-, R}, \omega_{+,R})$ are connected by $\mathcal{S}_2$ shock. 
Let $(\omega'_{-,M}, \omega'_{+,M})$ be middle state after the wave interaction which is uniquely determined by the shock curves $\mathcal{S}'_{1}$ and $\mathcal{S}'_{2}$ which issue from $(\omega_{-,L}, \omega_{+,L})$ and $(\omega_{-,R}, \omega_{+,R})$ respectively
(See Fig. \ref{fig3.3}).

\par In order to derive the wave interaction estimate, 
as shown in Fig. \ref{fig3.4}, we consider the wave curves
$\hat{\mathcal{S}}_{1}$ and $\hat{\mathcal{S}}_{2}$ instead of the wave curves $\mathcal{S}_{1}'$ and $\mathcal{S}_{2}'$,
such that the wave curves $\hat{\mathcal{S}}_{1}$ and $\hat{\mathcal{S}}_{2}$,
issuing from $(\omega_{-,L}, \omega_{+,L})$ and $(\omega_{-,R}, \omega_{+,R})$ respectively, intersect at point
$(\hat{\omega}_{-,M}, \hat{\omega}_{+,M})$.
By Lemma \ref{lem:2.2} and Lemma \ref{lem:2.3}, we know that the wave curves $\hat{\mathcal{S}}_{1}$ and $\hat{\mathcal{S}}_{2}$,
the straight lines $\omega_{-}=\hat{\omega}_{-,M}$ and $\omega_{+}=\hat{\omega}_{+,M}$, and the wave curves $S_1'$ and $\mathcal{S}_2'$ together form the boundaries of subregions $I$, $II$ and $III$. Moreover, $(\omega_{-,M}', \omega_{+,M}')$ must lie in one of them.

\vspace{5pt}
\begin{figure}[ht]
\begin{center}
\begin{tikzpicture}[scale=1.3]
\draw [line width=0.05cm] (6,2.5)to[out=-130, in=90](5,0);
\draw [line width=0.05cm] (5,0)to[out=170, in=30](2.0,-0.5);
\draw [line width=0.05cm] (6,2.5)to[out=150, in=30](2.0,2.2);
\draw [line width=0.05cm] (2.0,2.2)to[out=-130, in=110](2,-0.5);
\draw [line width=0.03cm][dashed][red] (2.95,2.4)to[out=-120, in=100](2.0,-0.5);
\draw [line width=0.03cm][dashed][red] (6,2.5)to[out=170, in=40](2.0,1.35);
\draw [line width=0.05cm][red](6,0.5)to[out=170, in=30](2.6,0);
\draw [line width=0.04cm][blue] (0.7,1.8)--(2.6,1.8);
\draw [line width=0.04cm][blue] (2.6,3.0)--(2.6,-0.7);
\draw [line width=0.04cm][blue] (2.6,2.55)--(0.7,2.55);
\draw [line width=0.04cm][green] (6,2.5)--(6,0.5);
\draw [line width=0.04cm][blue] (6.5,0.75)--(6,0.5)--(5,0)--(4,-0.5);

\draw [thin](6.1,2.5)--(8,2.5);
\draw [thin](5.3,0)--(8,0);
\draw [thin](2.6,2.2)--(-0.7,2.2);
\draw [thin](2.0,-0.5)--(-0.7,-0.5);
\draw [thin](6.0,0.5)--(7.5,0.5);
\draw [thin](5.0,-0.1)--(5.0,-1.6);
\draw [thin](2.0,-0.6)--(2.0,-1.6);
\draw [thin](6.0,2.5)--(6.0,3.8);
\draw [thin](2.0,2.3)--(2.0,3.8);

\draw [thick][<->](7.8,2.4)--(7.8,0.1);
\draw [thick][<->](-0.1,2.1)--(-0.1,-0.4);
\draw [thick][<->](1,1.7)--(1,-0.4);
\draw [thick][<->](1.5,2.5)--(1.5,1.9);
\draw[thick][<->](6.9,2.4)--(6.9,0.6);
\draw[thick][<->](2.7,1.8)--(5.9,1.8);
\draw[thick][<->](2.1,-1.2)--(4.9,-1.2);
\draw[thick][<->](2.1,3.3)--(5.9,3.3);

\node at (6.9, 2.7) {$(\omega_{-,L},\omega_{+,L})$};
\node at (6.9, 0.9){$(\omega_{-,L}, \omega^{1}_{+,M})$};
\node at (1.0, 2.8) {$(\omega'_{-,M},\omega'_{+,M})$};
\node at (6.0, -0.2) {$(\omega_{-,M},\omega_{+,M})$};
\node at (1.1, -0.8) {$(\omega_{-, R},\omega_{+,R})$};
\node at (3.5, 1.6){$(\hat{\omega}_{-,M}, \hat{\omega}_{+,M})$};
\node at (3.0, 3.0){$(\hat{\omega}_{-,M}, \hat{\omega}^{2}_{+,M})$};
\node at (3.4, 0.7){$(\hat{\omega}_{-,M}, \hat{\omega}^{1}_{+,M})$};
\node at (6.1, 0.2){$v=v_{M}$};

\node at (8.2, 1.4) {$|\beta|$};
\node at (-0.4, 0.8) {$|\beta'|$};
\node at (0.7, 0.8) {$|\beta|$};
\node at (1.1, 2.2) {$|\beta''|$};
\node at (7.2, 1.5) {$|\beta_{0}|$};
\node at (4, 1.9) {$|\nu|$};
\node at (3.5, -1.0) {$|\nu|$};
\node at (3.8, 3.6) {$|\nu'|$};

\node at (5.7, 1.5) {$\mathcal{S}_{2}$};
\node at (3.4, -0.3) {$\mathcal{S}_{1}$};
\node at (4.3, 2.3) {$\hat{\mathcal{S}}_{1}$};
\node at (3.3, 2.6) {$\mathcal{S}'_{1}$};
\node at (2.3, 0.7) {$\hat{\mathcal{S}}_{2}$};
\node at (1.4, 0.5) {$\mathcal{S}'_{2}$};
\node at (2.3, 2.0) {$I$};
\node at (2.75, 2.3) {$II$};
\node at (2.1, 1.6) {$III$};

\end{tikzpicture}
\end{center}
\caption{}\label{fig3.4}
\end{figure}
\vspace{5pt}

We first consider the case that
$(\omega'_{-,M},\omega'_{+,M})$ lies in the region $I$, \emph{i.e.}, $\omega'_{-,M}<\hat{\omega}_{-,M}$ and  $\omega'_{+,M}>\hat{\omega}_{+,M}$ (See Fig. \ref{fig3.4}).
In this case, we know that
\begin{eqnarray}\label{eq:3.11}
|\beta'|-|\beta|\leq |\beta''|,\ \ \ \omega_{+,L}-\omega^{1}_{+,L}=\hat{\omega}_{+,M}-\hat{\omega}^{1}_{+,M}=|\beta_{0}|.
\end{eqnarray}

Notice that
\begin{eqnarray}\label{eq:3.12}
\begin{split}
\omega^{1}_{+,M}-\hat{\omega}^{1}_{+,M}-(\omega_{+,L}-\hat{\omega}^{2}_{+,M})
&=\hat{\omega}^{2}_{+,M}-\hat{\omega}^{1}_{+,M}-(\omega_{+,L}-\omega^{1}_{+,M})\\[5pt]
&=\hat{\omega}^{2}_{+,M}-\hat{\omega}^{1}_{+,M}-(\hat{\omega}_{+,M}-\hat{\omega}^{1}_{+,M})\\[5pt]
&=|\beta''|.
\end{split}
\end{eqnarray}

By Lemma \ref{lem:2.4}, there exists a constant $C_{5}>0$ such that
\begin{eqnarray}\label{eq:3.13}
\begin{split}
\omega^{1}_{+,M}-\hat{\omega}^{1}_{+,M}-(\omega_{+,L}-\hat{\omega}^{2}_{+,M})&\leq C_{5}(\gamma-1+\tau^{2})|\nu||\beta_{0}|\\[5pt]
&\leq C_{5}(\gamma-1+\tau^{2})|\nu||\beta|.
\end{split}
\end{eqnarray}

Then combing \eqref{eq:3.11}-\eqref{eq:3.13} together, we have that
\begin{eqnarray*}
\begin{split}
|\beta'|-|\beta|\leq C_{5}(\gamma-1+\tau^{2})|\nu||\beta|.
\end{split}
\end{eqnarray*}

By the same way and by Lemma \ref{lem:2.5}, one can also show that
$$
|\nu'|-|\nu|\leq C_{5}(\gamma-1+\tau^{2})|\nu||\beta|.
$$
Therefore, we show estimate $(a)$ for the first case in Lemma \ref{lem:3.5}.

\vspace{5pt}
\begin{figure}[ht]
\begin{center}
\begin{tikzpicture}[scale=1.3]
\draw [line width=0.05cm] (6,2.5)to[out=-130, in=90](5,0);
\draw [line width=0.05cm] (5,0)to[out=170, in=30](2.0,-0.5);
\draw [line width=0.05cm] (6,2.5)to[out=150, in=30](2.5,2.25);
\draw [line width=0.05cm] (2.7,2.4)to[out=-140, in=120](2,-0.5);
\draw [line width=0.04cm][dashed][red] (2.3,1.6)to[out=-110, in=110](2.0,-0.5);
\draw [line width=0.04cm][dashed][red] (6,2.5)to[out=170, in=50](2.0,1.2);
\draw [line width=0.05cm][red](6,0.5)to[out=170, in=30](2.7,0);
\draw [line width=0.04cm][blue] (-0.7,2.35)--(3,2.35);
\draw [line width=0.04cm][blue] (0.7,1.9)--(3,1.9);
\draw [line width=0.04cm][blue] (0.7,1.45)--(3,1.45);
\draw [line width=0.04cm][green] (2.7,3.7)--(2.7,-1.0);
\draw [line width=0.04cm][green] (2.25,3.0)--(2.25,-1.8);
\draw [line width=0.04cm][green] (6,2.5)--(6,-1.8);
\draw [line width=0.04cm][blue] (6.5,0.75)--(6,0.5)--(5,0)--(4,-0.5);

\draw [thin](6.1,2.5)--(8,2.5);
\draw [thin](5.3,0)--(8,0);
\draw [thin](2.0,-0.5)--(-0.7,-0.5);
\draw [thin](6.0,0.5)--(7.5,0.5);
\draw [thin](6.0,2.5)--(6.0,3.8);

\draw [thick][<->](7.8,2.4)--(7.8,0.1);
\draw [thick][<->](-0.1,2.2)--(-0.1,-0.4);
\draw [thick][<->](1,1.42)--(1,-0.4);
\draw [thick][<->](1.5,1.85)--(1.5,1.5);
\draw[thick][<->](6.9,2.4)--(6.9,0.6);
\draw[thick][<->](2.8,1.8)--(5.9,1.8);
\draw[thick][<->](2.8,-0.7)--(5.9,-0.7);
\draw[thick][<->](2.8,3.3)--(5.9,3.3);
\draw[thick][<->](2.3,-1.5)--(5.9,-1.5);
\draw[thick][<->](2.3,-0.8)--(2.7,-0.8);

\node at (6.9, 2.7) {$(\omega_{-,L},\omega_{+,L})$};
\node at (6.9, 0.8){$(\omega_{-,L}, \omega^{1}_{+,M})$};
\node at (1.8, 2.6) {$(\omega'_{-,M},\omega'_{+,M})$};
\node at (5.6, -0.3) {$(\omega_{-,M},\omega_{+,M})$};
\node at (1.4, -0.8) {$(\omega_{-,R},\omega_{+,R})$};
\node at (3.2, 1.2){$(\hat{\omega}_{-,M}, \hat{\omega}_{-,M})$};
\node at (3.6, 1.6){$(\omega'_{-,M}, \omega'^{2}_{+,M})$};
\node at (3.6, 0.4){$(\omega'_{-,M}, \omega'^{1}_{+,M})$};
\node at (6.1, 0.2){$v=v_{M}$};

\node at (8.2, 1.4) {$|\beta|$};
\node at (-0.4, 0.8) {$|\beta'|$};
\node at (0.7, 0.8) {$|\beta|$};
\node at (1.3, 1.7) {$\eta$};
\node at (7.2, 1.5) {$|\beta_{0}|$};
\node at (4.8, 2.0) {$|\nu'|$};
\node at (4, -0.9) {$|\nu'|$};
\node at (3.8, 3.6) {$|\nu'|$};
\node at (3.7, -1.8) {$|\nu|$};
\node at (2.5, -0.6) {$\zeta$};

\node at (5.4, 1.2) {$\mathcal{S}_{2}$};
\node at (3.4, -0.2) {$\mathcal{S}_{1}$};
\node at (4.3, 2.3) {$\hat{\mathcal{S}}_{1}$};
\node at (4.2, 3.1) {$\mathcal{S}'_{1}$};
\node at (2.2, 0.5) {$\hat{\mathcal{S}}_{2}$};
\node at (1.5, 0.7) {$\mathcal{S}'_{2}$};

\end{tikzpicture}
\end{center}
\caption{}\label{fig3.5}
\end{figure}
\vspace{5pt}

\par Next, let us consider the case that $(\omega'_{-,M},\omega'_{+,M})\in II$.
As shown in Fig. \ref{fig3.5}, we can see that
\begin{eqnarray*}
\begin{split}
|\beta'|=|\beta|-\zeta>0,
\end{split}
\end{eqnarray*}
and
\begin{eqnarray*}
\begin{split}
|\nu'|-|\nu|-\eta=\omega'_{+,M}-\omega'^{2}_{+,M}.
\end{split}
\end{eqnarray*}

For the estimate of $\omega'_{+,M}-\omega'^{2}_{+,M}$, by Lemma \ref{lem:2.4}, there exists a constant $C_{5}>0$ such that
\begin{eqnarray*}
\begin{split}
\omega'_{+,M}-\omega'^{2}_{+,M}&=\omega'_{+,M}-\omega'^{1}_{+,M}-(\omega'^{2}_{+,M}-\omega'^{1}_{+,M})\\[5pt]
&=\omega^{1}_{+,M}-\omega'^{1}_{+,M}-(\omega_{+,M}-\omega'_{+,M})\\[5pt]
&\leq  C_{5}(\gamma-1+\tau^{2})|\nu'||\beta_{0}|\\[5pt]
&\leq C_{5}(\gamma-1+\tau^{2})|\nu||\beta|.
\end{split}
\end{eqnarray*}

For the estimate of $\eta$, by Lemma \ref{lem:2.2}, we have
\begin{eqnarray*}
\begin{split}
\eta=\Phi_{1}(|\nu|, U_{L}; \tau^{2})-\Phi_{1}(|\nu|-\zeta, U_{L}; \tau^{2})=\Phi'_{1}(\theta,U_{L}; \tau^{2})\zeta,
\ \ \ \mbox{where } \theta \in(|\nu|-\zeta, |\nu|),
\end{split}
\end{eqnarray*}
which implies that $\eta\leq \delta \zeta$ by taking
\begin{eqnarray*}
\begin{split}
\delta=\sup_{\theta \in(|\nu|-\zeta, |\nu|), \rho_{L}\in (\hat{\rho},\check{\rho})} \Phi'_{1}(\theta,U_{L}; \tau^{2})\in (0,1).
\end{split}
\end{eqnarray*}

Therefore, combing the estimates above together, we can get estimate $(b)$  in Lemma \ref{lem:3.5} for the second case. 

Finally, by a similar argument as the one for the second case, for the case that $(\omega'_{-,M}, \omega'_{+,M}) \in III$,
we can obtain the estimate $(c)$ in Lemma \ref{lem:3.5}. 

It completes the proof of case $(1)$.

\smallskip
\par Now, let us study case \rm (2).

Similar to case $(1)$, let $(\omega_{-,L}, \omega_{+,L}), (\omega_{-,M}, \omega_{+,M}), (\omega_{-,R}, \omega_{+,R})$ be the left, middle and right states before the wave interaction and let $(\omega'_{-,M}, \omega'_{+,M})$ be the middle state after the wave interaction which is uniquely determined by the rarefaction wave $\mathcal{R}'_{1}$ and the shock wave $\mathcal{S}'_{2}$. Notice that $\omega_{+,L}=\omega'_{+,M}$ and $\omega_{+,M}=\omega_{+,R}$, then by the monotonicity of function $\Phi_{2}$, we have that
\begin{eqnarray*}
\begin{split}
|\beta'|=|\omega_{+,R}-\omega'_{+,M}|=|\omega_{+,M}-\omega_{+,L}|=|\beta|.
\end{split}
\end{eqnarray*}

\smallskip
\par The proof of the estimates for case \rm (3) is similar to the one for case $(2)$. In fact, 
by the monotonicity of function $\Phi_{2}$,
we have $\omega_{+,L}<\omega_{+,M}<\omega_{+,R}$ and $\omega'_{+,M}=\omega_{+,L}$. Then
\begin{eqnarray*}
\begin{split}
|\beta'|=|\omega_{+,R}-\omega'_{+,M}|=|\omega_{+,R}-\omega_{+,L}|=|\omega_{+,R}-\omega_{+,M}|+|\omega_{+,M}-\omega_{+,L}|
=|\beta_{1}|+|\beta_{2}|.
\end{split}
\end{eqnarray*}

\vspace{5pt}
\begin{figure}[ht]
\begin{center}
\begin{tikzpicture}[scale=1.1]
\draw [line width=0.05cm] (6,4.5)to[out=-90, in=10](3.5,0);
\draw [line width=0.05cm] (3.5,0)--(3.5,1.5);
\draw [line width=0.05cm][red] (5.7,2.2)to[out=170, in=40](3.5,1.5);
\draw [line width=0.05cm] (6,4.5)to[out=150, in=50](4.2,4.0);
\draw [line width=0.05cm] (4.2,4.0)to[out=-90, in=60](3.5,1.5);

\draw [thin][dashed](6.1,4.5)--(8,4.5);
\draw [thin][dashed](5.8,2.2)--(7,2.2);
\draw [thin][dashed](3.6,0)--(8,0);
\draw [thin](6.0,4.5)--(6.0,-0.9);
\draw [thin](5.7,2.2)--(5.7,-0.4);
\draw [thin](3.5,0)--(3.5,-0.9);
\draw [thin](4.1,4.0)--(2.0,4.0);
\draw [thin](3.4,1.5)--(2.0,1.5);
\draw [thin](6,4.6)--(6,5.3);
\draw [thin](4.2,4.1)--(4.2,5.3);

\draw [thick][<->](7.8,4.4)--(7.8,0.1);
\draw [thick][<->](6.7,4.4)--(6.7,2.3);
\draw [thick][dashed][<->](3.5,0)--(5.7,0);
\draw [thick][<->](3.5,-0.6)--(6.0,-0.6);
\draw [thick][<->](2.7,3.9)--(2.7,1.6);
\draw [thick][<->](4.3,4.9)--(5.9,4.9);

\node at (7.0, 4.7) {$(\omega_{-,L},\omega_{+,L})$};
\node at (3.1, 4.3) {$(\omega'_{-,M},\omega'_{+,M})$};
\node at (2.3, 0) {$(\omega_{-,M},\omega_{+,M})$};
\node at (6.8, 2.0) {$(\hat{\omega}_{-,M},\hat{\omega}_{+,M})$};
\node at (4.5, 1.4) {$(\omega_{-,R},\omega_{+,R})$};

\node at (8.2, 2.6) {$|\beta|$};
\node at (7.1, 3.2) {$|\beta_{0}|$};
\node at (5, 0.3) {$|\nu_{0}|$};
\node at (5, -0.4) {$|\nu|$};
\node at (2.3, 2.7) {$|\beta'|$};
\node at (5, 5.2) {$|\nu'|$};

\node at (4.3, 0.7) {$\mathcal{S}_{2}$};
\node at (3.2, 0.8) {$\mathcal{R}_{2}$};
\node at (5.0, 2.5) {$\hat{\mathcal{S}}_{1}$};
\node at (5.0, 4.3) {$\mathcal{S}'_{1}$};
\node at (3.8, 2.8) {$\mathcal{S}'_{2}$};
\end{tikzpicture}
\end{center}
\caption{Interactions between $\mathcal{S}_{2}$ and $\mathcal{R}_{2}$ waves}\label{fig3.6}
\end{figure}
\vspace{5pt}

\smallskip
\par Next, let us consider case\ \rm{(4)}.
As shown in Fig. \ref{fig3.6}, we can find a shock wave $\hat{\mathcal{S}}_{1}$ such that $\mathcal{S}_{2}+\hat{\mathcal{S}}_{1}\rightarrow \mathcal{S}'_{1}+\mathcal{S}'_{2}$, and then one can follow the argument for the proof of case $(1)$ exactly to have that
\begin{eqnarray*}
\begin{split}
|\nu'|\leq |\nu_{0}|+C_{5}(\gamma-1+\tau^{2})|\nu_{0}||\beta_{0}|,\ \ \
|\beta'|\leq |\beta_{0}|+C_{5}(\gamma-1+\tau^{2})|\nu_{0}||\beta_{0}|.
\end{split}
\end{eqnarray*}

 Now, we will consider the estimate between $\nu$ and $\nu_{0}$, and the estimate between $\beta$ and $\beta_{0}$. By Lemma \ref{lem:2.3},
 we can have that
\begin{eqnarray*}
\begin{split}
|\nu_{0}|&=\omega_{-,L}-\omega_{-,M}-(\omega_{-,L}-\hat{\omega}_{-,M})\\[5pt]
&=\Phi_{2}(|\beta|, U_{L};\tau^{2})-\Phi_{2}(|\beta_{0}|, U_{L};\tau^{2})\\[5pt]
&=\Phi'_{2}(\theta, U_{L}; \tau^{2})(\beta-\beta_{0}),\ \ \  \theta \in(|\beta_{0}|, |\beta|),
\end{split}
\end{eqnarray*}
 which implies that
\begin{eqnarray*}
\begin{split}
|\nu_{0}|+|\beta_{0}|&=|\beta|-\big(\frac{1}{\Phi'_{2}(\theta,U_{L};\tau^{2})}-1\big)|\nu_{0}|\\[5pt]
&\leq |\beta|-\Big(\frac{1}{\Phi'_{2}(\theta,U_{L};\tau^{2})}-1\Big)|\nu_{0}|.
\end{split}
\end{eqnarray*}

Notice that $\lim_{|\beta|\rightarrow +\infty}\Phi'_{2}(|\beta|,U_{L}; \tau^{2})=1$, then we get that
 \begin{eqnarray*}
\begin{split}
\lim_{|\beta|\rightarrow +\infty}\Big(\frac{1}{\Phi'_{2}(\theta,U_{L};\tau^{2})}-1\Big)=0,
\end{split}
\end{eqnarray*}
which implies that
\begin{eqnarray*}
\begin{split}
C_{0}:=\inf_{\theta\in \{(\omega_{-},\omega_{+}): \ 0<\hat{\rho}<\rho<\check{\rho}\}, 0<\hat{\rho}<\rho_{L}<\check{\rho}}\Big(\frac{1}{\Phi'_{2}(\theta,U_{L};\tau^{2})}-1\Big)>0.
\end{split}
\end{eqnarray*}

\smallskip
\par Now we continue to study case \rm(5), that is the wave interaction between $\mathcal{R}_{2}$ and $\mathcal{S}_{2}$.
Let $(\omega_{-,L}, \omega_{+,L})$, $(\omega_{-,M}, \omega_{+,M})$, $(\omega_{-,R}, \omega_{+,R})$,
and $(\omega'_{-,M}, \omega'_{+,M})$ be defined similarly as before.
Then
\begin{eqnarray*}
\begin{split}
|\nu'|&=\omega'_{-,M}-\omega_{-,R}-(\omega_{-,M}-\omega_{-,R})\\[5pt]
&=\Phi_{2}(-|\beta'|, U_{R}; \tau^{2})-\Phi_{2}(-|\beta|, U_{R}; \tau^{2})\\[5pt]
&=\Phi'_{2}(\theta, U_{R}; \tau^{2})(|\beta|-|\beta'|),\ \ \  \theta \in(-|\beta|, -|\beta'|).
\end{split}
\end{eqnarray*}
So
\begin{eqnarray*}
\begin{split}
|\nu'|+|\beta'|&=|\beta|-\Big(\frac{1}{\Phi'_{2}(\theta,U_{L};\tau^{2})}-1\Big)|\nu'|.
\end{split}
\end{eqnarray*}

Based on the proof for case $(4)$, we know that
\begin{eqnarray*}
\begin{split}
C_{0}:=\inf_{\theta\in \{(\omega_{-},\omega_{+}):  0<\hat{\rho}<\rho<\check{\rho}\}, 0<\hat{\rho}<\rho_{L}<\check{\rho}}\Big(\frac{1}{\Phi'_{2}(\theta,U_{L};\tau^{2})}-1\Big)>0.
\end{split}
\end{eqnarray*}

\smallskip
\par The estimate in case \rm (6) is obviously.

\smallskip
\par Now, we will prove the estimate for case \rm{(7)}. Similarly, 
Let $(\omega_{-,L}, \omega_{+,L})$, $(\omega_{-,M}, \omega_{+,M})$ and $(\omega_{-,R}, \omega_{+,R})$ be the three states
before the wave interaction, and let $(\omega'_{-,M}, \omega'_{+,M})$ be the middle state after the wave interaction.
Then
\begin{eqnarray*}
\begin{split}
|\beta'|&=\omega'_{+,M}-\omega_{+,R}=\omega_{+,L}-\omega_{+,M}-(\omega_{+,L}-\omega'_{+,M})\\[5pt]
&=\Phi_{1}(|\nu|, U_{L};\tau^{2})-\Phi_{1}(|\nu'|, U_{L}; \tau^{2})\\[5pt]
&=\Phi'_{1}(\theta, U_{L}; \tau^{2})(|\nu|-|\nu'|),\ \ \  \theta \in(|\nu'|, |\nu|).
\end{split}
\end{eqnarray*}
So
\begin{eqnarray*}
\begin{split}
|\nu'|+|\beta'|&=|\nu|-\Big(\frac{1}{\Phi'_{1}(\theta, U_{L}; \tau^{2})}-1\Big)|\nu'|.
\end{split}
\end{eqnarray*}
Again, we know that
\begin{eqnarray*}
\begin{split}
C_{0}:=\inf_{\theta\in \{(\omega_{-},\omega_{+}): 0<\hat{\rho}<\rho<\check{\rho}\}, 0<\hat{\rho}<\rho_{L}<\check{\rho}}\Big(\frac{1}{\Phi'_{1}(\theta,U_{L};\tau^{2})}-1\Big)>0.
\end{split}
\end{eqnarray*}

\smallskip
Finally, the proof of the estimate for case \rm{(8)} is exactly the same as the one for case \rmfamily{(3)}.

This completes the proof of this lemma.
\end{proof}

Next, let us consider the interaction estimates near the boundary. First, we study the case that $\mathcal{S}_{1}$ wave
hit the boundary and then reflects (see Fig. \ref{fig3.1xw}).
\begin{figure}[ht]
\begin{center}
\begin{tikzpicture}[scale=0.75]
\draw [line width=0.03cm] (-3.5,-4) --(-3.5,1.5);
\draw [line width=0.03cm] (-0.5,-4) --(-0.5,1.0);
\draw [line width=0.03cm] (2.5,-4) --(2.5,0.5);
\draw [thin] (-3.5,1.5)--(2.5,0.5);
\draw [thin] (-3.5,-1.5)--(2.5,-2.5);

\draw [thin] (-3, 1.4) --(-2.6, 1.8);
\draw [thin] (-2.6, 1.35) --(-2.2, 1.75);
\draw [thin] (-2.2, 1.30) --(-1.8, 1.70);
\draw [thin] (-1.8, 1.23) --(-1.4, 1.63);
\draw [thin] (-1.4, 1.16) --(-1.0, 1.56);
\draw [thin] (-1.0, 1.10) --(-0.6, 1.50);
\draw [thin] (-0.6, 1.03) --(-0.2, 1.43);
\draw [thin] (-0.2, 0.97) --(0.2, 1.37);
\draw [thin] (0.2, 0.9) --(0.6, 1.30);
\draw [thin] (0.6, 0.83) --(1, 1.23);
\draw [thin] (1, 0.76) --(1.4, 1.16);
\draw [thin] (1.4, 0.67) --(1.8, 1.07);
\draw [thin] (1.8, 0.60) --(2.2, 1.0);
\draw [thin] (2.2, 0.55) --(2.6, 0.95);

\draw [thick][blue](-3.5,-1.5)--(-1.2,-0.2);
\draw [thick][red](-0.5,1)--(1.7,-0.5);

\draw [thick][<-] (-1.8,2.7)--(-2.5,2);
\draw [thick][<-] (1.5,2.2)--(0.6,1.5);

\node at (-1.8, 3.0){$\mathbf{n}_{k-1}$};
\node at (1.8, 2.4){$\mathbf{n}_{k}$};

\node at (-1.8, -3.2){$\Omega_{\Delta, k-1}$};
\node at (0.8, -3.2){$\Omega_{\Delta, k}$};

\node at (-2.4, -0.2){$\mathcal{S}_{1}$};
\node at (0.5, -0.4){$\mathcal{S}'_{2}$};

\node at (-0.9, 0){$\nu$};
\node at (2.0, -0.7){$\beta'$};

\node at (-1.9, -1.4){$(\rho_{L}, v_{L})$};
\node at (-2.1 , 0.6){$(\rho_{R}, v_{R})$};

\node at (0.9, -1.4){$(\rho_{L}, v_{L})$};
\node at (1.5 , 0.3){$(\rho'_{R}, v'_{R})$};

\node at (-3.5, -4.5){$x_{k-1}$};
\node at (-0.5, -4.5){$x_{k}$};
\node at (2.5, -4.5){$x_{k+1}$};
\end{tikzpicture}
\end{center}
\caption{$\mathcal{S}_{1}$ wave hits the boundary and $\mathcal{S}'_{2}$ wave reflects}\label{fig3.1xw}
\end{figure}
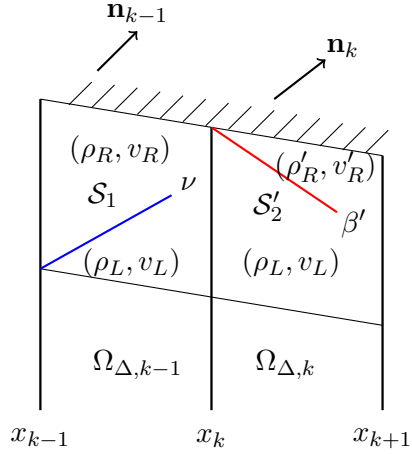

\begin{lemma}\label{lem:3.6}
Let $\gamma\in[1,2]$, $0<\hat{\rho}<\check{\rho}<\infty$ and $b_{0}<0$. Suppose that the constant states $U_{L}, U_{R}\in \mathcal{O}(U_{\infty})$ with $\rho_{L}, \rho_{R}\in [\hat{\rho}, \check{\rho}]$, satisfies that
\begin{eqnarray}\label{eq:3.22}
v_{R}=(1+\tau^{2}u_{R})b_{0}, \qquad \omega_{+, R}=-\Phi_{1}(\nu, U_{L}, \tau^{2})+\omega_{+, L}.
\end{eqnarray}
Then, for constant state $U'_{R}\in \mathcal{O}(U_{\infty})$ with $\rho'_{R}\in [\hat{\rho}, \check{\rho}]$
which satisfies that
\begin{eqnarray}\label{eq:3.23}
v'_{R}=(1+\tau^{2}u'_{R})b_{0}, \qquad \omega'_{-, R}=-\Phi_{2}(\beta', U_{L}, \tau^{2})+\omega_{-, L},
\end{eqnarray}
we have
\begin{eqnarray}\label{eq:3.24}
\beta'=K_{b}\nu,
\end{eqnarray}
where
\begin{eqnarray}\label{eq:3.25}
K_{b}=-1+\mathcal{O}(1)(\gamma-1+\tau^{2}),
\end{eqnarray}
with the bound $\mathcal{O}(1)$ depending only on the system and $U_{L}$.
\end{lemma}

\begin{proof}
Denote
\begin{eqnarray}\label{eq:3.23x}
\mathscr{L}_{0}(\beta', \nu, \gamma-1, \tau^{2}):=(1+\tau^{2}u_{R})v'_{R}-(1+\tau^{2}u'_{R})v_{R}.
\end{eqnarray}

When $\gamma=1$ and $\tau=0$, \eqref{eq:3.23x} is reduced to
\begin{eqnarray*}\label{eq:3.23a}
\mathscr{L}_{0}(\beta', \nu, \gamma-1, \tau^{2})\Big|_{\gamma=1, \tau=0}
=\frac{1}{2a_{\infty}}\Big(\beta'+\nu+g(-\beta')-g(\nu)\Big),
\end{eqnarray*}
where $g(\nu):=\Phi_{1}(\nu, U_{L};\tau^{2})\Big|_{\gamma=1,\tau=0}$ with $0<g'<1, g''>0$, and $\Phi_{2}(\beta', U_{L};\tau^{2})\Big|_{\gamma=1,\tau=0}=-g(-\beta')$.
In this case, equation $\mathscr{L}_{0}(\beta', \nu, \gamma-1, \tau^{2})\Big|_{\gamma=1, \tau=0}=0$ admits a unique solution $\beta'=-\nu$.
Note that
\begin{eqnarray*}
\frac{\partial\mathscr{L}_{0}(\beta', \nu, 0, 0)}{\partial \beta'}\Big|_{\gamma=1,\tau=0, \beta'=-\nu}=\frac{1}{2a_{\infty}}\Big(1-g'(-\nu)\Big)>C>0,
\end{eqnarray*}
where constant $C$ depends only on $\hat{\rho}$ and $\check{\rho}$.
So it follows from the implicit function theorem that $\beta'$ can be solved as a $C^{2}$ function of $\nu$, $\gamma-1$, $\tau^{2}$, $b_{0}$
and $U_{L}$.
Moreover,
\begin{eqnarray*}
\beta'=\beta'(\nu, \gamma-1, \tau^{2})
=\nu\int^{1}_{0}\partial_{\nu}\beta'(\mu\nu, \gamma-1, \tau^{2})d\mu,
\end{eqnarray*}
where we have used the fact that $\beta'(0, \gamma-1, \tau^{2})=0$.

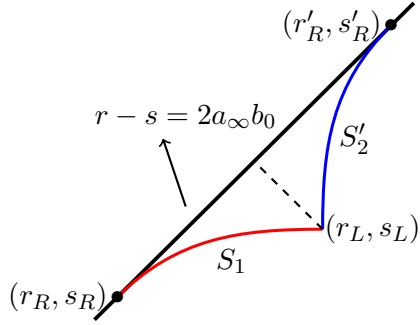
\begin{figure}[ht]
\begin{center}
\begin{tikzpicture}[scale=0.6]
\draw [line width=0.05cm] (-6.5,-4.5) --(0.5,2.5);
\draw [line width=0.04cm][red](-1.5,-2.5)to[out=180, in=45](-6,-4);
\draw [line width=0.04cm][blue](-1.5,-2.5)to[out=90, in=-135](0,2);
\draw [thick][dashed] (-1.5,-2.5)--(-3,-1);

\draw [thick][->] (-4.5,-2)--(-5,-0.5);

\node at (-6,-4){$\bullet$};
\node at (0,2){$\bullet$};

\node at (-7.3,-4){$(r_{R}, s_{R})$};
\node at (-1.3,2){$(r'_{R}, s'_{R})$};

\node at (-0.4, -2.5){$(r_{L}, s_{L})$};
\node at (-3.5, -3.2){$S_{1}$};
\node at (-0.8, -0.5){$S'_{2}$};
\node at (-4.5, 0){$r-s=2a_{\infty}b_{0}$};
\end{tikzpicture}
\end{center}
\caption{$S_{1}$ wave hits on the boundary and $S'_{2}$ wave reflects}\label{fig3.1x2}
\end{figure}

Since $\beta'(\nu, 0, 0)=-\nu$, then $\partial_{\nu}\beta'(\nu, 0, 0)=-1$, which gives that
\begin{eqnarray*}
\begin{split}
\beta'&=\int^{1}_{0}\Big(\partial_{\nu}\beta'(\mu\nu, \gamma-1, \tau^{2})
-\partial_{\nu}\beta'(\mu\nu, 0, 0)\Big)d\mu\nu-\nu
=\big(-1+\mathcal{O}(1)(\gamma-1+\tau^{2})\big)\nu.
\end{split}
\end{eqnarray*}

So by taking $K_{b}=-1+\mathcal{O}(1)(\gamma-1+\tau^{2})$, we have equality \eqref{eq:3.24}.
\end{proof}

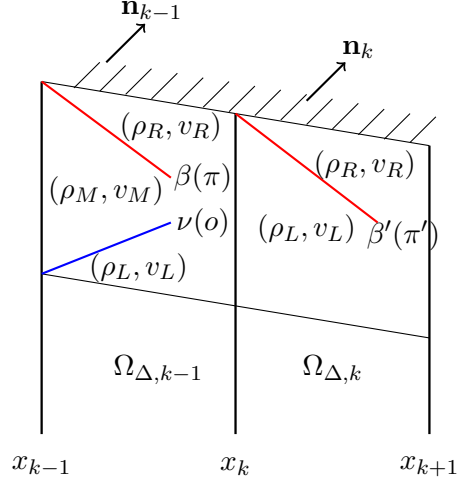
\begin{figure}[ht]
\begin{center}
\begin{tikzpicture}[scale=0.85]
\draw [line width=0.03cm] (-3.5,-4) --(-3.5,1.5);
\draw [line width=0.03cm] (-0.5,-4) --(-0.5,1.0);
\draw [line width=0.03cm] (2.5,-4) --(2.5,0.5);
\draw [thin] (-3.5,1.5)--(2.5,0.5);
\draw [thin] (-3.5,-1.5)--(2.5,-2.5);

\draw [thin] (-3, 1.4) --(-2.6, 1.8);
\draw [thin] (-2.6, 1.35) --(-2.2, 1.75);
\draw [thin] (-2.2, 1.30) --(-1.8, 1.70);
\draw [thin] (-1.8, 1.23) --(-1.4, 1.63);
\draw [thin] (-1.4, 1.16) --(-1.0, 1.56);
\draw [thin] (-1.0, 1.10) --(-0.6, 1.50);
\draw [thin] (-0.6, 1.03) --(-0.2, 1.43);
\draw [thin] (-0.2, 0.97) --(0.2, 1.37);
\draw [thin] (0.2, 0.9) --(0.6, 1.30);
\draw [thin] (0.6, 0.83) --(1, 1.23);
\draw [thin] (1, 0.76) --(1.4, 1.16);
\draw [thin] (1.4, 0.67) --(1.8, 1.07);
\draw [thin] (1.8, 0.60) --(2.2, 1.0);
\draw [thin] (2.2, 0.55) --(2.6, 0.95);

\draw [thick][red] (-3.5,1.5)--(-1.5,0);
\draw [thick][blue](-3.5,-1.5)--(-1.5,-0.7);
\draw [thick][red](-0.5,1)--(1.7,-0.7);

\draw [thick][<-](-1.9,2.4)--(-2.5,1.8);
\draw [thick][<-] (1.2,1.8)--(0.6,1.2);

\node at (-1.8, 2.6){$\mathbf{n}_{k-1}$};
\node at (1.4, 2.0){$\mathbf{n}_{k}$};

\node at (-1.7, -3.0){$\Omega_{\Delta, k-1}$};
\node at (1.0, -3.0){$\Omega_{\Delta, k}$};

\node at (-2.0, -1.4){$(\rho_{L}, v_{L})$};
\node at (-2.5, -0.2){$(\rho_{M}, v_{M})$};
\node at (-1.5, 0.8){$(\rho_{R}, v_{R})$};

\node at (0.6, -0.8){$(\rho_{L}, v_{L})$};
\node at (1.5, 0.2){$(\rho_{R}, v_{R})$};

\node at (-1.0, -0.7){$\nu(o)$};
\node at (-1.0, 0){$\beta(\pi)$};
\node at (2.1, -0.9){$\beta'(\pi')$};

\node at (-3.5, -4.5){$x_{k-1}$};
\node at (-0.5, -4.5){$x_{k}$};
\node at (2.5, -4.5){$x_{k+1}$};
\end{tikzpicture}
\end{center}
\caption{Local interaction estimates near the boundary}\label{fig3.1}
\end{figure}

Now, let us consider the local interaction estimates near the boundary.
\begin{lemma}\label{lem:3.7}
Let $\gamma\in[1,2]$, and let $0<\hat{\rho}<\check{\rho}<\infty$. Suppose constant states $U_{L}, U_{M}, U_{R}\in \mathcal{O}(U_{\infty})$ with $\rho_{L}, \rho_{M}, \rho_{R}\in [\hat{\rho}, \check{\rho}]$, satisfy that
\begin{eqnarray}\label{eq:3.26}
v_{R}=(1+\tau^{2}u_{R})b_{0}, \quad \omega_{R}=\mathscr{H}_{2}(\textsl{z}_{2}, \omega_{M}, \tau^{2}), \quad
\omega_{M}=\mathscr{H}_{1}(\textsl{z}_{1}, \omega_{L}, \tau^{2}).
\end{eqnarray}
Then, there exist constants $C_{b0}>0$, $C_{b1}>0$ and $C_{6}>0$ independent of $\gamma$,$\tau$, $\textsl{z}_{1}$, $\textsl{z}_{2}$
such that for any constant state $U'_{R}\in \mathcal{O}(U_{\infty})$
with $\rho'_{R}\in [\hat{\rho}, \check{\rho}]$ which satisfies that
\begin{eqnarray}\label{eq:3.27}
v'_{R}=(1+\tau^{2}u'_{R})b_{0}, \qquad \omega'_{R}=\mathscr{H}_{2}(\textsl{z}'_{2}, \omega_{L}, \tau^{2}),
\end{eqnarray}
the following interaction estimates hold:
\begin{itemize}
\item[(1)] For the case that $\mathcal{S}_{1}+\mathcal{S}_{2}\rightarrow \mathcal{S}'_{2}$,\ \emph{i.e.},
for the wave strength interaction that $\nu+\beta\rightarrow \beta'$, it holds that
\begin{eqnarray}\label{eq:3.28}
\begin{split}
|\beta'|\leq K_{b0}|\nu|+|\beta|+C_{6}(\gamma-1+\tau^{2})|\beta||\nu|,
\end{split}
\end{eqnarray}
with
\begin{eqnarray}\label{eq:3.29}
K_{b0}\Big|_{\gamma=1, \tau=0}=1+C_{b0}.
\end{eqnarray}

\item[(2)] For the case that $\mathcal{R}_{1}+\mathcal{S}_{2}\rightarrow \mathcal{S}'_{2}$
(or $\mathcal{R}_{1}+\mathcal{S}_{2}\rightarrow \mathcal{R}'_{2}$),\ \emph{i.e.},
for the wave interaction that $o + \beta \rightarrow \beta'$, it holds that
\begin{eqnarray}\label{eq:3.30}
\begin{split}
|\beta'|\leq |\beta|+C_{6}(\gamma-1+\tau^{2})|\beta||o|-C_{b1}|o|.
\end{split}
\end{eqnarray}

\item[(3)] For the case that $\mathcal{S}_{1}+\mathcal{R}_{2}\rightarrow \mathcal{S}'_{2}$
(or $\mathcal{S}_{1}+\mathcal{R}_{2}\rightarrow \mathcal{R}'_{2}$),\ \emph{i.e.},
for the wave interaction that $\nu+\pi\rightarrow \beta'$ (or $\beta+o\rightarrow \beta'$),
it holds that
\begin{eqnarray}\label{eq:3.31}
\begin{split}
|\beta'|\leq K_{b1}|\nu|+C_{6}(\gamma-1+\tau^{2})|\nu|^{2},
\end{split}
\end{eqnarray}
where
\begin{eqnarray}\label{eq:3.31x}
\begin{split}
K_{b1}\Big|_{\gamma=1, \tau=0}=1.
\end{split}
\end{eqnarray}

\item[(4)] For the case that  $\mathcal{R}_{1}+\mathcal{R}_{2}\rightarrow \mathcal{R}'_{2}$,
\ \emph{i.e.}, for the wave interaction that $o +\pi\rightarrow \pi'$, it holds that $|o|+|\pi|=|\pi'|$.
\end{itemize}
\end{lemma}

\begin{proof}
For the notational simplicity, for $\gamma=1$ and $\tau=0$, let
\begin{eqnarray}\label{eq:3.32a}
\Phi_{1}(\alpha, U; \tau^{2})\big|_{\gamma=1, \tau=0}=:g(\alpha),
\end{eqnarray}
for some $\alpha>0$ and $U\in \mathcal{O}(U_{\infty})$. Then function $g$ satisfies the properties that
\begin{eqnarray}\label{eq:3.32b}
0<g'(\alpha)<1, \ \  \ \  g''(\alpha)>0,
\end{eqnarray}
for $\alpha>0$. As shown in Remark \ref{rem:2.4},
\begin{eqnarray}\label{eq:3.32c}
\Phi_{2}(\beta, U; \tau^{2})\big|_{\gamma=1, \tau=0}=-g(-\beta),
\end{eqnarray}
for some $\beta<0$ and $U\in \mathcal{O}(U_{\infty})$.

For the first case $\mathcal{S}_{1}+\mathcal{S}_{2}\rightarrow \mathcal{S}'_{2}$, from \eqref{eq:3.26}-\eqref{eq:3.27}, we have
\begin{eqnarray}\label{eq:3.32}
v_{R}=(1+\tau^{2}u_{R})b_{0}, \quad \omega_{-,M}-\omega_{-,R}=\Phi_{2}(\beta, U_{M}, \tau^{2}), \quad
\beta=\omega_{+,M}-\omega_{+,R},
\end{eqnarray}
\begin{eqnarray}\label{eq:3.33}
\omega_{+,L}-\omega_{+,M}=\Phi_{1}(\nu, U_{L}, \tau^{2}), \quad \nu=\omega_{-,L}-\omega_{-,M},
\end{eqnarray}
and
\begin{eqnarray}\label{eq:3.34}
v'_{R}=(1+\tau^{2}u'_{R})b_{0}, \quad \omega_{-,L}-\omega'_{-,R}=\Phi_{2}(\beta', U_{L}, \tau^{2}),
\quad \beta'=\omega_{+,L}-\omega'_{+,R}.
\end{eqnarray}
Then,
\begin{eqnarray}\label{eq:3.35}
(1+\tau^{2}u_{R})v'_{R}=(1+\tau^{2}u'_{R})v_{R}.
\end{eqnarray}
where $(u_{R},v_{R})=(u,v)(\beta, \nu,  \gamma-1,\tau^{2}, U_{L})$ and $(u'_{R},v'_{R})=(u',v')(\beta', \gamma-1,\tau^{2}, U_{L})$.
Let
\begin{eqnarray}\label{eq:3.35x}
\mathscr{L}_{1}(\beta',\beta, \nu,  \gamma-1,\tau^{2},b_{0}, U_{L}):=(1+\tau^{2}u_{R})v'_{R}-(1+\tau^{2}u'_{R})v_{R}.
\end{eqnarray}
When $\gamma=1$ and $\tau=0$, equation \eqref{eq:3.35x} is
\begin{eqnarray}\label{eq:3.35a}
\mathscr{L}_{1}\big|_{\gamma=1, \tau=0}=
\frac{1}{2a_{\infty}}\Big(\beta'+g(-\beta')-\beta-g(-\beta)+\nu-g(\nu)\Big),
\end{eqnarray}
for $\beta<0$ and $\nu>0$.

Notice that $\frac{\partial\big(\mathscr{L}_{1}\big|_{\gamma=1, \tau=0}\big)}{\partial \beta'}=\frac{1-g'(-\beta')}{2a_{\infty}}>0$, $\lim_{\beta'\rightarrow-\infty}\mathscr{L}_{1}\big|_{\gamma=1, \tau=0}=-\infty$, and
\begin{eqnarray*}
\begin{split}
\mathscr{L}_{1}\big|_{\gamma=1, \tau=0,\beta'=\beta-\nu}&=\frac{1}{2a_{\infty}}\Big(g(\nu-\beta)-g(-\beta)-g(\nu)\Big)
=-\beta\nu\int^{1}_{0}\int^{1}_{0}g''(\xi\nu-\eta\beta)d\xi d\eta>0.
\end{split}
\end{eqnarray*}

So equation $\mathscr{L}_{1}(\beta',\beta, \nu, 0,0, b_{0}, U_{L})=0$ admits a unique root $\beta'_{0}$.
By Lemma \ref{lem:2.3},
\begin{eqnarray*}
\frac{\partial \mathscr{L}_{1}(\beta',\beta, \nu,  \gamma-1,\tau^{2},b_{0}, U_{L})}{\partial \beta'}\Big|_{\gamma=1, \tau=0, \beta'=\beta'_{0}}
=\frac{1}{2a_{\infty}}\Big(1-g'(-\beta'_{0})\Big)>C>0,
\end{eqnarray*}
for some $C>0$ depends only on the $\hat{\rho}$ and $\check{\rho}$.

Therefore, it follows from the implicit function theorem that $\beta'$ can be solved as a $C^{2}$ function of
$\beta, \nu,  \gamma-1,\tau^{2},b_{0}$ and $U_{L}$, that is 
\begin{eqnarray*}
\begin{split}
\beta'&=\beta'(\beta, \nu, \gamma-1, \tau^{2})\\[5pt]
&=\beta'(0, \nu, \gamma-1, \tau^{2})+\beta'(\beta,0, \gamma-1, \tau^{2})+\mathcal{O}(\beta, \nu, \gamma-1, \tau^{2})\beta\nu\\[5pt]
&=K_{b}\nu+\beta+\mathcal{O}(\beta, \nu, \gamma-1, \tau^{2})\beta\nu,
\end{split}
\end{eqnarray*}
where coefficient $K_{b}$ is given by \eqref{eq:3.24} in Lemma \ref{lem:3.6}. Moreover,
\begin{eqnarray*}
\begin{split}
\beta'_{0}&:=\beta'(\beta, \nu, 0, 0)=\beta'(0, \nu, 0, 0)+\beta'(\beta,0, 0, 0)+\mathcal{O}(\beta, \nu, 0, 0)\beta\nu
=-\nu+\beta+\mathcal{O}(\beta, \nu, 0, 0)\beta\nu.
\end{split}
\end{eqnarray*}

Subtracting the two identities above implies that
\begin{eqnarray}\label{eq:3.36}
\begin{split}
\beta'=\beta'_{0}+\mathcal{O}(1)(\gamma-1+\tau^{2})\nu+\mathcal{O}(1)(\gamma-1+\tau^{2})\beta\nu.
\end{split}
\end{eqnarray}

\begin{figure}[ht]
\begin{center}
\begin{tikzpicture}[scale=0.6]
\draw [line width=0.05cm] (-6,-4.0) --(0.5,2.5);
\draw [line width=0.04cm][red](-1.5,-2.5)to[out=180, in=45](-4.5,-4);
\draw [line width=0.04cm][blue](-4.5,-4)to[out=90, in=-135](-4.0,-2.0);
\draw [line width=0.04cm][blue](-1.5,-2.5)to[out=90, in=-135](0,2);

\draw [thick][->] (-4.5,-2)--(-5,-0.5);

\node at (-1.5, -2.5){$\bullet$};
\node at (-4.5, -4.0){$\bullet$};
\node at (0,2){$\bullet$};
\node at (-4.0,-2.0){$\bullet$};

\node at (-0.3, -2.5){$(r_{L}, s_{L})$};
\node at (-4.3, -4.3){$(r_{M}, s_{M})$};
\node at (1.3,1.9){$(r'_{R}, s'_{R})$};
\node at (-2.7,-2.0){$(r_{R}, s_{R})$};

\node at (-3.0, -3.5){$S_{1}$};
\node at (-4.0, -2.8){$S_{2}$};
\node at (-0.8, -0.5){$S'_{2}$};
\node at (-4.5, 0){$r-s=2a_{\infty}b_{0}$};
\end{tikzpicture}
\end{center}
\caption{$S_{1}$ and $S_{2}$ waves interaction and reflection on the boundary}\label{fig20}
\end{figure}
So the remaining task is to estimate $\beta'_{0}$ more carefully (see Fig. \ref{fig20}) for the case that $\gamma=1$ and $\tau=0$. By \eqref{eq:3.32}-\eqref{eq:3.35}, we have the relation that
\begin{eqnarray}\label{eq:3.37}
\begin{split}
\beta'_{0}+\nu-\beta=g(\nu)+g(-\beta)-g(-\beta'_{0}),
\end{split}
\end{eqnarray}
where $\nu=\big(\omega_{-,L}-\omega_{-,M}\big)\big|_{\gamma=1,\tau=0}>0$, $\beta=\big(\omega_{+,M}-\omega_{+,R}\big)\big|_{\gamma=1,\tau=0}<0$ and $\beta'_{0}=\big(\omega_{+,L}-\omega'_{+,R}\big)\big|_{\gamma=1,\tau=0}<0$.
Direct computation shows that
\begin{eqnarray*}
\begin{split}
g(\nu)-g(-\beta)-g(-\beta'_{0})&=g(\nu-\beta)-g(-\beta'_{0})+g(\nu)+g(-\beta)-g(\nu-\beta)\\[5pt]
&\geq g'(\xi_{1})(\nu-\beta+\beta'_{0})+g(-\beta)-g(\nu-\beta)\\[5pt]
&\geq g'(\xi_{1})(\nu-\beta+\beta'_{0})+g'(\xi_{2})(-\nu),
\end{split}
\end{eqnarray*}
where $\xi_{1}\in (-\beta'_{0}, \nu-\beta)$ and $\xi_{2}\in (\nu-\beta,-\beta')$.
This together with \eqref{eq:3.37} yields that
\begin{eqnarray*}
\begin{split}
-\beta'_{0}-\nu+\beta\leq \frac{g'(\xi_{2})}{1-g'(\xi_{1})}\nu.
\end{split}
\end{eqnarray*}

Let $C_{b0}=\sup_{\xi_{1}\in (-\beta'_{0}, \nu-\beta), \xi_{2}\in (\nu-\beta,-\beta')}
\frac{g'(\xi_{2})}{1-g'(\xi_{1})}$, then we have
\begin{eqnarray*}
\begin{split}
|\beta'_{0}|\leq (1+C_{b0})|\nu|+\beta.
\end{split}
\end{eqnarray*}

So it follows from \eqref{eq:3.36} that
\begin{eqnarray*}
\begin{split}
|\beta'|&
&\leq\Big(1+C_{b0}+\mathcal{O}(1)(\gamma-1+\tau^{2})\Big)|\nu|+\beta+\mathcal{O}(1)(\gamma-1+\tau^{2})|\nu||\beta|.
\end{split}
\end{eqnarray*}
This completes the proof for the first case.

Next, for the second case $\mathcal{R}_{1}+\mathcal{S}_{2}\rightarrow \mathcal{S}'_{2}$, note that $\beta$ and $\beta'$ satisfy \eqref{eq:3.32} and \eqref{eq:3.34}, 
\begin{eqnarray}\label{eq:3.38}
\omega_{+,M}=\omega_{+,L}, \quad o=\omega_{-,L}-\omega_{-,M}<0,\qquad\mbox{on }\mathcal{R}_{1},
\end{eqnarray}
and equality \eqref{eq:3.35} holds on the boundaries $\Gamma_{k}$ and $\Gamma_{k+1}$
with $(u_{R},v_{R})=(u,v)(\beta, o,  \gamma-1,\tau^{2}, U_{L})$ and $(u'_{R},v'_{R})=(u',v')(\beta', \gamma-1,\tau^{2}, U_{L})$.
Let
\begin{eqnarray*}
\mathscr{L}_{2}(\beta',\beta, o,  \gamma-1,\tau^{2},b_{0}, U_{L}):=(1+\tau^{2}u_{R})v'_{R}-(1+\tau^{2}u'_{R})v_{R}.
\end{eqnarray*}

As done for the first case, similarly, it follows from the implicit function theorem that $\beta'$ can be solved as a $C^{2}$ function of
$\beta, \nu,  \gamma-1,\tau^{2},b_{0}, U_{L}$ with the estimate that
\begin{eqnarray}\label{eq:3.39}
\begin{split}
|\beta'|\leq |\beta'(\beta, o, 0, 0)|+\mathcal{O}(1)(\gamma-1+\tau^{2})|\beta| |o|.
\end{split}
\end{eqnarray}

\begin{figure}[ht]
\begin{center}
\begin{tikzpicture}[scale=0.6]
\draw [line width=0.05cm] (-5.5,-3.5) --(0.5,2.5);
\draw [line width=0.04cm][blue](-1.5,-3)--(-3.5,-3);
\draw [line width=0.04cm][red](-3.5,-3)to[out=90, in=-135](-3.0,-1);
\draw [line width=0.04cm][red](-1.5,-3)to[out=90, in=-135](0,2);

\draw [thick][->] (-2.5, 0.5)--(-3.5,1.5);

\node at (-1.5, -3){$\bullet$};
\node at (-3.5, -3){$\bullet$};
\node at (-3,-1){$\bullet$};
\node at (0,2){$\bullet$};

\node at (-3.9, -3.5){$(r_{L}, s_{L})$};
\node at (-0.1, -3.0){$(r_{M}, s_{M})$};
\node at (1.4,2){$(r_{R}, s_{R})$};
\node at (-4.3,-0.9){$(r'_{R}, s'_{R})$};

\node at (-2.3, -3.5){$R_{1}$};
\node at (-0.8, -0.5){$S_{2}$};
\node at (-3.0, -2.1){$S'_{2}$};
\node at (-4.0, 2){$r-s=2a_{\infty}b_{0}$};
\end{tikzpicture}
\end{center}
\caption{$R_{1}$ and $S_{2}$ waves interaction and reflection on the boundary}\label{fig21}
\end{figure}

Now, we will estimate $\beta'(\beta, 0, 0, 0)$ (see Fig.\ref{fig21}). 
Let $\beta'_{1}=\beta'(\beta, 0, 0, 0)$. Then
\begin{eqnarray*}
\begin{split}
\beta'_{1}-\beta+g(-\beta'_{1})-g(-\beta)=|o|.
\end{split}
\end{eqnarray*}

By the mean value theorem, we further have 
\begin{eqnarray*}
\begin{split}
\beta'_{1}-\beta=\frac{|o|}{1-g'(\xi_{3})}, \quad \xi_{3}\in(-\beta,-\beta'_{1}),
\end{split}
\end{eqnarray*}
which implies that
\begin{eqnarray*}
\begin{split}
|\beta'_{1}|\leq |\beta|-C_{b1}|o|,
\end{split}
\end{eqnarray*}
where $C_{b1}=\inf_{\xi_{3}\in(-\beta,-\beta'_{1})}\frac{1}{1-g'(\xi_{3})}$.
This together with \eqref{eq:3.39} yields estimate \eqref{eq:3.30}.

For the third case that $\mathcal{S}_{1}+\mathcal{R}_{2}\rightarrow \mathcal{S}'_{2}$, we know that \eqref{eq:3.33}
and \eqref{eq:3.34} hold on $\mathcal{S}_{1}$ and $\mathcal{S}'_{2}$, 
\begin{eqnarray*}
\begin{split}
\omega_{-, M}=\omega_{-,R}, \quad \pi=\omega_{+, M}-\omega_{+,R}>0, \qquad\mbox{on }\mathcal{R}_{2},
\end{split}
\end{eqnarray*}
and equality \eqref{eq:3.35} holds on boundaries $\Gamma_{k}$ and $\Gamma_{k+1}$
with $(u_{R},v_{R})=(u,v)(\pi, \nu,  \gamma-1,\tau^{2}, U_{L})$ and $(u'_{R},v'_{R})=(u',v')(\beta', \gamma-1,\tau^{2}, U_{L})$.
Let
\begin{eqnarray*}
\mathscr{L}_{3}(\beta',\beta, o,  \gamma-1,\tau^{2},b_{0}, U_{L}):=(1+\tau^{2}u_{R})v'_{R}-(1+\tau^{2}u'_{R})v_{R}.
\end{eqnarray*}

Then similarly as done for the first case, by the implicit function theorem, $\beta'$ can be solved as a $C^{2}$ function of
$\pi, \nu,  \gamma-1,\tau^{2},b_{0}, U_{L}$, with the following estimate 
\begin{eqnarray}\label{eq:3.40}
\begin{split}
|\beta'|\leq |\beta'(\pi, \nu, 0, 0)|+\mathcal{O}(1)(\gamma-1+\tau^{2})|\nu|+\mathcal{O}(1)(\gamma-1+\tau^{2})|\nu|^{2}.
\end{split}
\end{eqnarray}

\begin{figure}[ht]
\begin{center}
\begin{tikzpicture}[scale=0.6]
\draw [line width=0.05cm] (-6.5,-4.5) --(0.5,2.5);
\draw [line width=0.04cm][red](-0.5,-1.5)to[out=180, in=35](-6.0,-2.5);
\draw [line width=0.04cm][blue](-6.0,-2.5)--(-6.0,-4.0);
\draw [line width=0.04cm][blue](-0.5,-1.5)to[out=90, in=-125](0,2);

\draw [thick][->] (-2.5, 0.5)--(-3.5,1.5);

\node at (-0.5, -1.5){$\bullet$};
\node at (-6.0, -2.5){$\bullet$};
\node at (0,2){$\bullet$};
\node at (-6.0,-4.0){$\bullet$};

\node at (-0.4, -2.0){$(r_{L}, s_{L})$};
\node at (-7.4, -2.4){$(r_{M}, s_{M})$};
\node at (1.4,1.8){$(r'_{R}, s'_{R})$};
\node at (-4.5,-4.0){$(r_{R}, s_{R})$};

\node at (-3.0, -2.0){$S_{1}$};
\node at (-6.5, -3.2){$R_{2}$};
\node at (0, 0){$S'_{2}$};
\node at (-4.0, 2){$r-s=2a_{\infty}b_{0}$};
\end{tikzpicture}
\end{center}
\caption{$S_{1}$ and $R_{2}$ waves interaction and reflection on the boundary}\label{fig22}
\end{figure}

For the term $\beta'_{2}=\beta'(\pi, \nu, 0, 0)$ (see Fig.\ref{fig22}), we have that
\begin{eqnarray*}
\begin{split}
\beta'_{2}+\nu=g(\nu)-g(-\beta'_{2})+\pi\geq g'(\xi_{4})(\beta'_{2}+\nu), \quad \xi_{4}\in(-\beta'_{2}, \nu)
\end{split}
\end{eqnarray*}
which implies that
$|\beta'_{2}|\leq |\nu|$.
Thus, it with \eqref{eq:3.40} yields estimate \eqref{eq:3.31}.

Finally, for the fourth case that $\mathcal{R}_{1}+\mathcal{R}_{2}\rightarrow \mathcal{R}'_{2}$, estimate is obvious since across the rarefaction waves the strength of the waves is unchanged.
\end{proof}

\section{Global entropy solutions with large data}
In this section, we first construct the approximate solution for the initial-boundary value problem \eqref{eq:1.16}--\eqref{eq:1.18}
by employing the modified Glimm scheme in an approximate domain $\Omega_{\Delta}$ which will be defined below, and then show the existence of global entropy solutions with large data.

\subsection{Modified Glimm scheme for the problem \eqref{eq:1.16}--\eqref{eq:1.18}}\label{sec:3.1}
Since $T.V. (U_{0})<\infty$, limits $\lim_{y\rightarrow \pm \infty}U_{0}(y)$ exist, which are denoted by $U_{\pm}$. Let
\begin{eqnarray}\label{eq:3.1}
\begin{split}
\mathcal{O}(U_{\pm})=\big\{U: |U-U_{-}|+|U-U_{+}|< 4T.V. (U_{0})\big\}.
\end{split}
\end{eqnarray}

\begin{figure}[ht]
	\begin{center}
		\begin{tikzpicture}[scale=0.6]
		\draw [line width=0.02cm] (-3.5,-6.5) --(-3.5,1.5);
		\draw [line width=0.02cm] (-0.5,-6.5) --(-0.5,1.0);
		\draw [line width=0.02cm] (2.5,-6.5) --(2.5,0.5);
		\draw [line width=0.03cm] (-3.5,1.5)--(2.5,0.5);
		\draw [thin] (-3.5,-0.5)--(2.5,-1.5);
		\draw [thin] (-3.5,-2.5)--(2.5,-3.5);
		\draw [thin] (-3.5,-4.5)--(2.5,-5.5);
		
		\draw [thin] (-3, 1.4) --(-2.6, 1.8);
		\draw [thin] (-2.6, 1.35) --(-2.2, 1.75);
		\draw [thin] (-2.2, 1.30) --(-1.8, 1.70);
		\draw [thin] (-1.8, 1.23) --(-1.4, 1.63);
		\draw [thin] (-1.4, 1.16) --(-1.0, 1.56);
		\draw [thin] (-1.0, 1.10) --(-0.6, 1.50);
		\draw [thin] (-0.6, 1.03) --(-0.2, 1.43);
		\draw [thin] (-0.2, 0.97) --(0.2, 1.37);
		\draw [thin] (0.2, 0.9) --(0.6, 1.30);
		\draw [thin] (0.6, 0.83) --(1, 1.23);
		\draw [thin] (1, 0.76) --(1.4, 1.16);
		\draw [thin] (1.4, 0.67) --(1.8, 1.07);
		\draw [thin] (1.8, 0.60) --(2.2, 1.0);
		\draw [thin] (2.2, 0.55) --(2.6, 0.95);
		
		\draw [thick][red] (-3.5,1.5)--(-1.5,0.5);
		\draw [thick][blue](-3.5,-0.5)--(-1.5,-0.3);
		\draw [thick][blue](-3.5,-0.5)--(-1.5,-1.2);
		\draw [thick][red] (-3.5,-4.5)--(-1.5,-4.3);
		\draw [thick][red] (-3.5,-4.5)--(-1.5,-5.3);
		
		\draw [thick][blue](-0.5,1)--(1.2,0);
		\draw [thick][red] (-0.5,-3)--(1.3,-2.8);
		\draw [thick][red] (-0.5,-3)--(1.3,-3.9);
		
		\draw [thick][->] (-1.8,2.7)--(-2.5,2);
		\draw [thick][->] (1.5,2.2)--(0.6,1.5);
		
		\node at (-1.8, 3.0){$\mathbf{n}_{k-1}$};
		\node at (1.8, 2.4){$\mathbf{n}_{k}$};
		\node at (-1.7, -2.0){$\Omega_{\Delta, k-1}$};
		\node at (0.8, -2.0){$\Omega_{\Delta, k}$};
		\node at (-3.5, -6.9){$x_{k-1}$};
		\node at (-0.5, -6.9){$x_{k}$};
		\node at (2.5, -6.9){$x_{k+1}$};
		\end{tikzpicture}
	\end{center}
	\caption{The modified Glimm scheme}\label{fig3.1x}
\end{figure}

\par Let $\Delta x$ be the mesh length in the $x$-direction. Choose
a set of points $\{A_{k}\}_{k=0}$ with $A_{k}=(x_{k}, b_{k})=(k\Delta x, b_{0}k\Delta x)$ on the straight
boundary $y=b_{0}x$ in order.
As shown in Fig. \ref{fig3.1x}, define
\begin{eqnarray}\label{eq:3.2}
\begin{split}
&b_{\Delta}(x)=b_{k}+(x-x_{k})b_{0},  \quad  \forall x\in[k\Delta x, \ (k+1)\Delta x),\  k\geq0, \\[5pt]
&\Omega_{\Delta, k}=\{(x,y):\ k\Delta x \leq x<(k+1)\Delta x,\ y<b_{\Delta}(x) \},\\[5pt]
&\Gamma_{\Delta, k}=\{(x,y):\ k\Delta x \leq x<(k+1)\Delta x,\ y=b_{\Delta}(x) \},\\[5pt]
&\Omega_{\Delta}=\bigcup_{k\geq0}\Omega_{\Delta, k},\ \ \  \Gamma_{\Delta}=\bigcup_{k\geq0}\Gamma_{\Delta, k}.
\end{split}
\end{eqnarray}
Let $\mathbf{n}_{k}$ be the outer unit normal vector to $\Gamma_{\Delta, k}$ as
\begin{eqnarray}\label{eq:3.3}
\mathbf{n}_{k}=\frac{(b_{k+1}-b_{k}, -x_{k+1}+x_{k})}{\sqrt{(b_{k+1}-b_{k})^{2}+(x_{k+1}-x_{k})^{2}}}
=\frac{(b_{0}, -1)}{\sqrt{1+b^{2}_{0}}}.
\end{eqnarray}

\par We choose the mesh length in the $y$-direction as $\Delta y$ such that the following Courant-Friedrichs-Lewy condition holds:
\begin{eqnarray}\label{eq:3.4}
\begin{split}
\frac{\Delta y}{\Delta x}<
\sup_{U\in \mathcal{O}(U_{\pm}), \tau\in(0, \epsilon_{*})}\big\{\max_{l=\pm}|\lambda_{l}(U,\tau^{2})|\big\}-b_{0},
\end{split}
\end{eqnarray}
where $\epsilon_{*}=\min\{\epsilon_{6}, \epsilon_{7}, \epsilon_{8}, \epsilon_{9}\}$.
\par For any non-negative integer $k$ and negative integer $n$, \emph{i.e.},  for $k\geq 0$ and $n\leq -1$, define
\begin{eqnarray}\label{eq:3.5}
\begin{split}
y_{k,n}=b_{k}+(2n+1+\theta_{k})\Delta y,
\end{split}
\end{eqnarray}
where $\theta_{k}$ is randomly chosen in $(-1,1)$. Then, let
\begin{eqnarray}\label{eq:3.6}
\begin{split}
P_{k,n}=(x_{k}, y_{k,n}),
\end{split}
\end{eqnarray}
be the mesh points and define the approximate solutions $U_{\Delta, \theta}(x,y)$ in $\Omega_{\Delta}$ in $\Omega_{\Delta}$
for any $\mathbf{\theta}=(\theta_{0}, \theta_{1}, \cdot\cdot\cdot)$ via the Glimm Scheme inductively as follows.

\par $Step\ 1.$ For $k=0$, we approximate the initial data by piecewise constant functions. 
\begin{eqnarray}\label{eq:3.7}
&U_{\Delta, \theta}(x=0, y)=
\left\{
\begin{array}{lllll}
U_{0}(y_{0,n}),\ \ \ &b_{k}+2(n+1)\Delta y\leq y\leq b_{k}+2n\Delta y,\\[5pt]
U_{0}(y_{0,n+1}),\ \ \ &b_{k}+2(n+2)\Delta y\leq y\leq b_{k}+2(n+1)\Delta y,
\end{array}
\right.
\end{eqnarray}
where $U_{0}(y_{0,n})$ and $U_{0}(y_{0,n+1})$ are constant states.

\par $Step\ 2.$ Assume the approximate solution $U_{\Delta,\theta}(x,y)$ has been defined in
$\Omega_{\Delta}\cap \{0<x<x_{k}\}$ for $k>0$. Then, for any $n\leq -1$ and $y\in(b_{k}+2n\Delta y, b_{k}+2(n+1)\Delta y)$, define $U^{0}_{k,n}$ by
\begin{eqnarray}\label{eq:3.8}
U^{0}_{k,n}=U_{\Delta, \theta}(x_k-,y_{k,n}),
\end{eqnarray}

\par Now, we first solve the Riemann problem in the diamond $T_{k,0}$ whose vertices are $(x_{k}, b_{k})$, $(x_{k}, b_{k}-\Delta y)$,
$(x_{k+1}, b_{k})$ and $(x_{k+1}, b_{k}-\Delta y)$ with initial data $U_{\Delta,\theta}=U^{0}_{k,0}$, that is
\begin{eqnarray}\label{eq:3.9}
\left\{
\begin{array}{lllll}
\partial_{x}W(U_{k,0},\tau^{2})+\partial_{y}F(U_{k,0},\tau^{2})=0, \ \ \ &\mbox{in} \ \ T_{k,0},  \\[5pt]
U_{k,0}|_{x=x_{k}}=U^{0}_{k,0},  \ \ \ &\mbox{on}\ \ \{b_{k}-\Delta y<y<b_{k}\},  \\[5pt]
\big((1+\tau^{2}u(\rho_{k,0},v_{k,0},\tau^{2})), v_{k,0}\big)\cdot\mathbf{n}_{k}=0,
\ \ \ &\mbox{on} \ \ \Gamma_{k},
\end{array}
\right.
\end{eqnarray}
to obtain the Riemann solution $U_{k,0}$ in $T_{k,0}$ by Proposition \ref{prop:2.2}.
Define
\begin{eqnarray}\label{eq:3.10}
U_{\Delta, \theta}=U_{k,0}, \ \ \ \mbox{in}\  T_{k,0}.
\end{eqnarray}

\par Next, we solve the Riemann problem in each diamond $T_{k,n}$ for $n\leq -2$ whose
vertices are $(x_{k},b_{k}+2n\Delta y)$, $(x_{k},b_{k}+2(n+1)\Delta y)$, $(x_{k+1},b_{k}+2n\Delta y)$
and $(x_{k+1},b_{k}+2(n+1)\Delta y)$
\begin{eqnarray}\label{eq:3.11x}
\left\{
\begin{array}{lllll}
\partial_{x}W(U_{k,n},\tau^{2})+\partial_{y}F(U_{k,n},\tau^{2})=0,  \qquad \  \ \  \mbox{in} \ \ T_{k,n},  \\[5pt]
U_{k,n}|_{x=x_{k}}=\left\{
\begin{array}{lllll}
U^{0}_{k, n},\ \ \ &b_{k}+2n\Delta y<y< b_{k}+2(n+1)\Delta y,\\[5pt]
U^{0}_{k, n-1},\ \ \ & b_{k}+2(n-1)\Delta y<y< b_{k}+2n\Delta y.
\end{array}
\right.
\end{array}
\right.
\end{eqnarray}
\par By Proposition \ref{prop:2.1}, Riemann problem \eqref{eq:3.11x} admits a Riemann solution $U_{k,n}$ in $T_{k,n}$. Define
\begin{eqnarray}\label{eq:3.12x}
U_{\Delta, \theta}=U_{k,n}, \ \ \ \mbox{in}\ \ \  T_{k,n}.
\end{eqnarray}

Therefore, we can construct the approximate solution $U_{\Delta, \theta}(x,y)$ globally provided that we can obtain the uniform bound of the approximate solutions, which will be the main goal in the next subsection.

\subsection{Glimm-type functional and the global existence of entropy solutions}
In this subsection, we will introduce the weighted Glimm-type functional and apply the functional to show the convergence of the approximation solutions and then obtain the global existence of entropy solutions of problem \eqref{eq:1.16}-\eqref{eq:1.18} of large data.
To obtain it, as done in \cite{glimm}, we introduce mesh curves $J$ which is space-like, and consists
of the line segments jointing the random points $P_{k,n}$ one by one in the order of $n$.
Obviously, region $\Omega_{\Delta}$ is the union of the diamonds whose boundaries are the line segements of the mesh curves with four adjacant random points as their vertices. Moreover, $J$ divides the region $\Omega_{\Delta}$
into two subregions denoted by $J^{-}$ and $J^{+}$, where $J^{-}$ denotes the subregion containing the $y$-axis and $J^{+}=\Omega_{\Delta}\backslash J^{-}$.
Now we can define the order of the mesh curves.
\begin{definition}\label{def:4.1}
Assume that $I$ and $J$ are two mesh curves, we call $J>I$ if and only if every mesh point of the curve $J$ is either on $I$
or contained in $I^{+}$. Moreover, if $J>I$ and every mesh points of $J$ except one lie on $I$, then we call $J$ is an immediate
successor to $I$.
\end{definition}

For the approximate solution $U_{\Delta, \theta}(x,y)$, let $S_{j}(J)$, where $j=1$ or $2$, be the set of $j$-shock
waves which go across the mesh curve $J$. Let $S(J):=S_{1}(J)\cap S_{2}(J)$. Define the Glimm-type functional
\begin{eqnarray}\label{eq:4.1}
\begin{split}
F(J)=L(J)+4C_{*}(\gamma-1+\tau^{2})Q(J),
\end{split}
\end{eqnarray}
where
\begin{eqnarray}\label{eq:4.2x}
\begin{split}
L(J)=\mathscr{K}_{b}L_{1}(J)+L_{2}(J),
\end{split}
\end{eqnarray}
\begin{eqnarray}\label{eq:4.2}
\begin{split}
L_{1}(J)=\sum \big\{|\alpha|: \alpha\in S_{1}(J) \big\}, \quad L_{2}(J)=\sum \big\{|\beta|: \alpha\in S_{2}(J) \big\},
\end{split}
\end{eqnarray}
and
\begin{eqnarray}\label{eq:4.3}
\begin{split}
Q(J)=\sum \big\{|\alpha||\beta|: \alpha\in S_{1}(J),\ \beta\in S_{2}(J)\ and \ \alpha, \beta\ are\ approaching \big\}.
\end{split}
\end{eqnarray}

Constants $\mathscr{K}_{b}$ and $C_{*}$ satisfy that
\begin{eqnarray}\label{eq:4.3g}
\begin{split}
\max\big\{ K_{b0}, \ K_{b1}, \ 1 \big\}<\mathscr{K}_{b}< \min\big\{\frac{1}{\delta}, \ 1+C_{0},\ 4 \big\}, \quad \
C_{*}>\max\big\{C_{5}, \  C_{6}, \mathscr{K}_{b}\big\}.
\end{split}
\end{eqnarray}

Then, we have the following lemma for functional $F(J)$, which ensures the uniform bound of the approximate solutions.
\begin{lemma}\label{lem:4.1}
Suppose that $I$ and $J$ are any two space-like mesh curves satisfying $J>I$. There exists a constant $C_{7}>0$
depending only on $C_{0}$ and $\delta$, such that if $C_{*}(\gamma-1+\tau^{2})F(I)\leq C_{7}$, then it holds that
\begin{eqnarray}\label{eq:4.5}
\begin{split}
F(J)< F(I).
\end{split}
\end{eqnarray}
\end{lemma}

\begin{proof}
Without loss of the generality, we only consider the case that  $J$ is an immediate successor to $I$, since the other cases can be treated easily by the induction method. Let $\Lambda$ be the diamond between $I$ and $J$, \emph{i.e.}, $\Lambda=I'\cup J'$, where $I=I_{0}\cup I'$
and $J=I_{0}\cup J'$. The proof is devided into two cases depending the location of $\Lambda$.

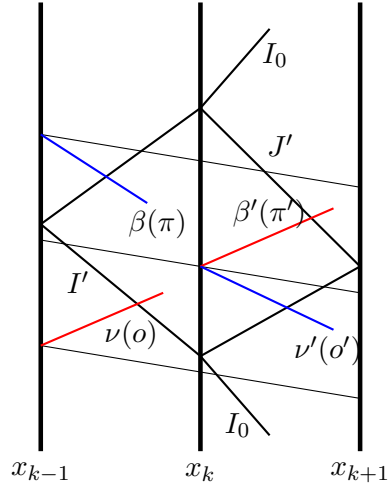
\begin{figure}[ht]
\begin{center}
\begin{tikzpicture}[scale=0.7]
\draw [line width=0.06cm] (-3.5,-6.5) --(-3.5,2.0);
\draw [line width=0.06cm] (-0.5,-6.5) --(-0.5,2.0);
\draw [line width=0.06cm] (2.5,-6.5) --(2.5,2.0);

\draw [thick] (-3.5,-2.2)--(-0.5,0);
\draw [thick] (-3.5,-2.2)--(-0.5,-4.7);
\draw [thick] (-0.5,0)--(2.5,-3.0);
\draw [thick] (-0.5,-4.7)--(2.5,-3.0);
\draw [thick] (-0.5,0)--(0.8,1.5);
\draw [thick] (-0.5,-4.7)--(0.8, -6.2);

\draw [thin] (-3.5,-0.5)--(2.5,-1.5);
\draw [thin] (-3.5,-2.5)--(2.5,-3.5);
\draw [thin] (-3.5,-4.5)--(2.5,-5.5);

\draw [thick][blue](-3.5,-0.5)--(-1.5,-1.8);
\draw [thick][red] (-3.5,-4.5)--(-1.2,-3.5);

\draw [thick][red] (-0.5,-3)--(2.0,-1.9);
\draw [thick][blue] (-0.5,-3)--(2.0,-4.2);

\node at (-1.3, -2.2){$\beta(\pi)$};
\node at (-1.8, -4.3){$\nu(o)$};
\node at (0.8, -2.0){$\beta'(\pi')$};
\node at (1.9, -4.6){$\nu'(o')$};

\node at (1.0, -0.7){$J'$};
\node at (-2.8, -3.3){$I'$};
\node at (0.9, 1.0){$I_{0}$};
\node at (0.2, -6.0){$I_{0}$};
\node at (-3.5, -6.9){$x_{k-1}$};
\node at (-0.5, -6.9){$x_{k}$};
\node at (2.5, -6.9){$x_{k+1}$};
\end{tikzpicture}
\end{center}
\caption{$\Lambda$ lies in $\Omega_{\Delta}$}\label{fig23}
\end{figure}

\par $\mathbf{Case}$ $\mathbf{1}.$\ $\Lambda$ lies in the interior of $\Omega_{\Delta}$ (see Fig.\ref{fig23}).
Let us start with case $(1)$ as listed in Lemma \ref{lem:3.5}. For the subcase $(a)$, we have
\begin{eqnarray*}
\begin{split}
L(J)-L(I)\leq C_{5}(\mathscr{K}_{b}+1)(\gamma-1+\tau^{2})|\beta||\nu|.
\end{split}
\end{eqnarray*}

For $Q(J)$, we have that
\begin{eqnarray*}
\begin{split}
Q(J)-Q(I)&=Q(J',I_{0})+Q(I_{0})-Q(I',I_{0})-Q(I_{0})-Q(I')\\[5pt]
&\leq \sum_{\mu\in S(I)}|\mu|\big(|\beta'|+|\nu'|-|\beta|-|\nu|\big)-|\beta||\nu|\\[5pt]
&\leq \Big(C_{5}(\gamma-1+\tau^{2})F(I)-1\Big)|\beta||\nu|.
\end{split}
\end{eqnarray*}

Then 
\begin{eqnarray*}
\begin{split}
F(J)-F(I)&=C_{5}(\mathscr{K}_{b}+1)(\gamma-1+\tau^{2})|\beta||\nu|
+4C_{*}(\gamma-1+\tau^{2})\big(C_{5}(\gamma-1+\tau^{2})F(I)-1\big)|\beta||\nu|\\[5pt]
&\leq 2C_{*}(\gamma-1+\tau^{2})\Big(2C_{*}(\gamma-1+\tau^{2})F(I)-1\Big)|\beta||\nu|.
\end{split}
\end{eqnarray*}

Therefore, if we choose $(\gamma-1+\tau^{2})F(I)<\frac{1}{2C_{*}}$, then we have
$F(J)<F(I)$.

Next, let us consider subcase $(b)$ of case $(1)$ as listed in Lemma \ref{lem:3.5}. By Lemma \ref{lem:3.5}, we have
\begin{eqnarray*}
\begin{split}
L(J)-L(I)\leq &-\mathscr{K}_{b}\zeta +C_{5}(\gamma-1+\tau^{2})|\beta||\nu|+\eta\\[5pt]
\leq&-(\mathscr{K}_{b}-\delta)\zeta + C_{5}(\gamma-1+\tau^{2})|\beta||\nu|,
\end{split}
\end{eqnarray*}
and
\begin{eqnarray*}
\begin{split}
Q(J)-Q(I)&=Q(J',I_{0})+Q(I_{0})-Q(I',I_{0})-Q(I_{0})-Q(I')\\[5pt]
&\leq \sum_{\mu\in S(I)}|\mu|\big(|\beta'|-|\beta|\big)+\sum_{\mu'\in S(I)}|\mu'|\big(|\nu'|-|\nu|\big)-|\beta||\nu|\\[5pt]
&\leq \sum_{\mu\in S(I)}|\mu|\big(\eta +C_{5}(\gamma-1+\tau^{2})|\beta||\nu|\big)-\sum_{\mu'\in S(I)}|\mu'|\zeta-|\beta||\nu|\\[5pt]
&\leq \Big(\delta\zeta +C_{5}(\gamma-1+\tau^{2})|\beta||\nu|\Big)F(I)-|\beta||\nu|.
\end{split}
\end{eqnarray*}
Then
\begin{eqnarray*}
\begin{split}
&F(J)-F(I)\\[5pt]
\leq& -(\mathscr{K}_{b}-\delta)\zeta + C_{5}(\gamma-1+\tau^{2})|\beta||\nu|
+4C_{*}(\gamma-1+\tau^{2})\Big(\big(\delta\zeta +C_{3}(\gamma-1+\tau^{2})|\beta||\nu|\big)F(I)-|\beta||\nu|\Big)\\[5pt]
\leq& 4\delta\zeta \Big(C_{*}(\gamma-1+\tau^{2})F(I)-\frac{\mathscr{K}_{b}-\delta}{4\delta}\Big)
+4C_{*}(\gamma-1+\tau^{2})\Big(C_{*}(\gamma-1+\tau^{2})F(I)-\frac{3}{4}\Big)|\beta||\nu|.
\end{split}
\end{eqnarray*}
Therefore, if $(\gamma-1+\tau^{2})F(I)<\min\big\{\frac{3}{4C_{*}}, \frac{\mathscr{K}_{b}-\delta}{4\delta C_{*}}\big\}$, then
$F(J)<F(I)$.

Finally, let us consider subcase $(c)$ of case $(1)$ at listed in Lemma \ref{lem:3.5}. Note that 
\begin{eqnarray*}
\begin{split}
L(J)-L(I)
\leq-(1-\mathscr{K}_{b}\delta)\zeta + \mathscr{K}_{b}C_{5}(\gamma-1+\tau^{2})|\beta||\nu|,
\end{split}
\end{eqnarray*}
and
\begin{eqnarray*}
\begin{split}
Q(J)-Q(I)&=Q(J',I_{0})+Q(I_{0})-Q(I',I_{0})-Q(I_{0})-Q(I')\\[5pt]
&\leq \sum_{\mu\in S(I)}|\mu|\big(|\nu'|-|\nu|\big)+\sum_{\mu'\in S(I)}|\mu'|\big(|\beta'|-|\beta|\big)-|\beta||\nu|\\[5pt]
&\leq \sum_{\mu\in S(I)}|\mu|\big(\eta +C_{5}(\gamma-1+\tau^{2})|\beta||\nu|\big)-\sum_{\mu'\in S(I)}|\mu'|\zeta-|\beta||\nu|\\[5pt]
&\leq \Big(\delta\zeta +C_{5}(\gamma-1+\tau^{2})|\beta||\nu|\Big)F(I)-|\beta||\nu|.
\end{split}
\end{eqnarray*}
So, we deduce that
\begin{eqnarray*}
\begin{split}
&F(J)-F(I)\\[5pt]
\leq& -(1-\mathscr{K}_{b}\delta)\zeta + \mathscr{K}_{b}C_{5}(\gamma-1+\tau^{2})|\beta||\nu|
+4C_{*}(\gamma-1+\tau^{2})\Big(\big(\delta\zeta +C_{5}(\gamma-1+\tau^{2})|\beta||\nu|\big)F(I)-|\beta||\nu|\Big)\\[5pt]
\leq& 4\delta\zeta \Big(C_{*}(\gamma-1+\tau^{2})F(I)-\frac{1-\mathscr{K}_{b}\delta}{4\delta}\Big)
+4C_{*}(\gamma-1+\tau^{2})\Big(C_{*}(\gamma-1+\tau^{2})F(I)-\frac{4-\mathscr{K}_{b}}{4}\Big)|\beta||\nu|.
\end{split}
\end{eqnarray*}

If we choose $(\gamma-1+\tau^{2})F(I)<\min\big\{\frac{1-\mathscr{K}_{b}\delta}{4\delta C_{*}}, \frac{4-\mathscr{K}_{b}}{4C_{*}}\big\}$,
then
$F(J)<F(I)$.

\smallskip
\par For case \rm (2) as listed in Lemma \ref{lem:3.5}, we have 
\begin{eqnarray*}
\begin{split}
L(J)-L(I)=0,\ \ \ Q(J)-Q(I)=Q(J',I_{0})+Q(I_{0})-Q(I',I_{0})-Q(I_{0})=0.
\end{split}
\end{eqnarray*}
Therefore
$F(J)=F(I)$.

\par Next, let us consider case \rm(3) as listed in Lemma \ref{lem:3.5}. By Lemma \ref{lem:3.5}, we have
\begin{eqnarray*}
\begin{split}
L(J)-L(I)=0,
\end{split}
\end{eqnarray*}
and
\begin{eqnarray*}
\begin{split}
 Q(J)-Q(I)&=Q(J',I_{0})+Q(I_{0})-Q(I',I_{0})-Q(I_{0})-Q(I')\\[5pt]
&\leq \sum_{\mu\in S(I)}|\mu|\big(|\beta'|-|\beta_{1}|-|\beta_{2}|\big)-|\beta_1||\beta_2|\\[5pt]
&=-|\beta_1||\beta_2|<0.
\end{split}
\end{eqnarray*}
So $F(J)<F(I)$.

\par Now, for case \rm(4) as listed in Lemma \ref{lem:3.5}, with the notations introduced in Lemma \ref{lem:3.1}, we introduce a new mesh curve $\tilde{J}$ between the mesh curves $I$ and $J$ such that we have the local wave interaction $\beta+0\rightarrow\beta_{0}+\nu_{0}$ from $I$ to $\tilde{J}$, and
the local wave interaction $\beta_{0}+\nu_{0}\rightarrow \nu'+\beta'$ from $\tilde{J}$ to $J$.
Then by Lemma \ref{lem:3.5}, we have
\begin{eqnarray*}
\begin{split}
F(J)<F(\tilde{J}),
\end{split}
\end{eqnarray*}
provided that $(\gamma-1+\tau^{2})F(\tilde{J})<\frac{1}{2C_{*}}$.
Next, we also have that
\begin{eqnarray*}
\begin{split}
L(\tilde{J})-L(I)&\leq (\mathscr{K}_{b}-1-C_{0})|\nu_{0}|,\\[5pt]
Q(\tilde{J})-Q(I)&\leq \sum_{\mu\in S(I)}|\mu|\big(|\beta_{0}|+|\nu_{0}|-|\beta|\big)+|\beta_{0}||\nu_{0}|\\[5pt]
&\leq -C_{0}|\nu_{0}|F(I)+|\beta_{0}||\nu_{0}|.
\end{split}
\end{eqnarray*}
So
\begin{eqnarray*}
\begin{split}
F(\tilde{J})-F(I)&\leq (\mathscr{K}_{b}-1-C_{0})|\nu_{0}|+4C_{*}(\gamma-1+\tau^{2})\Big(-C_{0}|\nu_{0}|F(I)+|\beta_{0}||\nu_{0}|\Big)\\[5pt]
&\leq |\nu_{0}|\Big(4C_{*}(\gamma-1+\tau^{2})|\beta_{0}|+(\mathscr{K}_{b}-C_{0})-4C_{*}(\gamma-1+\tau^{2})C_{0}F(I)\Big)\\[5pt]
&\leq 4|\nu_{0}| \Big(C_{*}(\gamma-1+\tau^{2})F(I)-\frac{C_{0}+1-\mathscr{K}_{b}}{4}\Big).
\end{split}
\end{eqnarray*}
Then, if we choose $(\gamma-1+\tau^{2})F(I)<\frac{C_{0}+1-\mathscr{K}_{b}}{4C_{*}}$, then $F(\tilde{J})<F(I)$. Therefore,
$$
F(J)<F(\tilde{J})<F(I).
$$

\par For case \rm(5) as listed in Lemma \ref{lem:3.5}, we have that
\begin{eqnarray*}
\begin{split}
L(J)-L(I)&\leq \mathscr{K}_{b}|\nu'|+|\beta'|-|\beta|
\leq(\mathscr{K}_{b}-1-C_{0})|\nu'|<0,
\end{split}
\end{eqnarray*}
and
\begin{eqnarray*}
\begin{split}
Q(J)-Q(I)&\leq \sum_{\mu\in S(I)}|\mu|\big(|\nu'|+|\beta'|-|\beta|\big)
\leq -C_{0}F(I)|\nu'|.
\end{split}
\end{eqnarray*}
It follows that
$F(J)<F(I)$.

\par For case \rm(6) as listed in Lemma \ref{lem:3.5}, obviously, we have
$F(J)=F(I)$.

\par For case $(7)$ as listed in Lemma \ref{lem:3.5}, we have
\begin{eqnarray*}
\begin{split}
L(J)-L(I) &\leq \mathscr{K}_{b}(|\nu'|-|\nu|)+|\beta'|
\leq(1-\mathscr{K}_{b}-C_{0}\mathscr{K}_{b})|\beta'|<0,
\end{split}
\end{eqnarray*}
and
\begin{eqnarray*}
\begin{split}
Q(J)-Q(I)&\leq \sum_{\mu\in S(I)}|\mu|\big(|\nu'|+|\beta'|-|\nu|\big)
\leq -C_{0}F(I)|\beta'|.
\end{split}
\end{eqnarray*}
It implies that
\begin{eqnarray*}
\begin{split}
F(J)-F(I)&\leq \Big(1-\mathscr{K}_{b}-C_{0}\mathscr{K}_{b}-4C_{*}C_{0}(\gamma-1+\tau^{2})F(I)\Big)|\beta'|<0.
\end{split}
\end{eqnarray*}
Therefore,
$F(J)<F(I)$.

Finally, for case \rm{(8)} as listed in Lemma \ref{lem:3.5}, it can be treated similarly as the argument above for case $(3)$ at listed in Lemma \ref{lem:3.5} to obtain \eqref{eq:4.5}.

\begin{figure}[ht]
\begin{center}
\begin{tikzpicture}[scale=0.7]
\draw [line width=0.06cm] (-3.5,-4) --(-3.5,1.5);
\draw [line width=0.06cm] (-0.5,-4) --(-0.5,1.0);
\draw [line width=0.06cm] (2.5,-4) --(2.5,0.5);
\draw [thin] (-3.5,1.5)--(2.5,0.5);
\draw [thin] (-3.5,-0.5)--(2.5,-1.5);

\draw [thick] (-2.8,1.4)--(-0.5,-2.5);
\draw [thick] (-0.5,-2.5)--(1.5,0.7);
\draw [thick] (-0.5,-2.5)--(1.5,-3.5);

\draw [thin] (-3, 1.4) --(-2.6, 1.8);
\draw [thin] (-2.6, 1.35) --(-2.2, 1.75);
\draw [thin] (-2.2, 1.30) --(-1.8, 1.70);
\draw [thin] (-1.8, 1.23) --(-1.4, 1.63);
\draw [thin] (-1.4, 1.16) --(-1.0, 1.56);
\draw [thin] (-1.0, 1.10) --(-0.6, 1.50);
\draw [thin] (-0.6, 1.03) --(-0.2, 1.43);
\draw [thin] (-0.2, 0.97) --(0.2, 1.37);
\draw [thin] (0.2, 0.9) --(0.6, 1.30);
\draw [thin] (0.6, 0.83) --(1, 1.23);
\draw [thin] (1, 0.76) --(1.4, 1.16);
\draw [thin] (1.4, 0.67) --(1.8, 1.07);
\draw [thin] (1.8, 0.60) --(2.2, 1.0);
\draw [thin] (2.2, 0.55) --(2.6, 0.95);

\draw [thick][red] (-3.5,1.5)--(-1.5,0.5);
\draw [thick][blue](-3.5,-0.5)--(-1.1,-0.1);
\draw [thick][red](-0.5,1)--(1.7,-0.2);

\node at (-1.9, -1.0){$I'$};
\node at (0.6, -1.5){$J'$};
\node at (0.6, -3.5){$I_{0}$};
\node at (-1.9, -3.0){$\Omega_{\Delta, k-1}$};
\node at (1.5, -2.0){$\Omega_{\Delta, k}$};
\node at (-1.0, -0.4){$\nu(o)$};
\node at (-1.1, 0.7){$\beta(\pi)$};
\node at (1.7, -0.5){$\beta'(\pi')$};
\node at (-3.5, -4.5){$x_{k-1}$};
\node at (-0.5, -4.5){$x_{k}$};
\node at (2.5, -4.5){$x_{k+1}$};
\end{tikzpicture}
\end{center}
\caption{$\Lambda$ covers part of the approximate boundary $\Gamma_{\Delta}$}\label{fig24}
\end{figure}
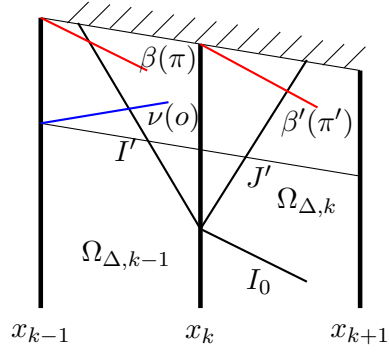

$\mathbf{Case}$ $\mathbf{2}$. $\Lambda$ covers part of the approximate boundary $\Gamma_{\Delta}$ (see Fig. \ref{fig24}).
For case \rm (1) as listed in Lemma \ref{lem:3.7}, we have
$L_{2}(J)-L_{2}(I)\leq K_{b}|\nu|+C_{6}(\gamma-1+\tau^{2})|\nu||\beta|$ and $L_{1}(J)-L_{1}(I)\leq-|\nu|$.
So
\begin{eqnarray*}
\begin{split}
L(J)-L(I)\leq -(\mathscr{K}_{b}-K_{b})|\nu|+C_{6}(\gamma-1+\tau^{2})|\nu||\beta|.
\end{split}
\end{eqnarray*}

For $Q(J)$, we have that
\begin{eqnarray*}
\begin{split}
Q(J)-Q(I)&=Q(J',I_{0})+Q(I_{0})-Q(I',I_{0})-Q(I_{0})-Q(I')\\[5pt]
&\leq \sum_{\mu\in S(I)}|\mu|\big(|\beta'|-|\beta|-|\nu|\big)-|\beta||\nu|\\[5pt]
&\leq \big(K_{b}-1\big)F(I)|\nu|+\Big(C_{6}(\gamma-1+\tau^{2})F(I)-1\Big)|\beta||\nu|.
\end{split}
\end{eqnarray*}
Then, it follows from the estimates of $L(J)$ and $Q(J)$ that
\begin{eqnarray*}
\begin{split}
F(J)-F(I)&\leq\Big(4C_{*}(K_{b}-1)(\gamma-1+\tau^{2})F(I)-(\mathscr{K}_{b}-K_{b})\Big)|\nu|\\[5pt]
&\ \ \  \ +C_{6}(\gamma-1+\tau^{2})|\nu||\beta|
+4C_{*}(\gamma-1+\tau^{2})\Big(C_{6}(\gamma-1+\tau^{2})F(I)-1\Big)|\beta||\nu|\\[5pt]
&\leq \Big(2C_{*}K_{b}(\gamma-1+\tau^{2})F(I)-(\mathscr{K}_{b}-K_{b})\Big)|\nu|\\[5pt]
&\qquad \ \  \ +4C_{*}(\gamma-1+\tau^{2})\Big(C_{*}(\gamma-1+\tau^{2})F(I)-\frac{3}{4}\Big)|\beta||\nu|.
\end{split}
\end{eqnarray*}

Therefore, if we choose $(\gamma-1+\tau^{2})F(I)<\min\{\frac{3}{4C_{*}}, \frac{\mathscr{K}_{b}-K_{b}}{2K_{b}C_{*}}\}$,
then $F(J)<F(I)$.

Next, let's consider case \rm (2) as listed in Lemma \ref{lem:3.7}. Note that
\begin{eqnarray*}
\begin{split}
L(J)-L(I)\leq C_{6}(\gamma-1+\tau^{2})|\beta||o|-C_{b1}|o|
\end{split}
\end{eqnarray*}
and
\begin{eqnarray*}
\begin{split}
Q(J)-Q(I)&=Q(J',I_{0})+Q(I_{0})-Q(I',I_{0})-Q(I_{0})-Q(I')\\[5pt]
&\leq \sum_{\mu\in S(I)}|\mu|\big(|\beta'|-|\beta|\big)\\[5pt]
&\leq \Big(C_{6}(\gamma-1+\tau^{2})|\beta||o|-C_{b1}|o|\Big)F(I)\\[5pt]
&\leq \Big(C_{6}(\gamma-1+\tau^{2})F(I)-C_{b1}\Big)F(I)|o|.
\end{split}
\end{eqnarray*}

So if $C_{*}(\gamma-1+\tau^{2})F(I)\leq C_{b1}$, then
\begin{eqnarray*}
\begin{split}
F(J)-F(I)&\leq C_{6}(\gamma-1+\tau^{2})|\beta||o|-C_{b1}|o|\\[5pt]
&\ \ \ \  +4C_{*}(\gamma-1+\tau^{2})\Big(C_{6}(\gamma-1+\tau^{2})F(I)-C_{b1}\Big)F(I)|o|\\[5pt]
&\leq \Big(C_{*}(\gamma-1+\tau^{2})F(I)-C_{b1}\Big)|o|\\[5pt]
&\ \ \ \ +4C_{*}(\gamma-1+\tau^{2})\Big(C_{*}(\gamma-1+\tau^{2})F(I)-C_{b1}\Big)F(I)|o|\\[5pt]
&\leq 0.
\end{split}
\end{eqnarray*}

Finally, let us consider case \rm (3) as listed in Lemma \ref{lem:3.7}. By direct computations,
\begin{eqnarray*}
\begin{split}
L(J)-L(I)\leq -\big(\mathscr{K}_{b}-K_{b1}\big)|\nu|+C_{6}(\gamma-1+\tau^{2})|\nu|^{2}
\end{split}
\end{eqnarray*}
and
\begin{eqnarray*}
\begin{split}
Q(J)-Q(I)&=Q(J',I_{0})+Q(I_{0})-Q(I',I_{0})-Q(I_{0})-Q(I')\\[5pt]
&\leq \sum_{\mu\in S(I)}|\mu|\big(|\beta'|-|\nu|\big)\\[5pt]
&\leq \Big((K_{b1}-1)|\nu|+C_{6}(\gamma-1+\tau^{2})|\nu|^{2}\Big)F(I).
\end{split}
\end{eqnarray*}

So
\begin{eqnarray*}
\begin{split}
F(J)-F(I)&\leq -\big(\mathscr{K}_{b}-K_{b1}\big)|\nu|+C_{6}(\gamma-1+\tau^{2})|\nu|^{2}\\[5pt]
&\ \ \ \ +4C_{*}(\gamma-1+\tau^{2})F(I)\Big((K_{b1}-1)|\nu|+C_{6}(\gamma-1+\tau^{2})|\nu|^{2}\Big)\\[5pt]
&\leq \Big(-\big(\mathscr{K}_{b}-K_{b1}\big)+C_{6}(\gamma-1+\tau^{2})F(I)+4C_{*}(K_{b1}-1)(\gamma-1+\tau^{2})F(I)\\[5pt]
&\ \ \ \  +4C_{*}C_{6}(\gamma-1+\tau^{2})^{2}F^{2}(I)\Big)|\nu|\\[5pt]
&\leq \Big(-\big(\mathscr{K}_{b}-K_{b1}\big)+4K_{b1}C_{*}(\gamma-1+\tau^{2})F(I)+\big(2C_{*}(\gamma-1+\tau^{2})F(I)\big)^{2}\Big)|\nu|.
\end{split}
\end{eqnarray*}

So, if we choose $(\gamma-1+\tau^{2})F(I)\leq \min\{\frac{1}{C_{*}}, \frac{\mathscr{K}_{b}-K_{b1}}{4K_{b1}C_{*}}\}$, then
$F(J)-F(I)\leq 0$.

Based on all the arguments above, let
\begin{eqnarray}\label{eq:4.6}
\begin{split}
C_{7}&=\min\Bigg\{ \frac{1}{2}, \ \min\Big\{\frac{3}{4},\frac{\mathscr{K}_{b}-\delta}{4\delta} \Big\}, \ \min\Big\{\frac{1-\mathscr{K}_{b}\delta}{4\delta},\frac{4-\mathscr{K}_{b}}{4} \Big\},\ \frac{1+C_{0}-\mathscr{K}_{b}}{4},\\[5pt]
&\qquad \qquad \ \min\Big\{\frac{3}{4}, \frac{\mathscr{K}_{b}-K_{b}}{2K_{b}} \Big\},\ C_{b1}, \
\min\Big\{1,\frac{\mathscr{K}_{b}-K_{b1}}{4K_{b}} \Big\} \Bigg\}.
\end{split}
\end{eqnarray}

So if $(\gamma-1+\tau^{2})F(I)<\frac{C_{7}}{C_{*}}$, we can get estimate \eqref{eq:4.5}.
\end{proof}

\par Let $O$ stand for the initial mesh curve, \emph{i.e.}, for any mesh curve $J$, we have $O\leq J$.
Then, by Lemma \ref{lem:4.1}, we know that if $C_{*}(\gamma-1+\tau^{2})F(O)<C_{7}$, then
$$
F(J)<F(O).
$$

\par Next choose $\gamma_{0}\in (1,2)$ and $\epsilon_{*}>0$ such that $C_{*}(\gamma_0-1+\epsilon_*^{2})L(O)<1$ and $C_{*}(\gamma_0-1+\epsilon_*^{2})F(O)<C_{7}$. Then
for any $\gamma\in[1, \gamma_{0}]$ and $\tau\in (0,\epsilon_{*})$, we have
\begin{eqnarray*}
\begin{split}
F(J)<F(O)=L(O)+4C_{*}(\gamma-1+\tau^{2})Q(O)\leq L(O)+4C_{*}(\gamma-1+\tau^{2})L^{2}(O)<5L(O).
\end{split}
\end{eqnarray*}

\par Notice that $L(O)\leq C\big(T.V.\{U_{0}(\cdot); (-\infty, 0]\}+\|b_{0}\|_{L^{\infty}}\big)$ for some constant $C$ depending only $\mathscr{K}_{b}$
and $C_{*}$. So by the standard argument, (see \cite{glimm, smoller}), we have the following proposition.
\begin{proposition}\label{prop:4.1}
Suppose that $\rho_{0}\in [\rho_{*},\rho^{*} ]$ for some constant states $\rho_{*}$ and $\rho^{*}$ with
$0<\rho_{*}<\rho^{*}<\infty$. Then there exist constants $C_{8}>0$, $\gamma_{0}\in (1, 2)$ and $\epsilon_{*}>0$ such that
for any $\gamma\in [1,\gamma_{0}]$, $\tau\in (0,\epsilon_{*})$ and $\theta\in\prod_{k=0}^{\infty}\theta_k$ if
\begin{eqnarray}\label{eq:4.4}
(\gamma-1+\tau^{2})\Big(T.V.\big\{U_{0}(\cdot); (-\infty, 0] \big\}+\|b_{0}\|_{L^{\infty}}\Big)\leq C_{8},
\end{eqnarray}
then, a sequence of global approximate solutions $U_{\Delta,\theta}(x,y)$ for all $(x,y)\in\Omega_{\Delta}$
is constructed via the Glimm scheme as given in \S \ref{sec:3.1}.
Moreover, there exist positive constants $C_{9}>0$ and $C_{10}>0$ which is independent of $\Delta$ and $\theta$ such that 
\begin{eqnarray}\label{eq:4.5x}
\sup_{x>0}T.V.\big\{U_{\Delta,\theta}(x,\cdot); (-\infty, b_{0}x]\big\}
+\sup_{x>0}\|U_{\Delta,\theta}(x,\cdot)\|_{L^{\infty}((-\infty, b_{0}x])}\leq C_{9},
\end{eqnarray}
and
\begin{eqnarray}\label{eq:4.6x}
\int^{0}_{-\infty}\big|U_{\Delta,\theta}(x_{1},y+b_{0}x_{1})-U_{\Delta,\theta}(x_{2},y+b_{0}x_{2})\big|dy\leq C_{10}\Big(\Delta x+|x_{1}-x_{2}|\Big),
\end{eqnarray}
for any $x_{1}, x_{2}>0$.
\end{proposition}

\par Proposition \ref{prop:4.1} implies the compactness of the approximate solutions $\{U_{\Delta,\theta}(x,y)\}$ in $L^{1}_{loc}$
(see Theorem 2.4 of Chapter 2 in \cite{bressan}). Then, by the standard arguments as done in \cite{glimm, smoller, czz, zh1, zh2},
we can obtain the global existence of the entropy solutions of initial boundary value problem \eqref{eq:1.16}--\eqref{eq:1.18}.
\begin{theorem}\label{thm:4.1}
Assume that the range of the initial density $\rho_{0}$ lies in the interval $[\rho_{*},\rho^{*}]$ for some constants $\rho_{*}$ and $\rho^{*}$ with $0<\rho_{*}<\rho^{*}<\infty$. There exist constants $C_{11}>0$ , $C_{12}>0$, $C_{13}>0$ independent of $\gamma$, $\tau$, and $\gamma_{0}\in (1, 2)$,$\epsilon_{*}>0$ and a null set $\mathcal{N}$  such that for any $\gamma\in [1,\gamma_{0}]$, $\tau\in (0,\epsilon_{*})$ and $\theta\in\Big(\prod_{k=0}^{\infty}\theta_k\backslash \mathcal{N}\Big)$ if
\begin{eqnarray}\label{eq:4.7}
(\gamma-1+\tau^{2})\Big(T.V.\big\{U_{0}(\cdot); (-\infty, 0] \big\}+\|b_{0}\|_{L^{\infty}}\Big)\leq C_{11},
\end{eqnarray}
then, there exist a subsequence $\{\Delta_{i}\}^{\infty}_{i=0}$ and a function $U_{\theta}(x,y)$ with bounded total variation
such that $U_{\Delta_{i},\theta}\rightarrow U_{\theta}(x,y)$
in $L^{1}_{loc}((-\infty, b_{0}x])$ as $\Delta_{i}\rightarrow 0$ for every $x>0$.
The function $U_{\theta}(x,y)$ is a global entropy solution of the initial boundary value problem \eqref{eq:1.16}--\eqref{eq:1.18}
with the properties that
\begin{eqnarray}\label{eq:4.12}
\sup_{x>0}T.V.\big\{U_{\theta}(x,\cdot); (-\infty, b_{0}x]\big\}
+\sup_{x>0}\|U_{\theta}(x,\cdot)\|_{L^{\infty}((-\infty, b_{0}x])}\leq C_{12},
\end{eqnarray}
and
\begin{eqnarray}\label{eq:4.13}
\int^{0}_{-\infty}\big|U_{\theta}(x_{1},y+b_{0}x_{1})-U_{\theta}(x_{2},y+b_{0}x_{2})\big|dy\leq C_{13}|x_{1}-x_{2}|, \ \ \ \forall x_{1}, x_{2}>0.
\end{eqnarray}
\end{theorem}

\begin{remark}\label{rem:4.1}
As the notations introduced in the last sentence in the introduction, \emph{i.e.}, in Section 1, solution $U_{\theta}(x,y)$ to the initial boundary value problem \eqref{eq:1.16}--\eqref{eq:1.18} which are obtained in Theorem \ref{thm:4.1} actually depends on $\tau$. So in order to pass the limit $\tau\rightarrow0$ to prove Theorem \ref{thm:1.1}, we will use the notations $U^{(\tau)}_{\theta}(\bar{x},\bar{y})$ and $(\bar{x},\bar{y})$ again as done in the introduction except the last sentence. 
\end{remark}

Now we are ready to prove Theorem \ref{thm:1.1}.
\begin{proof}[Proof of Theorem \ref{thm:1.1}]
First, the global existence of the entropy solutions $U^{(\tau)}_{\theta}$ to the initial boundary value problem \eqref{eq:1.16}--\eqref{eq:1.18} follows from the Theorem \ref{thm:4.1}. Since solution $U^{(\tau)}_{\theta}$
satisfies estimates \eqref{eq:4.12} and \eqref{eq:4.13} which is independent of $\tau$, we can further apply the Helly's compactness theorem to obtain a subsequence $\{\tau_{i}\}^{\infty}_{i=1}$ such that
$U^{(\tau_{i})}_{\theta}$ converges to $U_{\theta}^{(0)}$ \emph{a.e.} in $\Omega$ as $\tau_{i}\rightarrow 0$. 
Hence, $U^{(\tau_{i})}_{\theta}\rightarrow U_{\theta}^{(0)}$ in $L^{1}(\Omega\cap B_{\bar{R}}(O))$ as $\tau_{i}\rightarrow 0$ for any $\bar{R}>0$, where
$B_{\bar{R}}(O)=\big\{(\bar{x}, \bar{y}): \bar{x}^{2}+\bar{y}^{2}\leq \bar{R}^{2}\big\}$.
Then, by the definition of entropy \eqref{eq:1.19}, we can show that
$U_{\theta}^{(0)}$ is an entropy solution to the initial-boundary value problem \eqref{eq:1.14},
\eqref{eq:1.17} and \eqref{eq:1.10} with $(\mathcal{E}(W^{(0)},0), \mathcal{Q}(W^{(0)},0))$, defined by \eqref{eq:1.24}, being its convex entropy pair with entropy inequality \eqref{eq:1.25} in the distribution sense. This completes the proof of Theorem \ref{thm:1.1}.
\end{proof}

\bigskip

$\textbf{Ackowledgments}:$
The research of Jie Kuang was supported in part by the NSFC Project 11801549 and the Initial Scientific Research Fund
for Wuhan Institute of Physics and Mathematics, Chinese Academy of Sciences. 
The research of Wei Xiang was supported in part by the Research Grants Council of the HKSAR, China (Project No. CityU 21305215, Project No. CityU 11332916, Project No. CityU 11304817, and Project No. CityU 11303518).
The research of Yongqian Zhang was supported in part by the NSFC Project 11421061, NSFC Project 11031001, NSFC Project 11121101,
the 111 Project B08018(China) and the Shanghai Natural Science Foundation 15ZR1403900.


\begin{thebibliography}{10}

\bibitem{asakura} F. Asakura,
{Wave-front tracking for the equations of isentropic gas dynamics}, Q. Appl. Math., 63(2005), 20-33.

\bibitem{ac} F. Asakura, A. Corli,
{Wave-front tracking for the equations of non-isentropic gas dynamics}, Annali di Math., 194(2015), 581-618.

\bibitem{anderson} J. Anderson,
{Hypersonic and High-Temperature Gas Dynamics}, Second Edition, AIAA Education Series, Reston, 2006.


\bibitem{bressan} A. Bressan,
{Hyperbolic Systems of Conservation Laws. The One-Dimensional Cauchy Problem}, Oxford University Press, Oxford, 2000.



\bibitem{ccz2} G.-Q. Chen, C. Christoforou, Y. Zhang,
{Continuous dependence of entropy solutions to the Euler equations on the adiabatic exponent and Mach number},
Arch. Ration. Mech. Anal., 189(2008), 97-130.


\bibitem{cky2}G.-Q. Chen, V. Kukreja, H. Yuan,
{Well-posedness of transonic characteristic discontinuities in two-dimensional steady compressible Euler flows},
Z. Angew. Math. Phys., 64 (2013),1711-1727.

\bibitem{ckz} G.-Q. Chen, J. Kuang, Y. Zhang.
{Two-dimensional steady supersonic exothermically reacting Euler flow past Lipschitz bending walls},
SIAM J. Math. Anal., 49(2017), 818-873.

\bibitem{cl} G.-Q. Chen, T.-H. Li,
{Well-posedness for two-dimnsional steady supersonic Euler flows past a Lipschitz wedge},
 J. Differential Equations, 244(2008), 1521-1550.

\bibitem{cxz} G.-Q. Chen, W. Xiang, Y. Zhang,
{Weakly nonlinear geometric optics for hyperbolic systems of conservation laws},
Comm. Partial Differential Equations, 38(2015), 1936-1970.

\bibitem{czz} G.-Q. Chen, Y. Zhang, D.-W. Zhu,
{Existence and stability of supersonic Euler flows past Lipschitz wedges},
 Arch. Rational Mech. Anal. 181(2006), 261-310.


\bibitem{dyke} M. Van Dyke,
{A Study of Hypersonic Small Disturbance Theory}, NACA Rept., 1194, April, 1954.

\bibitem{glimm} J. Glimm,
{Solutions in the large for nonlinear hyperbolic systems of equations}, Comm. Pure Appl. Math., 18(1965), 697-715.

\bibitem{kyz} V. Kukreja, H. Yuan, Q. Zhao,
{Stability of transonic jet with strong shock in two-dimensional steady compressible Euler flows},
 J. Differential Equations, 258(2015), 2572-2617.

\bibitem{Landau} L. Landau, E. Lifschitz,
{Fluid Mechanics}, 2nd Edition,
Elsevier Ltd., Singapore, 2004.

\bibitem{liu} T.-P. Liu,
{Solutions in the large for the equations of nonisentropic}, Indiana Univ. Math., 26(1977), 147-177.

\bibitem{nishida} T. Nishida,
{Global solution for an initial-boundary value problem of a quasilinear hyperbolic system}, Proc. Jap. Acad., 44(1968), 642-646.

\bibitem{ns1} T. Nishida, J. Smoller,
{Solutions in the large for some nonlinear hyperbolic conservation laws}, Comm. Pure Appl. Math., 26(1973), 183-200.

\bibitem{ns2} T. Nishida, J. Smoller,
{Mixed problems for nonlinear conservation laws}, J.Differential Equations  23(1977), 244-269.

\bibitem{QYZ} A. Qu, H. Yuan, Q. Zhao,
{Hypersonic limit of two-dimensional steady compressible Euler flows passing a straight wedge}, preprinted at arXiv:1904.03360, 2019.

\bibitem{tsien} H.-S. Tsien,
{Similarity laws of hypersonic flows}, Jour. Math. Phys., 25(1946), 247-251.

\bibitem{smoller} J. Smoller,
{Shock Waves and Reaction-Diffusion Equations}, Second Edition, Springer-Verlag, Inc.: New York, 1994.

\bibitem{temple} J. Temple,
{Solutions in the large for the nonlinear hyperbolic conservation laws of gas dynamics}, J. Differential  Equations, 41(1981), 96-161.

\bibitem{xzz} W. Xiang, Y. Zhang, Q. Zhao,
{Two-dimensional steady supersonic exothermically reacting Euler flows with strong contact discontinuity over Lipschitz wall}, Interface Free Bound., 20(2018), 437-481.

\bibitem{zh1} Y. Zhang,
{Global existence of steady supersonic potential flow past a curved wedge with piecewise smooth boundary},
SIAM J. Math. Anal., 31(1999), 166-183.

\bibitem{zh2} Y. Zhang,
{Steady supersonic flow past an almost straight wedge with large vertex angle},
J. Differential Equations, 192(2003), 1-46.

\end{thebibliography}
\end{document}